\documentclass[12pt,a4paper, reqno]{amsart}
\usepackage[utf8]{inputenc}
\usepackage[english]{babel}
\usepackage{amsmath,mathrsfs,amsfonts,amssymb,amsxtra,latexsym,amscd,amsthm,marvosym,multirow,mathtools}
\usepackage[a4paper,inner=2.3cm, outer=2.3cm, top=2.3cm, bottom=2.3cm]{geometry}
\usepackage{setspace}
\usepackage[pdftex]{graphicx}
\usepackage{framed, fancybox}
\usepackage{epstopdf}
\usepackage{subcaption} 
\usepackage{eucal}
\usepackage{enumerate}
\usepackage{csquotes}
\usepackage[shortlabels]{enumitem}
\usepackage{cases}
\usepackage{cleveref}
\usepackage{tikz}
\usepackage{pgfplots}
\usepgfplotslibrary{fillbetween}
\newtheorem{theorem}{Theorem}[section]
\newtheorem{lemma}[theorem]{Lemma}

\newtheorem{remark}[theorem]{Remark}
\newtheorem{proposition}[theorem]{Proposition}
\newtheorem{corollary}[theorem]{Corollary}

\renewcommand{\Re}{\mathrm{Re}\,}
\renewcommand{\Im}{\mathrm{Im}\,}

\DeclareMathOperator*{\Ran}{Ran \;}
\DeclareMathOperator{\sgn}{sgn}

\def\beq{\begin{eqnarray*}}\def\eeq{\end{eqnarray*}}
\def\bq{\begin{equation}}\def\eq{\end{equation}}

\newcommand{\N}{\mathbb{N}}

\newcommand{\R}{\mathbb{R}}
\newcommand{\C}{\mathbb{C}}

\newcommand{\dd}{\mathrm{d}}

\makeatletter

\@addtoreset{equation}{section}  
\makeatother

\begin{document}
\title{Schr\"{o}dinger operator with a complex steplike potential}

\author{Tho Nguyen Duc}
\address[Tho Nguyen Duc]{Department of Mathematics, Faculty of Nuclear Sciences and Physical Engineering, Czech Technical University in Prague, ul. Trojanova 13/339, 12000 Prague, Czech Republic.}
\email{nguyed16@fjfi.cvut.cz}
\begin{abstract}
The purpose of this article is to study pseudospectral properties of the one-dimensional Schr\"{o}dinger operator perturbed by a complex steplike potential. By constructing the resolvent kernel, we show that the pseudospectrum of this operator is trivial if and only if the imaginary part of the potential is constant. As a by-product, a new method to obtain a sharp resolvent estimate is developed, answering a concern of Henry and Krej\v{c}i\v{r}\'{i}k, and a way to construct an optimal pseudomode is discovered, answering a concern of  Krej\v{c}i\v{r}\'{i}k and Siegl. The spectrum and the norm of the resolvent of the complex point interaction of the operator is also studied carefully in this article.
\end{abstract}
\maketitle
\section{Introduction}
\subsection{Context and Motivations}
Since the birth of quantum physics at the turn of the 20th century, spectral theory of self-adjoint operator has attracted considerable attention and  experienced many fertile developments. But its best friend \emph{non-self-adjoint} operator has only recently been noticed and investigated restricted to the last two decades. 
One of the challenges that we often face when we work with the non-self-adjoint operator is that there is no spectral theorem \cite{Krejcirik-Siegl15}. A clear evidence for this is the absence of the well-known formula for the norm of the resolvent
\begin{equation}\label{Norm equal 1/d}
\Vert (\mathscr{L}-z)^{-1} \Vert = \frac{1}{\textup{dist}(z,\sigma(\mathscr{L}))},
\end{equation}
which is valid for self-adjoint $\mathscr{L}$ (or generally, for unbounded normal operator, the reader can found a simple proof in subsection \ref{Subsec Non-trivial}). The lack of this formula is supported by the fact that there exist many non-self-adjoint operators whose resolvents blow up when the spectral parameter travels far away from their spectrum. Therefore, the notion of \emph{pseudospectrum} was suggested to address this pathological aspect of the non-self-adjoint operator \cite{Trefethen-Embree05,Davies07}. More precisely, given $\varepsilon>0$, the $\varepsilon$-\emph{pseudospectrum} of a linear operator $\mathscr{L}$ is defined by
\begin{equation}\label{Def Pseuspectrum 1}
\sigma_{\varepsilon}(\mathscr{L}) = \sigma(\mathscr{L})\cup\{z\in \C: \Vert (\mathscr{L}-z)^{-1} \Vert> \varepsilon^{-1} \}.
\end{equation}
The usefulness of pseudospectrum is that it can give the answer to the question \enquote{How does the spectrum respond to a slight change of the initial operator?} by virtue of the formula
\[ \sigma_{\varepsilon}(\mathscr{L}) = \bigcup_{V \in L(H),\, \Vert V \Vert<\varepsilon} \sigma(\mathscr{L}+V).\] 
Especially, when the pseudospectrum contains regions very far from the spectrum, it causes the instability of the spectrum under a small perturbation and reveals that it would be difficult to obtain the spectrum numerically. The reader may want to see a discussion about this topic in the introduction of \cite{Pravda-Starov08}.

According to the definition \eqref{Def Pseuspectrum 1}, the more information we have on the level set of the resolvent, the better described the pseudospectrum is. However, in almost all cases, it is not easy to calculate the resolvent of a given operator, and even when the resolvent is known, it is not easy to calculate its norm. Therefore, in practice, we often use an equivalent definition of the pseudospectrum, that is 
\begin{equation}\label{Def Pseuspectrum 2}
\sigma_{\varepsilon}(\mathscr{L}) = \sigma(\mathscr{L}) \cup \left\{z\in \C:\exists \Psi \in \textup{Dom}(\mathscr{L}),\, \Vert (\mathscr{L}-z) \Psi \Vert < \varepsilon \Vert \Psi \Vert \right\}.
\end{equation}
The number $z$ and the vector $\Psi$ in the definition \eqref{Def Pseuspectrum 2} are respectively called the \emph{pseudoeigenvalue} and \emph{pseudomode} (also known as the \emph{pseudoeigenfunction}, \emph{pseudoeigenvector}, or \emph{quasimode}) of $\mathscr{L}$. We list here some references \cite{Henry-Krejcirik17,Antonio-Petr22,Arnal-Siegl23} using the definition \eqref{Def Pseuspectrum 1} and some references \cite{Davies99,  Boulton02, Denker-Sjostrand-Zworski04, Krejcirik-Siegl-Tater-Viola15, Henry-Krejcirik17, Krejcirik-Siegl19,Krejcirik-Nguyen22, Nguyen22} using the definition \eqref{Def Pseuspectrum 2} to investigate pseudospectrum of a differential operator.

In this article, we would like to use both two aforementioned definitions to study the pseudospectrum of the free Schr\"{o}dinger operator adding with  a complex steplike potential (see Figure \ref{Fig: Potential}),
\[ \mathscr{L}= -\frac{\dd^2}{\dd x^2}+V(x),\qquad V(x)\coloneqq \left\{\begin{aligned}
&V_{+} &&\text{ for }x\geq 0,\\
&V_{-} &&\text{ for }x< 0,
\end{aligned} \right.\qquad \text{with } V_{+},V_{-}\in \C,\]
and its complex point interaction
\begin{equation}\label{Complex delta operator}
\mathscr{L}_{\alpha}= -\frac{\dd^2}{\dd x^2}+V(x) + \alpha \delta_{0},\qquad \alpha\in \C,
\end{equation}
where $\delta_{0}$ is the Dirac delta generalized function.
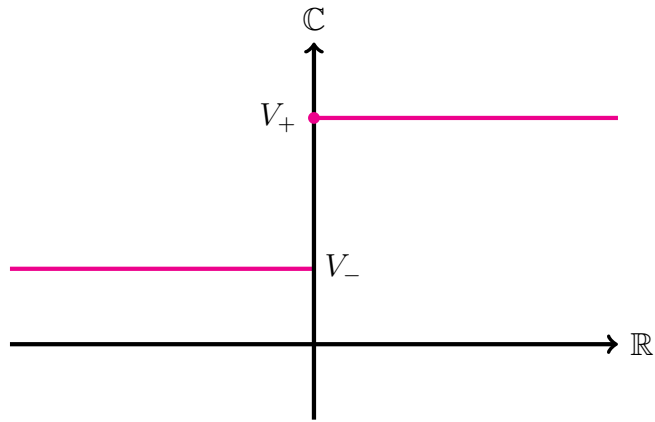
\begin{figure}[h!]
\centering
\begin{tikzpicture}
\draw[-,ultra thick, ultra thick, magenta] (0,3)--(4,3);
\draw[-,ultra thick, ultra thick, magenta] (0,1)--(-4,1);
\draw[->,ultra thick] (-4,0)--(4,0) node[right]{$\R$};
\draw[->,ultra thick] (0,-1)--(0,4) node[above]{$\C$};

\draw (0,3) node[fill,magenta,circle,scale=0.4, label=left:{$V_{+}$}] {};
\draw (0,1) node[fill,magenta,circle,scale =0, label=right:{$V_{-}$}]{};
\end{tikzpicture}
\caption{The complex steplike potential $V(x)$.}\label{Fig: Potential}
\end{figure}

There are three main results in our paper. The first result concerns with the spectra of two operators $\mathscr{L}$ and $\mathscr{L}_{\alpha}$. Theorem \ref{Theo Spectrum of L} and Theorem \ref{Theo Spectrum L delta} will provide explicit answers to the elementary questions:
\begin{enumerate}
\item Depending on $V$, what does the spectrum of $\mathscr{L}$ look like?
\item Depending on $\alpha\in \C$, how does the spectrum of $\mathscr{L}$ change under the affect of complex point interaction?
\end{enumerate}
For the first question, it can be predictable that the spectrum of $\mathscr{L}$ is obtained by shifting the ray $[0,+\infty)$ (the spectrum of the free Schr\"{o}dinger operator) by two vectors $V_{+}$ and $V_{-}$. For the second question, we will see that some eigenvalue will be released when $\alpha$ belongs to a geometric region $\Omega\subset \C$ depending on the difference $V_{+}-V_{-}$ (see Figure \ref{Fig:Region Omega}).

Our second result is related to the resolvents of $\mathscr{L}$ and $\mathscr{L}_{\alpha}$. As above, Theorem \ref{Theo Norm of resolvent} and Theorem \ref{Theo Norm of resolvent delta} will addresses two questions
\begin{enumerate}
\item What is the asymptotic behavior of the norm of the resolvent of $\mathscr{L}$ inside the numerical range?
\item How does the complex point interaction affect this behavior of the resolvent of $\mathscr{L}$?
\end{enumerate}
To answer the first question, we find a new method to obtain an explicit formula for the divergence of the resolvent
 \begin{equation*}
\left\Vert (\mathscr{L}-z)^{-1}\right\Vert = \frac{2|\Im V_{+}-\Im V_{-}|}{|V_{+}-V_{-}|} \frac{\Re z}{|\Im V_{+}-\Im z||\Im V_{-}-\Im z|} \left(1+\mathcal{O}\left( \frac{1}{|\Re z|}\right)\right),
\end{equation*}
as $\Re z\to +\infty$ and uniformly for all $\Im z$ between $\Im V_{+}$ and $\Im V_{-}$. Our finding addresses a concern of Henry and Krej\v{c}i\v{r}\'{i}k \cite{Henry-Krejcirik17}, who wondered whether we could find an optimal constant and a sharp dependence on the distance between the spectral parameter and the spectrum in their specific case $V(x)=i\, \textup{sgn}(x)$. What is more, by applying this method for the operator $\mathscr{L}_{\alpha}$, with $\alpha\in \C\setminus \{0\}$, we can answer the second question, that is
\begin{equation*}
\left\Vert (\mathscr{L}_{\alpha}-z)^{-1}\right\Vert = \frac{|\Im V_{+}-\Im V_{-}|\sqrt{\Re z}}{|\alpha||\Im V_{+}-\Im z||\Im V_{-}-\Im z|}\left(1+\mathcal{O}\left(\frac{1}{|\Re z|^{1/2}}\right)\right).
\end{equation*}
as $\Re z\to +\infty$ and uniformly for all $\Im z$ between $\Im V_{+}$ and $\Im V_{-}$. We see that not only the constant changes, now it depends both on $\Im V_{\pm}$ and $\alpha$, but the divergent rate of the resolvent norm also decreases from $\Re z$ to $\sqrt{\Re z}$.

Our final result is rather interesting and it is related to the pseudomode construction for the operators $\mathscr{L}$ and $\mathscr{L}_{\alpha}$. It is interesting because it gives us a hope that we may describe the pseudospectrum accurately when we do not have the formula for the resolvent of the operator. More precisely, from Theorem \ref{Theo Pseudomode}, we derive an explicit formula for the pseudomode $\Psi_{z}\in \textup{Dom}(\mathscr{L})$ such that
\[ \frac{\Vert(\mathscr{L}-z)\Psi_{z}\Vert}{\Vert \Psi_{z} \Vert} = \frac{|V_{+}-V_{-}|}{2|\Im V_{+}-\Im V_{-}|} \frac{|\Im V_{+}-\Im z||\Im V_{-}-\Im z|}{\Re z} \left(1+\mathcal{O}\left( \frac{1}{|\Re z|}\right)\right) ,\]
as $\Re z\to +\infty$ and uniformly for all $\Im z$ between $\Im V_{+}$ and $\Im V_{-}$. Our finding exceeds the expectation of the concern of Krej\v{c}i\v{r}\'{i}k and Siegl in \cite{Krejcirik-Siegl19}, in which they tried to construct the pseudomode for the Schr\"{o}dinger operator with $V=i \,\text{sgn}(x)$, \emph{i.e.}, $V_{+}=i$ and $V_{-}=-i$, but the best possible for the decaying rate that they could obtain is $\mathcal{O}\left(\frac{1}{z^{1/2}}\right)$ as $z$ goes to $+\infty$ on the real line. Here, our method provides an optimal pseudomode which give us the exact constant $\frac{|V_{+}-V_{-}|}{2|\Im V_{+}-\Im V_{-}|} $, the precise rate $\Re z$ and the correct distance to the spectrum $|\Im V_{+}-\Im z||\Im V_{-}-\Im z|$. In order to verify that our construction method is relevant, we apply it for the model $\mathscr{L}_{\alpha}$ ($\alpha \neq 0$) and it is still applicable and produce an optimal pseudomode in Theorem \ref{Theo Pseudomode delta}.

Although the model with steplike potential is simple, it also has its own application in scattering theory \cite{Grunert13,Teschl15} and in dispersive estimate \cite{D'Ancona-Selberg12}. Also because it is simple, it is often  chosen as a pioneering model to understand other models whose potentials behave asymptotically as a steplike function. We hope that our model may be chosen to provide more information about the pseudospectrum of the Schr\"{o}dinger with complex potentials, which is a trending hot topic recently.
\subsection{General notations}
Let us fix some notations employed throughout the paper.
\begin{enumerate}[label=\textup{(\arabic*)}]
\item We use the following conventions for number sets:
\begin{itemize}
\item As usual, $\R$ is for the real numbers, $\C$ is for complex numbers, $\R_{+}\coloneqq (0,+\infty)$ and $\R_{-}\coloneqq (-\infty,0)$.
\item For $a,b \in \R$ and $a\neq b$, we denote $|(a,b)|$ for the open interval whose boundaries are $a$ and $b$, \emph{i.e.}, $|(a,b)|=(a,b)$ if $a<b$ and $|(a,b)|=(b,a)$ if $a>b$. Similarly, we denote $|[a,b]|=[a,b]$ if $a<b$ and $|[a,b]|=[b,a]$ if $a>b$
\item For an axis starting from a complex number $C$ and running to infinity horizontally in $\C$, we denote by $[C,+\infty)$, \emph{i.e.}, $[C,+\infty)\coloneqq C+[0,+\infty)$.
\item When we write $\langle X, Y \rangle_{\R^2}$ for $X,Y\in \C$, we mean that we consider $X,Y$ as vectors in $\R^2$ and $\langle X, Y \rangle_{\R^2}$ is a real inner product between two vectors $X$ and $Y$ on $\R^2$, that is $\langle X, Y \rangle_{\R^2}=X_{1}X_{2}+Y_{1}Y_{2}$ for $X=X_{1}+iX_{2}$ and $Y=Y_{1}+iY_{2}$.
\item For two real-valued functions $a$ and $b$, we will occasionally write $a\lesssim b$ (respectively, $a\gtrsim b$) instead of $a\leq C b$ (respectively, $a \geq C b$) for an insignificant constant $C>0$. 
\end{itemize}
\item For an indicator function (characteristic function) of a subset $E$ in $\R$, we denote by $\textbf{\textup{1}}_{E}$, \emph{i.e.},  $\textbf{\textup{1}}_{E}(x)$ has value $1$ at points in $E$ and $0$ at points in $\R\setminus E$.
\item The inner product on $L^2(\R)$ is denoted by $\langle \cdot, \cdot \rangle$. We use the symbol $\Vert \cdot \Vert$ for $L^2$-norm of complex-valued functions defined on $\R$ and when we want to consider this norm restricted on $\R_{+}$ (or, on $\R_{-}$), we will use a clear symbol $\Vert \cdot \Vert_{L^2(\R_{+})}$ (respectively, $\Vert \cdot \Vert_{L^2(\R_{-})}$). The norm on Sobolev space $H^1(\R)$ is denoted by $\Vert \cdot \Vert_{H^1}$.
\item For a linear operator $S$, as usual, we employ the notations $\textup{Dom}(S)$, $\textup{Ran}(S)$, $\textup{Ker}(S)$, $\rho(S)$ and $\sigma(S)$ for, respectively, the domain, the range, the kernel, the resolvent set and the spectrum of $S$. When $S$ is a densely defined operator, let us recall here some classes of unbounded operators: 
\begin{itemize}
\item $S$ is called normal if $\textup{Dom}(S)=\textup{Dom}(S^{*})$ and $\Vert S \Vert = \Vert S^{*}\Vert$.
\item $S$ is called self-adjoint if $S=S^{*}$.
\item $S$ is called $\mathcal{T}$-self-adjoint, if $S^{*}=\mathcal{T}S\mathcal{T}$, where $\mathcal{T}$ is the antilinear operator of complex conjugation defined by $\mathcal{T} \psi(x)=\overline{\psi(x)}$.
\item $S$ is called $\mathcal{P}$-self-adjoint if $S^{*}=\mathcal{P}S \mathcal{P}$ with a parity operator $\mathcal{P}$ defined by $(\mathcal{P}\psi)(x)=\psi(-x)$.
\item $S$ is called $\mathcal{PT}$-symmetric if $\left[S,\mathcal{PT} \right]=0$, \emph{i.e.}, $\mathcal{PT}S \subset S \mathcal{PT}$, it means that whenever $f\in \textup{Dom}(S)$, $\mathcal{PT}f$ also belongs to $\textup{Dom}(S)$ and $\mathcal{PT}Sf=S\mathcal{PT}f$.
\end{itemize}
The properties of these kinds of operators can be found in \cite[Subsection 5.2.5]{Krejcirik-Siegl15}, in particular, a recent and interesting research on $\mathcal{T}$-self-adjointness under its modern name \emph{Complex-self-adjointness} can be found in \cite{Camara-Krejcirik23} . We often decompose the spectrum of a closed operator as follows
\[ \sigma(S)=\sigma_{\textup{p}}(S) \cup \sigma_{\textup{r}}(S) \cup \sigma_{\textup{c}}(S),\]
in which
\begin{align*}
&\sigma_{\textup{p}}(S)\coloneqq  \{z\in \C: S-z \text{ is not injective} \},\\
&\sigma_{\textup{r}}(S)\coloneqq  \{z\in \C: S-z \text{ is injective and } \overline{\textup{Ran}(S-z)} \subsetneq \mathcal{H} \},\\
&\sigma_{\textup{c}}(S)\coloneqq  \{z\in \C: S-z \text{ is injective and } \overline{\textup{Ran}(S-z)}=\mathcal{H} \text{ and } \textup{Ran}(S-z)\subsetneq \mathcal{H} \}.
\end{align*}
The set $\sigma_{\textup{p}}(S)$ (respectively, $\sigma_{\textup{r}}(S)$ and $\sigma_{\textup{c}}(S)$) is called the \emph{point spectrum} (respectively, the \emph{residual spectrum} and the \emph{continuous spectrum}) of $S$. For the \emph{essential spectrum}, we use the definitions of various types of the essential spectra defined in \cite[Sec. IX]{Edmunds-Evans18} or \cite[Sec. 5.4]{Krejcirik-Siegl15}: $\sigma_{\textup{ek}}(S)$ for $\textup{k}\in \{1,2,3,4,5\}$. The discrete spectrum of $S$, which is the set of isolated eigenvalues $z$ of $S$ which have finite algebraic multiplicity and such that $\Ran(S-z)$ is closed in $\mathcal{H}$, is labeled by $\sigma_{\textup{dis}}(S)$.
\item By abusing of notation, we shall denote integral operators and their kernels by the same symbol. For example, we will write the integral operator $\mathcal{R}$ as $\mathcal{R}f(x) = \int_{\R} \mathcal{R}(x,y)f(y)\, \dd y$. 
\end{enumerate}
\subsection{Structure of the paper}
The organization of this paper is as follows. Section \ref{Section Main Results} is devoted to all the statements and main results: the definition of the operator $\mathscr{L}$, its resolvent and spectrum, its pseudospectrum and optimal pseudomode, and finally the same achievements for $\mathscr{L}_{\alpha}$. As usual, the remain sections are used to provide proofs or to describe the methods that we have employed, more precisely,
\begin{itemize}
\item In Section \ref{Section Resolvent Spectrum}, the kernel of the resolvent of $\mathscr{L}$ is established and the spectrum of $\mathscr{L}$ is characterized.
\item In Section \ref{Section Resolvent Estimate}, the pseudospectrum of $\mathscr{L}$ is studied by estimating the resolvent norm inside the numerical range.
\item Section \ref{Section Complex point interaction} is used to study the spectral properties of the operator $\mathscr{L}_{\alpha}$. The stability of the essential spectra under the Dirac interaction is proved. Then, the existence of the discrete eigenvalues depending on $\alpha$ is discussed. After the spectrum is clear, the resolvent norm behavior of $\mathscr{L}_{\alpha}$ is determined as the spectral parameter goes to infinity in the region between two essential spectrum lines. Finally, the optimal pseudomode is also constructed for this delta interaction model.
\item Two appendixes \ref{Appendix 1} and \ref{Appendix L a} are employed to define and show all the beautiful properties of the operators $\mathscr{L}$ and $\mathscr{L}_{\alpha}$.
\end{itemize}
\subsection*{Acknowledgement}
I would like to thank Professor David Krej\v{c}i\v{r}\'{i}k for giving me many precious opportunities to continue my research career. To me, he is not only a great mathematician, but he is also a great leader who takes care his people very kindly. I am kindly thankful to Professor Petr Siegl for his comments when I visited him in Graz. This project was supported by the EXPRO grant number 20-17749X of the Czech Science Foundation (GA\v{C}R).
\section{Statements and main Results}\label{Section Main Results}
\subsection{The operator, the resolvent and the spectrum}\label{Subsec Resolvent L}
Let us begin by defining our operator via a sesquilinear form whose formula and domain are given by
\begin{align*}
Q(u,v)&\coloneqq\int_{\R} u'(x) \overline{v'(x)}\, \dd x + V_{+}\int_{0}^{+\infty} u(x) \overline{v(x)}\, \dd x +V_{-}\int_{-\infty}^{0} u(x) \overline{v(x)}\, \dd x,\\
\textup{Dom}(Q)&\coloneqq  H^1(\R).
\end{align*}
Then, the description and some useful properties of the operator are represented in our first proposition. The proof of this proposition is rather elementary, however, we would like to write it down for the convenience of the readers, especially for young researchers like the author. The proof of this proposition can be found in Appendix \ref{Appendix 1}.
\begin{proposition}\label{Prop Property}
There exists a closed densely defined operator $\mathscr{L}$ whose domain is given by
\begin{equation*}
\textup{Dom}\left(\mathscr{L}\right)=\left\{\begin{aligned}
u \in H^1(\R): &\text{ there exists } f\in L^2(\R) \text{ such that }\\
&Q(u,v)=\left\langle f,v \right\rangle \text{ for all } v\in H^1(\R)
\end{aligned} \right\},
\end{equation*}
and
\begin{equation}\label{Laxmilgram eq}
Q(u,v)=\left\langle \mathscr{L}u, v\right\rangle, \qquad \forall u\in \textup{Dom}(\mathscr{L}), \qquad \forall v\in H^1(\R).
\end{equation}
Then, the following holds.
\begin{enumerate}[label=\textbf{\textup{(\arabic*)}}]
\item \label{Nonempty res} The domain and the action $\mathscr{L}$ can be clarified that 
\begin{align*}
\textup{Dom}\left(\mathscr{L}\right) &= H^2(\R),\\
\mathscr{L}u &= -u'' + V(x)u, \qquad \forall u \in H^2(\R),
\end{align*}
and its resolvent set $\rho(\mathscr{L})$ is nonempty.
\item The numerical range of $\mathscr{L}$ is given by \textup{(}see Figure \textup{\ref{Numerical Range}}\textup{)}
\begin{equation}\label{Num Range}
 \textup{Num}(\mathscr{L})=\left\{\left(0,+\infty\right)+sV_{+}+(1-s)V_{-}:\,s\in [0,1]\right\},
\end{equation}
and as a consequence, $\mathscr{L}$ is a $m-$sectorial operator.
\item The adjoint of $\mathscr{L}$ is given by
\begin{equation}\label{Adjoint}
\mathscr{L}^{*}=-\frac{\dd^2}{\dd x^2}+\overline{V(x)}\,,\qquad \textup{Dom}(\mathscr{L}^{*})=H^2(\R),
\end{equation}
and as a consequence,
\begin{enumerate}[label=\textbf{\textup{(\alph*)}}]
\item $\mathscr{L}$ is normal if and only if $\Im V_{+} = \Im V_{-}$;
\item $\mathscr{L}$ is self-adjoint if and only if $\Im V_{+} = \Im V_{-}=0$;
\item $\mathscr{L}$ is always $\mathcal{T}$-self-adjoint;
\item $\mathscr{L}$ is $\mathcal{P}$-self-adjoint if and only if $\Re V_{+}=\Re V_{-}$ and $\Im V_{+}=-\Im V_{-}$;
\item $\mathscr{L}$ is $\mathcal{PT}$-symmetric if and only if $\Re V_{+}=\Re V_{-}$ and $\Im V_{+}=-\Im V_{-}$.
\end{enumerate}
\end{enumerate}
\end{proposition}
\begin{figure}[h!]
\centering
\begin{tikzpicture}
\fill[ fill=magenta!10](-2,1) -- (2,5)--(10,5)--(10,1);

\draw[-,ultra thick, ultra thick, magenta] (-2,1)--(10,1);
\draw[-,ultra thick, ultra thick, magenta] (2,5)--(10,5);
\draw[->,ultra thick] (-2,0)--(10,0) node[right]{$\Re z$};
\draw[->,ultra thick] (0,-1)--(0,7) node[above]{$\Im z$};

\draw (2,5) node[label=above:{$V_{+}$}] {};
\draw (-2,1) node[label=below:{$V_{-}$}]{};
\node[black] at (5,3) {$\textup{Num}(\mathscr{L})$};
\end{tikzpicture}
\caption{Illustration of the numerical range of $\mathscr{L}$ in the magenta color.}\label{Numerical Range}
\end{figure}
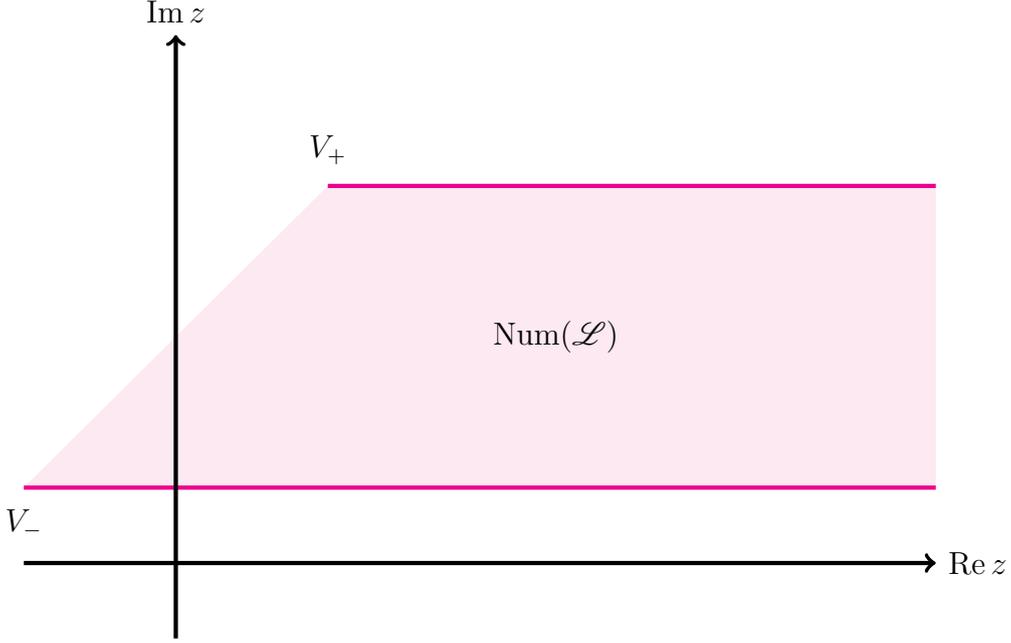
Our next proposition will describe explicitly the resolvent of the operator $\mathscr{L}$. It shows that the resolvent can be written in the integral form.
\begin{proposition}\label{Prop Resolvent}
Let $V_{\pm}\in \C$ and $\mathscr{L}$ be the operator defined as in Proposition \ref{Prop Property}. For all $z \in \C\setminus \left( [V_{+},+\infty)\cup [V_{-},+\infty)\right)$ and for every $f\in L^{2}(\R)$, we have
\begin{equation}\label{Integral Op}
\left[\left(\mathscr{L}-z\right)^{-1} f\right](x) = \int_{\R} \mathcal{R}_{z}(x,y) f(y)\, \dd y ,
\end{equation}
where $\mathcal{R}_{z}(x,y)$ is defined by
\begin{equation*}
\mathcal{R}_{z}(x,y)= \left\{
\begin{aligned}
			 & \frac{1}{2k_{+}(z)}e^{-k_{+}(z)|x-y|}+\frac{k_{+}(z)-k_{-}(z)}{2k_{+}(z)\left(k_{+}(z)+k_{-}(z) \right)}e^{-k_{+}(z)(x+y)} && \text{for } \{ x> 0, y> 0\}; \\
			 &  \frac{1}{2k_{-}(z)}e^{-k_{-}(z)|x-y|}-\frac{k_{+}(z)-k_{-}(z)}{2k_{-}(z)\left(k_{+}(z)+k_{-}(z) \right)}e^{k_{-}(z)(x+y)} && \text{for } \{x < 0, y< 0\};\\ 		 
			 & \frac{1}{k_{+}(z)+k_{-}(z)}e^{-k_{+}(z)x+k_{-}(z)y} && \text{for } \{x> 0, y< 0\}; \\
			 & \frac{1}{k_{+}(z)+k_{-}(z)}e^{k_{-}(z)x-k_{+}(z)y}  && \text{for } \{x< 0,y> 0\}; 
		\end{aligned}
\right.
\end{equation*}
where we set
\begin{equation}\label{kpkn}
k_{+}(z)\coloneqq \sqrt{V_{+}-z},\qquad k_{-}(z)\coloneqq \sqrt{V_{-}-z}.
\end{equation}
\end{proposition}
Here and throughout the article, we choose the principle value of the square root, \emph{i.e.}, $z\mapsto\sqrt{z}$ defined on $\C$ which is holomorphic on $\C\setminus (-\infty,0]$ and positive on $(0,+\infty)$.
\begin{remark}
When $V_{+}=V_{-}=v\in \C$, we obtain a simple formula for the resolvent kernel of the Schr\"{o}dinger operator $\mathscr{L}$:
\begin{equation}
\mathcal{R}_{z}(x,y)= \frac{1}{2\sqrt{v-z}} e^{-\sqrt{v-z}|x-y|},\qquad\text{ for almost everywhere } (x,y)\in \R^2.
\end{equation}
Obviously, the resolvent of the Schr\"{o}dinger operator with steplike potential has more terms than the resolvent of the free Schr\"{o}dinger operator, \emph{i.e.}, $V_{+}=V_{-}=0$. Besides the the exponential terms with the difference $x-y$, it also contains the exponential terms with $x$ and $y$ which can be separable. We will see that the latter will play the main role of the blowing-up of the resolvent inside the numerical range which makes the pseudospectrum highly non-trivial (see subsection \ref{Subsec Resolvent inside N}).
\end{remark}
By using Weyl sequence, we can show that the set $[V_{+},+\infty)\cup [V_{-},+\infty)$ is indeed the spectrum of $\mathscr{L}$ whose characterization is described explicitly in the following theorem.
\begin{theorem}\label{Theo Spectrum of L}
Let $V_{\pm}\in \C$ and $\mathscr{L}$ be the operator defined as in Proposition \ref{Prop Property}. The spectrum of $\mathscr{L}$ is given by \textup{(}see Figure \textup{\ref{Fig Spectrum L}}\textup{)}
\begin{equation}\label{Spectrum of L}
\sigma \left(\mathscr{L}\right)= [V_{+},+\infty)\cup [V_{-},+\infty).
\end{equation}
Furthermore, the spectrum of $\mathscr{L}$ is purely continuous, \emph{i.e.},
\[ \sigma_{\textup{p}}(\mathscr{L})=\emptyset,\qquad \sigma_{\textup{r}}(\mathscr{L})=\emptyset,\qquad \sigma_{\textup{c}}(\mathscr{L}_{m})=[V_{+},+\infty)\cup [V_{-},+\infty), \]
and all its essential spectra are identical:
\[ \sigma_{\textup{e1}}(\mathscr{L})=\sigma_{\textup{e2}}(\mathscr{L})=\sigma_{\textup{e3}}(\mathscr{L})=\sigma_{\textup{e4}}(\mathscr{L})=\sigma_{\textup{e5}}(\mathscr{L})=[V_{+},+\infty)\cup [V_{-},+\infty).\]
\end{theorem}
In view of Theorem \ref{Theo Spectrum of L}, it can be seen that all the essential spectra are identical even when $\mathscr{L}$ is not necessarily self-adjoint. When $V_{+}=V_{-}=0$, the well-known result on the spectrum of the free Schr\"{o}dinger operator, which is classically attained from the positiveness of $-\frac{\dd^2}{\dd x^2}$ and the existence of the Weyl sequence (\emph{approximate eigenfunctions}), is recovered
\[ \sigma(\mathscr{L})=[0,+\infty).\]
When $\Im V_{+} = \Im V_{-}$, the spectrum of $\mathscr{L}$ is restricted to the axis
\begin{equation}\label{V hat}
 \sigma(\mathscr{L})=[\widehat{V},+\infty),\qquad \widehat{V}\coloneqq\min(\Re V_{+},\Re V_{-})+i \Im V_{+},
\end{equation}
which is as same as the spectrum of the free Schr\"{o}dinger operator translated by a constant potential $\widehat{V}$. When $V(x)=i\, \textup{sgn}(x)$, we attain the result given in \cite[Proposition 2.1]{Henry-Krejcirik17}. 
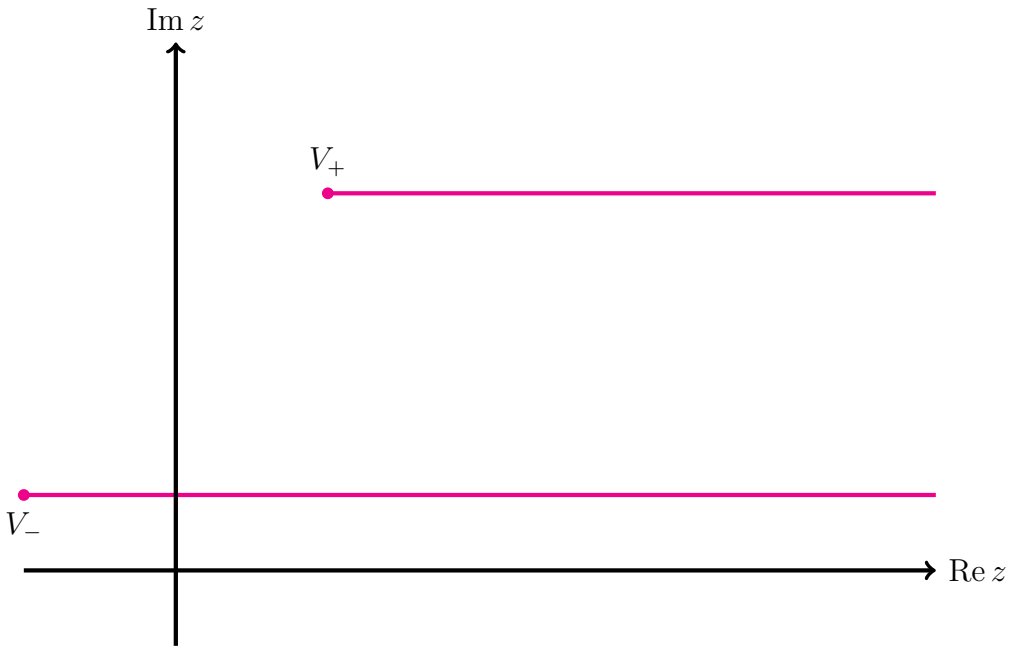
\begin{figure}[h!]
\centering
\begin{tikzpicture}

\draw[-,ultra thick, ultra thick, magenta] (-2,1)--(10,1);
\draw[-,ultra thick, ultra thick, magenta] (2,5)--(10,5);
\draw[->,ultra thick] (-2,0)--(10,0) node[right]{$\Re z$};
\draw[->,ultra thick] (0,-1)--(0,7) node[above]{$\Im z$};

\draw (2,5) node[fill,magenta,circle,scale=0.4, label=above:{$V_{+}$}] {};
\draw (-2,1) node[fill,magenta,circle,scale =0.4, label=below:{$V_{-}$}]{};

\end{tikzpicture}
\caption{The spectrum of $\mathscr{L}$ is expressed by the magenta lines whose starting points are $V_{+}$ and $V_{-}$.}\label{Fig Spectrum L}
\end{figure}
\subsection{The pseudospectrum}
Since the free Schr\"{o}dinger operator is self-adjoint, it is deduced from \eqref{Def Pseuspectrum 1} and \eqref{Norm equal 1/d} that
\[ \sigma_{\varepsilon}\left(-\frac{\dd^2}{\dd x^2}\right)=\left\{z\in \C: \textup{dist}\left(z,\sigma\left(-\frac{\dd^2}{\dd x^2}\right)\right)<\varepsilon \right\}.\]
Therefore, its pseudospectrum is trivial. Here, we call a pseudospectrum of a linear operator $S$ is trivial if there exists $C>0$ such that, for all $\varepsilon>0$, we have
\[ \sigma_{\varepsilon}(S) \subset \{z\in \C: \textup{dist}(z,\sigma(S))< C\varepsilon \}.\]
By adding with a steplike potential, which is a bounded perturbation, it is natural to wonder if the pseudospectrum is trivial or not. At the end of this subsection, we will show that this pseudospectrum is trivial if and only if $\Im V_{+}=\Im V_{-}$, \emph{i.e.}, $\Im V(x)$ is a constant. The following proposition is about the resolvent norm outside the numerical range which is a direct consequence of the operator theory.
\begin{proposition}\label{Prop Resolvent out}
Let $V_{\pm}\in \C$ and $\mathscr{L}$ be the operator defined as in Proposition \ref{Prop Property}. For all $z\in \C \setminus \overline{\textup{Num}(\mathscr{L})}$, we have
\begin{equation}\label{Resolvent out}
\frac{1}{\textup{dist}(z,\sigma(\mathscr{L}))} \leq \Vert (\mathscr{L}-z)^{-1} \Vert \leq \frac{1}{\textup{dis}(z,\textup{Num}(\mathscr{L}))}.
\end{equation}
As a consequence, we have
\begin{equation}\label{Resolvent out 2}
\Vert (\mathscr{L}-z)^{-1} \Vert=\frac{1}{\textup{dist}(z,\sigma(\mathscr{L}))},
\end{equation}
for all $z\in W(\mathscr{L})\coloneqq \left\{z\in \C \setminus \overline{\textup{Num}(\mathscr{L})}: \textup{dist}(z,\sigma(\mathscr{L}))= \textup{dist}(z,\textup{Num}(\mathscr{L}))\right\}$.
\end{proposition}
\begin{proof}
The first inequality in \eqref{Resolvent out} comes from \cite[Theorem 1.2.10]{Davies07} and the second one comes from \cite[Lemma 9.3.14]{Davies07}.
\end{proof}
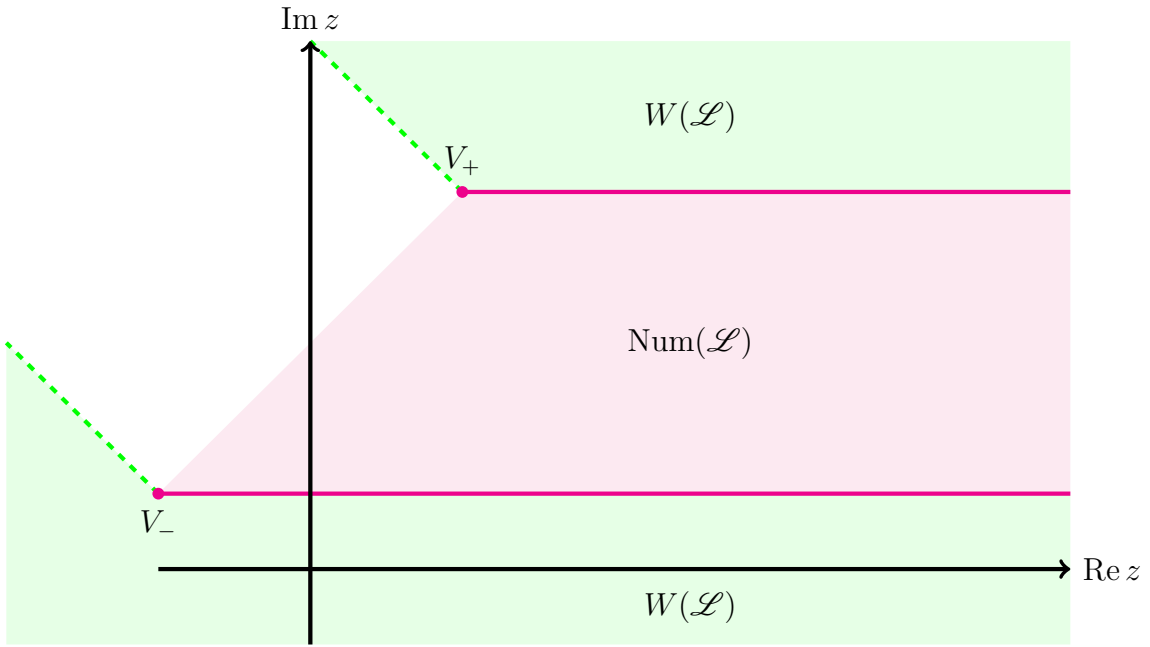
\begin{figure}[h!]
\centering
\begin{tikzpicture}
\fill[ fill=magenta!10](-2,1) -- (2,5)--(10,5)--(10,1);
\fill[ fill=green!10](2,5)--(10,5)--(10,7)--(0,7);
\fill[ fill=green!10](-2,1)--(10,1)--(10,-1)--(-4,-1)--(-4,3);

\draw[-,ultra thick, ultra thick, magenta] (2,5)--(10,5);
\draw[-,ultra thick, ultra thick, magenta] (-2,1)--(10,1);

\draw[-,dashed, ultra thick, green!100] (2,5)--(0,7);
\draw[-,dashed, ultra thick, green!100] (-2,1)--(-4,3);
\draw[->,ultra thick] (-2,0)--(10,0) node[right]{$\Re z$};
\draw[->,ultra thick] (0,-1)--(0,7) node[above]{$\Im z$};

\draw (2,5) node[fill,magenta,circle,scale=0.4, label=above:{$V_{+}$}] {};
\draw (-2,1) node[fill,magenta,circle,scale =0.4, label=below:{$V_{-}$}]{};
\node[black] at (5,6) {$W(\mathscr{L})$};
\node[black] at (5,-0.5) {$W(\mathscr{L})$};
\node[black] at (5,3) {$\textup{Num}(\mathscr{L})$};
\end{tikzpicture}
\caption{The region $W(\mathscr{L})$ is described in the green color, in which two dashed lines are perpendicular to line segment $V_{+}V_{-}$ at $V_{+}$ and $V_{-}$.}\label{Fig: W(L)}
\end{figure}
\begin{remark}
In Figure \ref{Fig: W(L)}, we give a description of the set $W(\mathscr{L})$ in Proposition \ref{Prop Resolvent out} in which the norm of the resolvent behaves like in the case of the normal operator. However, $W(\mathscr{L})$ is not the largest region in $\C$ such that the equality \eqref{Resolvent out 2} happens. To see this, for simplicity, let us consider a simple case $V_{+}=i$ and $V_{-}=-i$ \textup{(}see Figure \textup{\ref{Fig: Out W(L)}}\textup{)} and thanks to the integration by parts, we have, for all $u\in H^2(\R)$,
\begin{align*}
\Vert (\mathscr{L}-z)u \Vert^2=& |i-z|^2 \Vert u\Vert_{L^2(\R_{+})}^2+|-i-z|^2 \Vert u\Vert_{L^2(\R_{-})}^2+\Vert u'' \Vert^2-2\Re z \Vert u' \Vert^2 +4 \Im \left(u'(0)\overline{u(0)}\right).
\end{align*}
Notice that $|-i-z|^2-|i-z|^2 = 4 \Im z$ and we write
\begin{align*}
\Vert (\mathscr{L}-z)u \Vert^2=|i-z|^2 \Vert u\Vert^2 + 4 \Im z \Vert u\Vert_{L^2(\R_{-})}^2+\Vert u'' \Vert^2-2\Re z \Vert u' \Vert^2 +4 \Im \left(u'(0)\overline{u(0)}\right) 
\end{align*}
Assume that $\Re z<0$, by using the inequalities, see for instance \eqref{Inequality}, 
\[|u(0)|^2\leq 2 \Vert u \Vert_{L^2(\R_{-})}\Vert u' \Vert_{L^2(\R_{-})} \text{ and } |u'(0)|^2\leq 2 \Vert u' \Vert_{L^2(\R_{+})}\Vert u'' \Vert_{L^2(\R_{+})} \]
and AM-GM inequality for $4$ numbers, we have
\begin{align*}
\left| 4\Im \left(u'(0)\overline{u(0)}\right) \right|\leq &8 \Vert u'' \Vert_{L^2(\R_{+})}^{1/2} \Vert u' \Vert_{L^2(\R_{+})}^{1/2} \Vert u' \Vert_{L^2(\R_{-})}^{1/2}\Vert u \Vert_{L^2(\R_{-})}^{1/2}\\
\leq &\Vert u'' \Vert_{L^2(\R_{+})}^{2}-2 \Re z \Vert u' \Vert_{L^2(\R_{+})} -2 \Re z \Vert u' \Vert_{L^2(\R_{-})} + \frac{4}{(\Re z)^2}\Vert u \Vert_{L^2(\R_{-})}^{2}.
\end{align*}
Then, it yields that, for all $u\in H^2(\R)$,
\begin{align*}
\Vert (\mathscr{L}-z)u \Vert^2 \geq |i-z|^2 \Vert u\Vert^2 + \left( 4 \Im z -\frac{4}{(\Re z)^2} \right)\Vert u\Vert_{L^2(\R_{-})}^2.
\end{align*}
Therefore, if $\Im z \geq \frac{1}{(\Re z)^2}$ and $\Re z<0$, then, for all $u\in H^2(\R)$,
\[ \Vert (\mathscr{L}-z)u \Vert^2 \geq |i-z|^2 \Vert u \Vert^2.\]
In the same manner, it can be shown that if $-\Im z \geq \frac{1}{(\Re z)^2}$ and $\Re z<0$, then, for all $u\in H^2(\R)$,
\[ \Vert (\mathscr{L}-z)u \Vert^2 \geq |-i-z|^2 \Vert u \Vert^2.\]
This implies that if $|\Im z|\geq \frac{1}{(\Re z)^2}$ and $\Re z<0$, then
\[ \Vert  (\mathscr{L}-z)^{-1} \Vert\leq \frac{1}{\min \{ |i-z|,|-i-z|\}}= \frac{1}{\textup{dist}(z,\sigma(\mathscr{L}))}.\]
By using the first inequality in \eqref{Resolvent out}, we obtain the equality
\[ \Vert  (\mathscr{L}-z)^{-1} \Vert= \frac{1}{\textup{dist}(z,\sigma(\mathscr{L}))}\]
on the blue region in Figure \ref{Fig: Out W(L)}.
\end{remark}
\begin{figure}[h!]
\centering
\begin{tikzpicture}
\fill[ fill=magenta!10](0,1) -- (3,1)--(3,-1)--(0,-1);
\fill[ fill=green!10](-5,1)--(3,1)--(3,3)--(-5,3);
\fill[ fill=green!10](-5,-1)--(3,-1)--(3,-3)--(-5,-3);

\draw[-,ultra thick, ultra thick, magenta] (0,1)--(3,1);
\draw[-,ultra thick, ultra thick, magenta] (0,-1)--(3,-1);
\draw[name path= A,-,dashed, ultra thick, green!100] (0,1)--(-5,1);
\draw[name path= B,-,dashed, ultra thick, green!100] (0,-1)--(-5,-1);
\draw[->,ultra thick] (-5,0)--(3,0) node[right]{$\Re z$};
\draw[->,ultra thick] (0,-3)--(0,3) node[above]{$\Im z$};
\draw[name path= 1,ultra thick, domain=-5:-1, smooth, variable=\x, red]  plot ({\x}, {1/\x^2});
\draw[name path =2,ultra thick, domain=-5:-1, smooth, variable=\x, red]  plot ({\x}, {-1/\x^2});
\tikzfillbetween[
    of=A and 2
  ] {cyan!20};
  
\tikzfillbetween[
    of=B and 1
  ] {cyan!20};
\draw (0,1) node[fill,magenta,circle,scale=0.4,label= above right:{$i$}] {};
\draw (0,-1) node[fill,magenta,circle,scale=0.4,label=below right:{$-i$}]{};
\node[black] at (2,2) {$W(\mathscr{L})$};
\node[black] at (2,-2) {$W(\mathscr{L})$};
\node[black] at (2,0.5) {$\textup{Num}(\mathscr{L})$};

\end{tikzpicture}
\caption{The case $V_{+}=i$ and $V_{-}=-i$: Two red lines are the graph of the function $|\Im z|=\frac{1}{(\Re z)^2}$ and $\Vert (\mathscr{L}-z)^{-1} \Vert=\frac{1}{\textup{dist}(z,\sigma(\mathscr{L}))}$ happens on both the green and the blue regions. }\label{Fig: Out W(L)}
\end{figure}
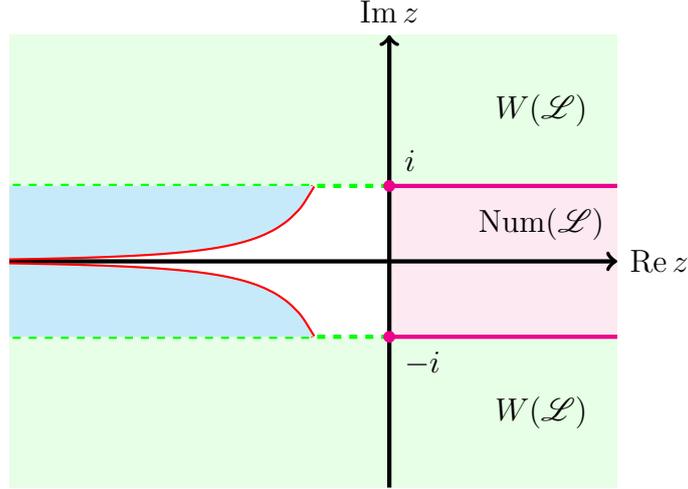
Since the resolvent $(\mathscr{L}-z)^{-1}$ is a holomorphic function on $\rho(\mathscr{L})$, its norm is continuous and thus, bounded on any compact set in $\rho(\mathscr{L})$. Therefore, if we want to find a place where the resolvent norm blows up, we should try to find this place in an unbounded set in $\rho(\mathscr{L})$. In view of Proposition \ref{Prop Resolvent out}, it can be seen that the resolvent norm is a constants when $z$ moves parallel to the numerical range to infinity outside the numerical range and it is decaying when $z$ goes far away from the numerical range. This tells us that the blowing-up place of the resolvent norm should be found inside the numerical range and this is the content of the following theorem.
\begin{theorem}\label{Theo Norm of resolvent}
Let $V_{\pm}\in \C$ such that $\Im V_{+} \neq \Im V_{-}$ and $\mathscr{L}$ be the operator defined as in Proposition \ref{Prop Property}. Let $z\in \C$ such that $\Im z\in |(\Im V_{+},\Im V_{-})|$, then 
\begin{equation*}
\left\Vert (\mathscr{L}-z)^{-1}\right\Vert = \frac{2|\Im V_{+}-\Im V_{-}|}{|V_{+}-V_{-}|} \frac{\Re z}{|\Im V_{+}-\Im z||\Im V_{-}-\Im z|} \left(1+\mathcal{O}\left( \frac{1}{|\Re z|}\right)\right),
\end{equation*}
as $\Re z\to +\infty$ and uniformly for all $\Im z\in |(\Im V_{+},\Im V_{-})|$.
\end{theorem}
Let us note that when $V_{+}=i$ and $V_{-}=-i$, thanks to Theorem \ref{Theo Norm of resolvent}, we obtain the formula
\[ \left\Vert (\mathscr{L}-z)^{-1}\right\Vert =\frac{2\Re z}{1-|\Im z|^2} \left(1+\mathcal{O}\left( \frac{1}{|\Re z|}\right)\right),\]
as $\Re z\to +\infty$ and uniformly for all $|\Im z|<1$, and this statement is clearly an improvement of \cite[Theorem 2.3]{Henry-Krejcirik17}. Furthermore, by letting $\frac{2\Re z}{1-|\Im z|^2}=\frac{1}{\varepsilon}$, we obtain the shape of level curves in Figure \ref{Fig: Level Curve} which are very similar to level set $\{z\in \C : \Vert (\mathscr{L}-z)^{-1} \Vert=\frac{1}{\varepsilon}\}$ inside the numerical range which is computed numerically in \cite[Figure 1]{Henry-Krejcirik17}. 
\begin{figure}[h!]
\includegraphics[width=\textwidth]{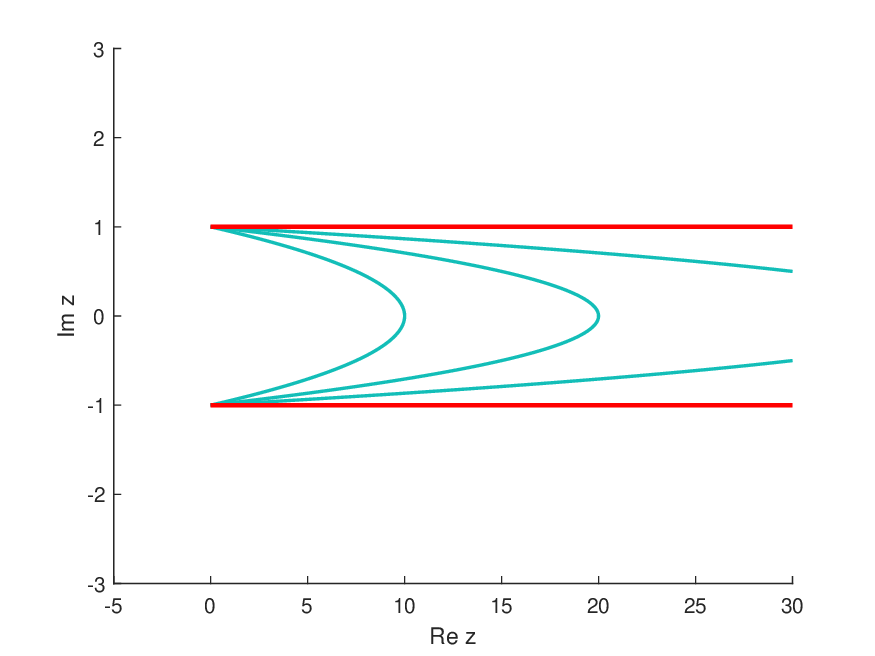}
\caption{The level curves $\frac{2\Re z}{1-|\Im z|^2}=\frac{1}{\varepsilon}$ in the complex $z$-plane are computed for several values of $\varepsilon$; the red lines are the essential spectrum of $\mathscr{L}$ in the case $V_{+}=i$ and $V_{-}=-i$; the smaller $\varepsilon$ is, the closer to infinity and red lines the level curve is.}\label{Fig: Level Curve}
\end{figure}

The direct result of Proposition \ref{Prop Resolvent out} and Theorem \ref{Theo Norm of resolvent} is the following corollary which shows us that the source of the non-trivial pseudospectra comes from the difference between $\Im V_{+}$ and $\Im V_{-}$, regardless of what values $\Re V_{+}$ and $\Re V_{-}$ are. 
\begin{corollary}\label{Cor Trivial}
Let $V_{\pm}\in \C$ and $\mathscr{L}$ be the operator defined as in Proposition \ref{Prop Property}. Then, the following holds.
\begin{enumerate}[label=\textbf{\textup{(\arabic*)}}]
\item If $\Im V_{+}=\Im V_{-}$, the pseudospectrum of $\mathscr{L}$ is given by
\[ \sigma_{\varepsilon}(\mathscr{L})=\{z\in \C: \textup{dist}(z,\sigma(\mathscr{L}))<\varepsilon \}.\]
\item If $\Im V_{+}\neq \Im V_{-}$, for every $\varepsilon'\in (0,1)$, there exists $M>0$ such that, for every $\varepsilon>0$, we have
\begin{align*}
\sigma_{\varepsilon}(\mathscr{L})\supset &\left\{z\in \C: \textup{dist}(z,\sigma(\mathscr{L}))<\varepsilon \right\}\\
 &\bigcup \left\{z\in \C:\begin{aligned}
 &\Re z>M \text{ and } \Im z \in |(\Im V_{+},\Im V_{-})|  \text{ and }\\
 &\Re z> \frac{1}{\varepsilon(1-\varepsilon')} \frac{|V_{+}-V_{-}||\Im V_{+}-z||\Im V_{-}-z|}{2|\Im V_{+}-\Im V_{-}|}
\end{aligned}  \right\}.
\end{align*}
\end{enumerate}
As a consequence, $\sigma_{\varepsilon}(\mathscr{L})$ is trivial if and only if $\Im V_{+}=\Im V_{-}$. 
\end{corollary}
\subsection{The optimal pseudomode}
One of the early studies on the pseudomode construction for the non-self-adjoint Hamiltonian can be attributed to Davies' work \cite{Davies99}. In his worked, Davies scaled variable so that the original Sch\"{o}dinger operator can be transformed to its semiclassical version and thanks to this fact, the WKB pseudomode is constructed. However, this method seems merely effective for a certain class of polynomial potentials in which the scaling can be performed, while it is inapplicable for other type of potentials such as logarithmic or exponential potentials. Twenty years later, Krej\v{c}i\v{r}\'{i}k and Siegl in \cite{Krejcirik-Siegl19} developed a \emph{direct construction of large-energy pseudomode} for the Schr\"{o}dinger operator, which does not require passage through semiclassical setting and can cover all of the above-mentioned potentials. Furthermore, they also created two methods, the first is called \emph{perturbative approach} and the second is called \emph{mollification strategy}, to deal with lower regularity potentials (discontinuous potentials are also included). However, when these methods are applied for the potential $V(x)=i \,\textup{sgn}(x)$, the best possible decaying rate of the quotient $\frac{\Vert (\mathscr{L}-z) \Psi_{z} \Vert}{\Vert \Psi_{z} \Vert}$ that may be attained is $\mathcal{O}\left(\frac{1}{z^{1/4}}\right)$ for the perturbative approach, see \cite[Example 4.5]{Krejcirik-Siegl19}, and $\mathcal{O}\left(\frac{1}{z^{1/2}}\right)$ for the mollification strategy, see \cite[Example 4.11]{Krejcirik-Siegl19}, as $z\to +\infty$ on the real axis. Our next result will present an explicit pseudomode $\Psi_z$ that yields the precise decaying rate which should be $\mathcal{O}\left(\frac{1}{z}\right)$. Better yet, an optimal outcome, which is verified from Theorem \ref{Theo Norm of resolvent}, is also attained.
\begin{theorem}\label{Theo Pseudomode}
Let $V_{\pm}\in \C$ such that $\Im V_{+} \neq \Im V_{-}$ and $\mathscr{L}$ be the operator defined as in Proposition \ref{Prop Property}, and $k_{\pm}(z)$ as in \eqref{kpkn}. For each $z\in \rho(\mathscr{L})$, we define a function $\Psi_{z}$ as follows
\begin{equation}\label{Pseudomode}
\Psi_{z}(x) = \left( n_{1}(z) e^{k_{-}(z)x}+n_{2}(z) e^{\overline{k_{-}(z)}x}\right)\textbf{\textup{1}}_{\R_{-}}(x)+\left(p_{1}(z) e^{-k_{+}(z)x}+p_{2}(z) e^{-\overline{k_{+}(z)}x}\right)\textbf{\textup{1}}_{\R_{+}}(x),
\end{equation}
where $n_{1}(z),p_{1}(z)$ and $n_{2}(z),p_{2}(z)$ are complex numbers depending on $z$ given by
\begin{equation}\label{Constants}
\left\{
\begin{aligned}
&n_{1}(z)=\frac{k_{+}(z)+\overline{k_{-}(z)}}{k_{+}(z)+k_{-}(z)}|\Im V_{+}-\Im z|+\frac{k_{+}(z)-\overline{k_{+}(z)}}{k_{+}(z)+k_{-}(z)}|\Im V_{-}-\Im z|,\\
&p_{1}(z)=-\frac{k_{-}(z)-\overline{k_{-}(z)}}{k_{+}(z)+k_{-}(z)}|\Im V_{+}-\Im z|-\frac{\overline{k_{+}(z)}+k_{-}(z)}{k_{+}(z)+k_{-}(z)}|\Im V_{-}-\Im z|,\\
&n_{2}(z)=-|\Im V_{+}-\Im z|,\\
&p_{2}(z)=|\Im V_{-}-\Im z|.
\end{aligned}
\right.
\end{equation}
Then, $\Psi_{z}\in H^2(\R)$ for all $z\in \rho(\mathscr{L})$ and it makes
\[ \frac{\Vert(\mathscr{L}-z)\Psi_{z}\Vert}{\Vert \Psi_{z} \Vert} = \frac{|V_{+}-V_{-}|}{2|\Im V_{+}-\Im V_{-}|} \frac{|\Im V_{+}-\Im z||\Im V_{-}-\Im z|}{\Re z} \left(1+\mathcal{O}\left( \frac{1}{|\Re z|}\right)\right) ,\]
as $\Re z \to +\infty$ and uniformly for all $\Im z\in |(\Im V_{+},\Im V_{-})|$.
\end{theorem}
\subsection{Complex point interaction}
When the operator $\mathscr{L}$ is rigorously investigated above, in this subsection, we want to discuss its perturbed model by the Dirac delta generalized function $\delta_{0}$, the so-called \emph{point interaction}. For the title of this subsection, we want to indicate that the adding distribution $\delta_{0}$ will be multiplied by a complex number $\alpha$ and then, the formal expression of this perturbed operator, denoted by $\mathscr{L}_{\alpha}$, is given as in \eqref{Complex delta operator}. Our first aim is to define its \emph{realization} in $L^2(\R)$ via the following sesquilinear
\begin{align*}
Q_{\alpha}(u,v)&\coloneqq\int_{\R} u'(x) \overline{v'(x)}\, \dd x + V_{+}\int_{0}^{+\infty} u(x) \overline{v(x)}\, \dd x +V_{-}\int_{-\infty}^{0} u(x) \overline{v(x)}\, \dd x+\alpha u(0)\overline{v(0)},\\
\textup{Dom}(Q_{\alpha})&\coloneqq  H^1(\R).
\end{align*}
The definition and some properties of $\mathscr{L}_{\alpha}$ is given in the following proposition, whose detailed proof can be found in Appendix \ref{Appendix L a}.
\begin{proposition}\label{Prop Property delta}
There exists a closed densely defined operator $\mathscr{L}_{\alpha}$ whose domain is given by
\begin{equation*}
\textup{Dom}\left(\mathscr{L}_{\alpha}\right)=\left\{\begin{aligned}
u \in H^1(\R): &\text{ there exists } f\in L^2(\R) \text{ such that }\\
&Q_{\alpha}(u,v)=\left\langle f,v \right\rangle \text{ for all } v\in H^1(\R)
\end{aligned} \right\},
\end{equation*}
and
\begin{equation}\label{Laxmilgram eq Interaction}
Q_{\alpha}(u,v)=\left\langle \mathscr{L}_{\alpha}u, v\right\rangle, \qquad \forall u\in \textup{Dom}(\mathscr{L}_{\alpha}), \qquad \forall v\in H^1(\R).
\end{equation}
Then, the following holds.
\begin{enumerate}[label=\textbf{\textup{(\arabic*)}}]
\item \label{Nonempty res delta} The domain and the action of  $\mathscr{L}_{\alpha}$ can be clarified that
\begin{align*}
&\textup{Dom}(\mathscr{L}_{\alpha})\coloneqq\{u\in H^1(\R)\cap H^2(\R\setminus\{0\}): u'(0^{+})-u'(0^{-})=\alpha u(0) \},\\
&\mathscr{L}_{\alpha}u=-u''+V(x)u, \qquad \forall u \in \textup{Dom}(\mathscr{L}_{\alpha}),
\end{align*}
and its resolvent set $\rho(\mathscr{L}_{\alpha})$ is nonempty;
\item The numerical range $\textup{Num}(\mathscr{L}_{\alpha})$ is included in the sector $S_{C_{V,\alpha}, \frac{\pi}{4}}$ with a vertex 
\begin{equation}\label{Vertex}
C_{V,\alpha}=\min \{\Re V_{+}, \Re V_{-}\}-\max\{|\Im V_{+}|,|\Im V_{-}| \}-(|\Re \alpha|+|\Im \alpha|)^2,
\end{equation}
and as a consequence, $\mathscr{L}_{\alpha}$ is a $m-$sectorial operator.
\item The adjoint of $\mathscr{L}_{\alpha}$ is given by
\begin{equation}\label{Adjoint delta}
\begin{aligned}
&\textup{Dom}(\mathscr{L}_{\alpha}^{*})=\{u\in H^1(\R)\cap H^2(\R\setminus\{0\}): u'(0^{+})-u'(0^{-})=\overline{\alpha} u(0) \},\\
&\mathscr{L}_{\alpha}^{*}u=-u''+\overline{V(x)}u, \qquad \forall u \in \textup{Dom}(\mathscr{L}_{\alpha}^{*}).
\end{aligned}
\end{equation}
and as a consequence
\begin{enumerate}[label=\textbf{\textup{(\alph*)}}]
\item $\mathscr{L}_{\alpha}$ is normal if and only if 
\begin{itemize}
\item $\Im V_{+} = \Im V_{-}=0$ and $\alpha\in \R$,
\end{itemize}
or
\begin{itemize}
\item $\Im V_{+} = \Im V_{-}\neq 0$ and $\alpha=0$;
\end{itemize}
\item $\mathscr{L}_{\alpha}$ is self-adjoint if and only if $\Im V_{+} = \Im V_{-}=0$ and $\alpha\in \R$;
\item $\mathscr{L}_{\alpha}$ is always $\mathcal{T}$-self-adjoint;
\item $\mathscr{L}_{\alpha}$ is $\mathcal{P}$-self-adjoint if and only if $\Re V_{+}= \Re V_{-}$ and $\Im V_{+}=-\Im V_{-}$ and $\Re \alpha =0 $;
\item $\mathscr{L}_{\alpha}$ is $\mathcal{PT}$-symmetric if and only if  $\Re V_{+}= \Re V_{-}$ and $\Im V_{+}=-\Im V_{-}$ and $\Re \alpha =0 $.
\end{enumerate}
\end{enumerate} 
\end{proposition}
In the light of this proposition, it is clear that the complex number $\alpha$ which appears in the action of quadratic form $Q_{\alpha}$ does not appear in the action of $\mathscr{L}_{\alpha}$ but rather enters the domain of the operator. It is amazing to notice that the conditions $\Im V_{+}=\Im V_{-}$ and $\alpha$ is real are not enough to say that the operator $\mathscr{L}_{\alpha}$ is normal. Furthermore, the study of the explicit numerical range of the operator $\mathscr{L}_{\alpha}$ constitutes an interesting open problem, even with a simple case $V_{+}=V_{-}=0$.

The main purpose of this subsection is to reproduce all the results for the operator $\mathscr{L}_{\alpha}$ that we achieved for $\mathscr{L}$ in the previous subsections. The first outcome that we want to obtain is the spectrum of the operator $\mathscr{L}_{\alpha}$. However, we do not need to find the resolvent set of $\mathscr{L}_{\alpha}$ to implies its spectrum as we did in subsection \ref{Subsec Resolvent L}. Instead, we can study $\sigma(\mathscr{L}_{\alpha})$ directly. To do that, we will show that some type of the essential spectra are stable under the interaction through the following lemma.
\begin{lemma}\label{Lem Essential Spectra delta}
Let $V_{\pm}\in \C$, $\alpha \in \C$ and $\mathscr{L}_{\alpha}$ be the operator defined as in Proposition \ref{Prop Property delta}. Then, the first three essential spectra of $\mathscr{L}$ and $\mathscr{L}_{\alpha}$ are identical:
\[ \sigma_{\textup{ek}}(\mathscr{L}) = \sigma_{\textup{ek}}(\mathscr{L}_{\alpha}) \qquad \text{ for }\textup{k}\in \{1,2,3\}.\]
As a consequence, we have
\[ \sigma_{\textup{ek}}(\mathscr{L}_{\alpha}) = [V_{+}, +\infty) \cup [V_{-},+\infty),\qquad \text{ for } \textup{k}\in \{1,2,3,4,5\}.\]
\end{lemma}
Now, we can describe the spectrum of the operator $\mathscr{L}_{\alpha}$ by the following theorem.
\begin{theorem}\label{Theo Spectrum L delta}
Let $V_{\pm}\in \C$, $\alpha \in \C$ and $\mathscr{L}_{\alpha}$ be the operator defined as in Proposition \ref{Prop Property delta}. Then, the spectrum of $\mathscr{L}_{\alpha}$ is composed of point spectrum and continuous spectrum,
\[ \sigma(\mathscr{L}_{\alpha})= \sigma_{\textup{c}}(\mathscr{L}_{\alpha}) \cup\sigma_{\textup{p}}(\mathscr{L}_{\alpha}),\]
in which, 
\begin{itemize}
\item the continuous spectrum is also the essential spectrum
\begin{equation}\label{Continuous Spec delta}
\sigma_{\textup{c}}(\mathscr{L}_{\alpha})= \sigma_{\textup{ek}}(\mathscr{L}_{\alpha}) = [V_{+}, +\infty) \cup [V_{-},+\infty),\qquad \text{ for } \textup{k}\in \{1,2,3,4,5\},
\end{equation}
\item the point spectrum is also the discrete spectrum
\begin{equation}
\sigma_{\textup{p}}(\mathscr{L}_{\alpha})= \sigma_{\textup{dis}}(\mathscr{L}_{\alpha})= \left\{\begin{aligned}
&\left\{ z(\alpha)\right\} \qquad &&\text{if } \alpha \in \Omega,\\
& \emptyset \qquad &&\text{if } \alpha \in \C \setminus \Omega,
\end{aligned}
 \right.
\end{equation}
where
\begin{align*}
z(\alpha)&\coloneqq \frac{V_{+}+V_{-}}{2}-\frac{(V_{+}-V_{-})^2}{4\alpha^2}-\frac{\alpha^2}{4},\\
\Omega &\coloneqq \left\{\alpha \in \C :\left\vert \langle V_{+}-V_{-}, \alpha \rangle_{\R^2}\right\vert< -|\alpha|^2 \Re \alpha \right\}.
\end{align*}
\end{itemize}
When $\alpha\in \Omega$, the eigenspace of $z(\alpha)$ is given by
\[ \textup{Ker}(\mathscr{L}_{\alpha}-z(\alpha)) = \textup{span} \left\{ u_{\alpha}(x) \right\},\]
where
\begin{equation}\label{Eigenfunction delta 0}
u_{\alpha}(x) \coloneqq  \left\{\begin{aligned}
& e^{\left(\frac{\alpha}{2}+\frac{V_{+}-V_{-}}{2\alpha} \right) x}, \qquad &&\text{ for }x\geq 0,\\
& e^{-\left(\frac{\alpha}{2}-\frac{V_{+}-V_{-}}{2\alpha}\right)x}, \qquad &&\text{ for }x\leq 0.
\end{aligned} \right.
\end{equation}
\end{theorem}
According to Theorem \ref{Theo Spectrum L delta}, there is some $\alpha\in \C$ at which the spectra of $\mathscr{L}$ and $\mathscr{L}_{\alpha}$ are the same, and there is some $\alpha\in \C$ at which they differ. The latter happens when $\alpha$ belongs to the set $\Omega$, which stays on the left of the imaginary axis in complex plane and whose shape depends on the difference $V_{+}-V_{-}$. 
\begin{remark}
In Figure \ref{Fig:Region Omega}, some shapes of $\Omega$ corresponding to $V_{+}-V_{-}$ are represented. In particular, in the case $V_{+}-V_{-}=2i$, we disprove wrong statements in \cite[Proposition 7.1 and Figure 3]{Henry-Krejcirik17} that the domain $\C\setminus \Omega$ is a curve in the complex plane. In any cases, the set $\Omega$ and $\C\setminus \Omega$ must be $2$-dimension areas in $\C$. Obviously, the case $V_{+}-V_{-}=0$ produces the largest region $\Omega$, that is $\Omega=\{z\in \C: \Re z<0\}$, in all circumstances corresponding to the values of $V_{+}-V_{-}$.
\end{remark}
\begin{figure*}[ht!]
     \centering
     \begin{subfigure}[c]{0.45\textwidth}
         \centering
         \includegraphics[width=\textwidth]{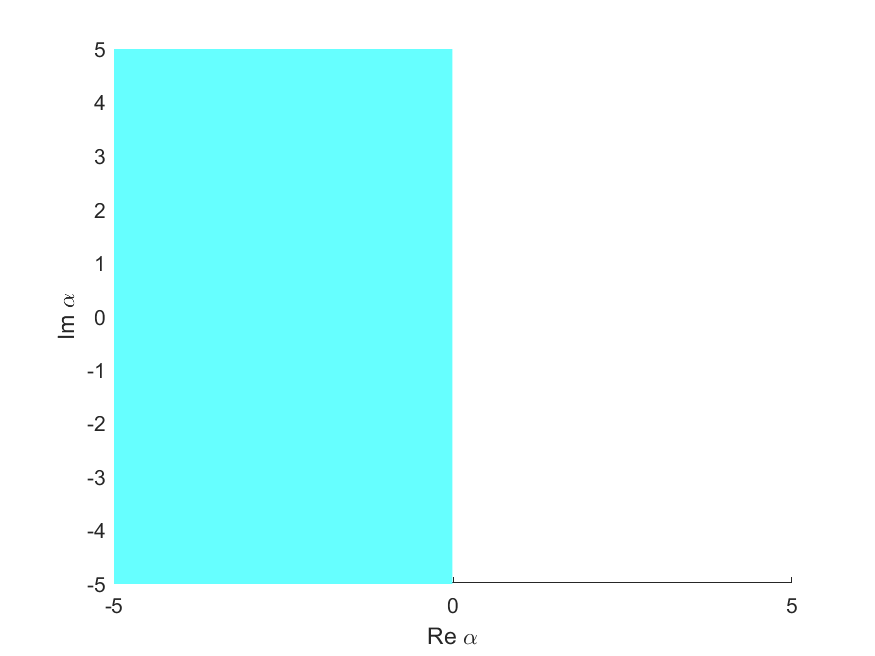}
         \caption{$V_{+}-V_{-}=0$.}
         \label{fig:Case 1}
     \end{subfigure}
     \hfill
     \begin{subfigure}[c]{0.45\textwidth}
         \centering
         \includegraphics[width=\textwidth]{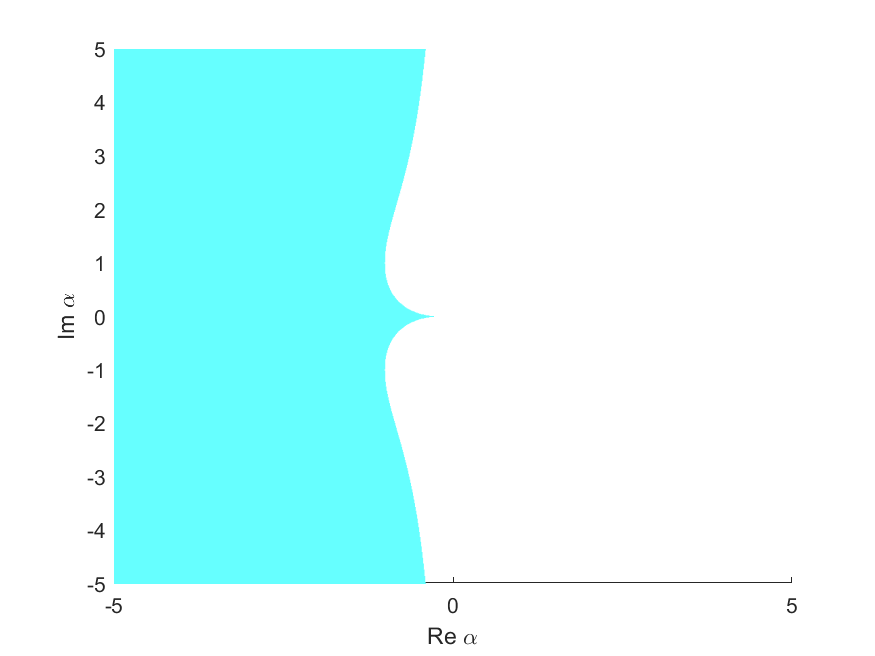}
         \caption{$V_{+}-V_{-}=2i$.}\label{SubFig: Omega}
         \label{fig:Case 2}
     \end{subfigure}
     \hfill
     \begin{subfigure}[c]{0.45\textwidth}
         \centering
         \includegraphics[width=\textwidth]{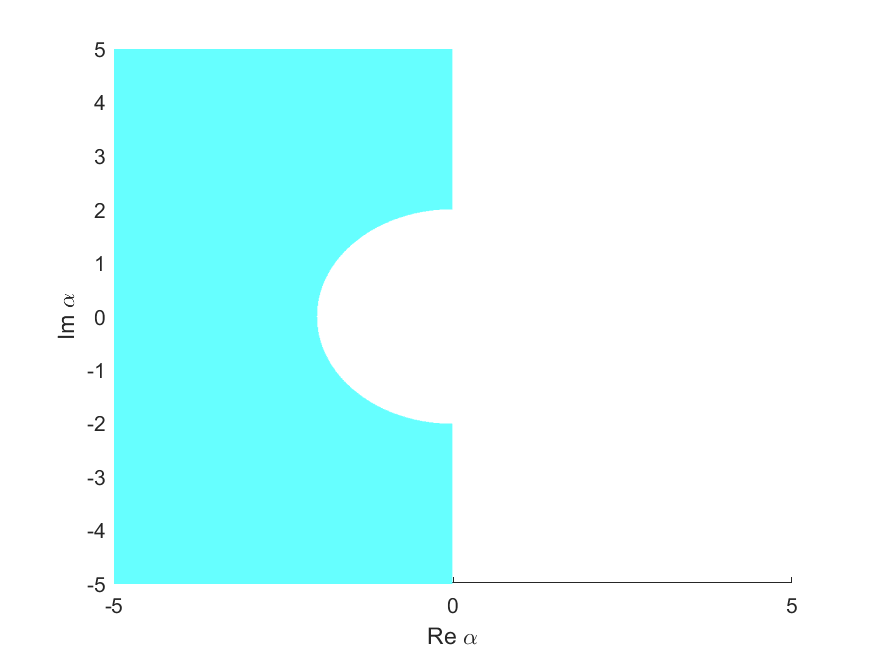}
         \caption{$V_{+}-V_{-}=4$.}
         \label{fig:Case 3}
     \end{subfigure}
     \hfill
     \begin{subfigure}[c]{0.45\textwidth}
         \centering
         \includegraphics[width=\textwidth]{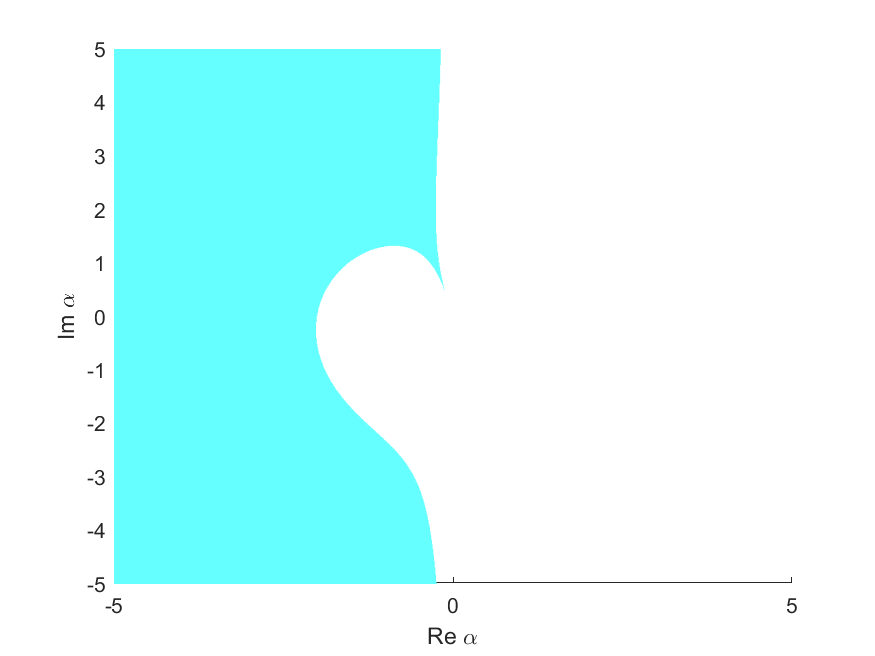}
         \caption{$V_{+}-V_{-}=4+i$.}
         \label{fig:Case 4}
     \end{subfigure}
        \caption{Illustrations of the set $\Omega$ in the complex plane, at which the eigenvalue of $\mathscr{L}_{\alpha}$ exists, corresponding to the values of $V_{+}-V_{-}$.}\label{Fig:Region Omega}
\end{figure*}
By applying the same method that we used to calculate the asymptotic behavior of the resolvent norm of $\mathscr{L}$ in Theorem \ref{Theo Norm of resolvent}, we deduce the following statement.
\begin{theorem}\label{Theo Norm of resolvent delta}
Let $V_{\pm}\in \C$ such that $\Im V_{+} \neq \Im V_{-}$ and let $\alpha\in \C\setminus \{0\}$ and $\mathscr{L}_{\alpha}$ be the operator defined as in Proposition \ref{Prop Property delta}. Let $z\in \C$ such that $\Im z\in |(\Im V_{+},\Im V_{-})|$, then 
\begin{equation*}
\left\Vert (\mathscr{L}_{\alpha}-z)^{-1}\right\Vert = \frac{|\Im V_{+}-\Im V_{-}|\sqrt{\Re z}}{|\alpha||\Im V_{+}-\Im z||\Im V_{-}-\Im z|}\left(1+\mathcal{O}\left(\frac{1}{|\Re z|^{1/2}}\right)\right).
\end{equation*}
as $\Re z\to +\infty$ and uniformly for all $\Im z\in |(\Im V_{+},\Im V_{-})|$.
\end{theorem}
It is clear that the closer to zero $\alpha$ is, the faster the resolvent norm increases. Although the blowing-up rate of the resolvent norm in the case $\alpha\neq 0$ is less than in the case free of interaction, the pseudospectrum is still non-trivial when $\Im V(x)$ is not constant. We do not know about the asymptotic behavior of the resolvent norm of $\mathscr{L}_{\alpha}$ outside the region bounded by two essential spectrum lines, but we conjecture that it will be bounded when the spectral parameter $z$ moves parallel to this region and decay when $z$ moves far away from this region to infinity.

Our final result is devoted to the pseudomode construction for the complex point interaction operator $\mathscr{L}_{\alpha}$. Again, our method for Theorem \ref{Theo Pseudomode} is still applicable and  yields an optimal pseudomode for the operator $\mathscr{L}_{\alpha}$.
\begin{theorem}\label{Theo Pseudomode delta}
Let $V_{\pm}\in \C$ such that $\Im V_{+} \neq \Im V_{-}$ and let $\alpha\in \C\setminus \{0\}$ and $\mathscr{L}_{\alpha}$ be the operator defined as in Proposition \ref{Prop Property delta}. We define a function $\Psi_{z,\alpha}$, for $z\in \rho(\mathscr{L}_{\alpha})$, as follows
\begin{equation}\label{Pseudomode delta}
\Psi_{z,\alpha}(x) = \left( n_{1}(z,\alpha) e^{k_{-}(z)x}+n_{2}(z) e^{\overline{k_{-}(z)}x}\right)\textbf{\textup{1}}_{\R_{-}}(x)+\left(p_{1}(z,\alpha) e^{-k_{+}(z)x}+p_{2}(z) e^{-\overline{k_{+}(z)}x}\right)\textbf{\textup{1}}_{\R_{+}}(x),
\end{equation}
where $n_{1}(z,\alpha),p_{1}(z,\alpha)$ and $n_{2}(z),p_{2}(z)$ are given by
\begin{equation}\label{Constants delta}
\left\{
\begin{aligned}
&n_{1}(z,\alpha)=\frac{k_{+}(z)+\overline{k_{-}(z)}+\alpha}{k_{+}(z)+k_{-}(z)+\alpha}|\Im V_{+}-\Im z|+\frac{k_{+}(z)-\overline{k_{+}(z)}}{k_{+}(z)+k_{-}(z)+\alpha}|\Im V_{-}-\Im z|,\\
&p_{1}(z,\alpha)=-\frac{k_{-}(z)-\overline{k_{-}(z)}}{k_{+}(z)+k_{-}(z)+\alpha}|\Im V_{+}-\Im z|-\frac{\overline{k_{+}(z)}+k_{-}(z)+\alpha}{k_{+}(z)+k_{-}(z)+\alpha}|\Im V_{-}-\Im z|,\\
&n_{2}(z)=-|\Im V_{+}-\Im z|,\\
&p_{2}(z)=|\Im V_{-}-\Im z|.
\end{aligned}
\right.
\end{equation}
Then, $\Psi_{z,\alpha}\in \textup{Dom}(\mathscr{L}_{\alpha})$ for all $z\in \rho(\mathscr{L}_{\alpha})$ and it makes
\[ \frac{\Vert(\mathscr{L}_{\alpha}-z)\Psi_{z,\alpha}\Vert}{\Vert \Psi_{z,\alpha} \Vert} = \frac{|\alpha|}{|\Im V_{+}-\Im V_{-}|} \frac{|\Im V_{+}-\Im z||\Im V_{-}-\Im z|}{\sqrt{\Re z}} \left(1+\mathcal{O}\left( \frac{1}{|\Re z|^{1/2}}\right)\right) ,\]
as $\Re z \to +\infty$ and uniformly for all $\Im z\in |(\Im V_{+},\Im V_{-})|$.
\end{theorem}
\section{Calculation resolvent and spectrum}\label{Section Resolvent Spectrum}
The main goal of this section is to study the resolvent and spectrum of the operator $\mathscr{L}$.
\subsection{Integral form of the resolvent: Proof of Proposition \ref{Prop Resolvent}}
We start by solving the resolvent equation. Let us fix $z\in \C \setminus ([V_{+},+\infty)\cup [V_{-},+\infty))$ and $f\in L^2(\R)$, we look for solution $u$ in $H^2(\R)$ such that
\[ (\mathscr{L}-z)u=f.\]
Because of the discontinuity of the potential $V$ at $0$, we look for the solution of the above equation in the form
\begin{align*}
u(x)=\left\{ \begin{aligned}
&u_{+}(x)\qquad &&\text{for }x>0,\\
&u_{-}(x)\qquad &&\text{for }x< 0.
\end{aligned} \right.
\end{align*}
It means that we need to find functions $u_{\pm}$ satisfying the corresponding equations
\begin{equation}\label{Non Homogeneous}
 -u_{\pm}''(x) +(V_{\pm}-z)u_{\pm}(x) = f(x).
\end{equation}
The variation of parameters method (VPM) is employed to find $u_{\pm}$, that is, we firstly find independent solutions associated with the homogeneous case (\emph{i.e.}, when $f=0$), they are
\[ e^{ k_{\pm}(z)x} \text{ and } e^{-k_{\pm}(z)x}.\]
Then, the general solutions of the non-homogeneous equations can be found in the form
\begin{equation}\label{u+-}
u_{\pm}(x) = \alpha_{\pm}(x) e^{ k_{\pm}(z)x} + \beta_{\pm}(x) e^{ -k_{\pm}(z)x},
\end{equation}
where $\alpha_{\pm}:\R_{\pm}\to \C$ and $\beta_{\pm}:\R_{\pm}\to \C$ are functions to be yet determined. By taking the first derivative of $u_{\pm}$, we obtain
\begin{align*}
u_{\pm}'(x) = \left[\alpha_{\pm}'(x) e^{ k_{\pm}(z)x} + \beta_{\pm}'(x) e^{ -k_{\pm}(z)x} \right]+ \left[\alpha_{\pm}(x) \left(e^{ k_{\pm}(z)x}\right)' + \beta_{\pm}(x) \left(e^{ -k_{\pm}(z)x}\right)'\right].
\end{align*}
The VPM starts by assuming that
\begin{equation}\label{VP 1}
\alpha_{\pm}'(x) e^{ k_{\pm}(z)x} + \beta_{\pm}'(x) e^{ -k_{\pm}(z)x}=0,
\end{equation}
an thus
\[u_{\pm}'(x) =\alpha_{\pm}(x) \left(e^{ k_{\pm}(z)x}\right)' + \beta_{\pm}(x) \left(e^{ -k_{\pm}(z)x}\right)' .\]
From this, the second derivative of $u_{\pm}$ is obtained, that is
\begin{align*}
u_{\pm}''(x) =  \left[\alpha_{\pm}'(x) \left(e^{ k_{\pm}(z)x}\right)' + \beta_{\pm}'(x) \left(e^{ -k_{\pm}(z)x}\right)'\right]+\left[\alpha_{\pm}(x) \left(e^{ k_{\pm}(z)x}\right)'' + \beta_{\pm}(x) \left(e^{ -k_{\pm}(z)x}\right)''\right].
\end{align*}
Replacing this into non-homogeneous equations \eqref{Non Homogeneous} and remember that $e^{ k_{\pm}(z)x}$ and $e^{-k_{\pm}(z)x}$ are solution of the homogeneous ones, it leads to
\begin{equation}\label{VP 2}
\alpha_{\pm}'(x) \left(e^{ k_{\pm}(z)x}\right)' + \beta_{\pm}'(x) \left(e^{ -k_{\pm}(z)x}\right)'=-f(x), \qquad \pm x>0.
\end{equation}
Solving \eqref{VP 1} and \eqref{VP 2}, we obtain 
\[ \alpha_{\pm}'(x)=-\frac{1}{2k_{\pm}(z)} e^{-k_{\pm}(z) x} f(x),\qquad \beta_{\pm}'(x)= \frac{1}{2k_{\pm}(z)} e^{k_{\pm}(z) x} f(x),\qquad \pm x>0\]
Hence, we can choose
\begin{equation}\label{alpha beta}
\begin{aligned}
\alpha_{\pm}(x)&=-\frac{1}{2k_{\pm}(z)}\int_{0}^{x} e^{-k_{\pm}(z) y} f(y)\, \dd y + A_{\pm}, \qquad &&\pm x >0,\\
\beta_{\pm}(x)&=\frac{1}{2k_{\pm}(z)}\int_{0}^{x} e^{k_{\pm}(z) y} f(y)\, \dd y + B_{\pm}, \qquad &&\pm x >0,
\end{aligned}
\end{equation}
where $A_{\pm}$, $B_{\pm}$ are some complex constants which are determined later by the fact that $u$ and its derivative $u'$ need to be continuous at zero and $u$ need to decay at infinities. We start with the decaying of $u$ at $+ \infty$. From the density of $C_{c}^{\infty}(\R)$ in $L^2(\R)$, for arbitrary $\varepsilon>0$, there exists $f_{\varepsilon}\in C_{c}^{\infty}(\R)$ such that $\Vert f-f_{\varepsilon} \Vert\leq \varepsilon$. Then, by the triangle and Holder inequalities, we have 
\begin{align*}
\left\vert \left(\int_{0}^{x} e^{k_{+}(z) y} f(y)\, \dd y\right) e^{-k_{+}(z)x} \right\vert &\leq \int_{0}^{x} e^{\Re k_{+}(z) (y-x)} |f(y)-f_{\varepsilon}(y) |\, \dd y +\int_{0}^{x} e^{\Re k_{+}(z) (y-x)} |f_{\varepsilon}(y) |\, \dd y\\
&\leq \sqrt{\frac{1-e^{-2\Re k_{+ }x}}{2\Re k_{+}(z)}} \Vert f -f_{\varepsilon}\Vert+\int_{0}^{x} e^{\Re k_{+}(z) (y-x)} |f_{\varepsilon}(y) |\, \dd y.
\end{align*}
Since $z\notin [V_{+},+\infty)$, thus $\Re k_{+}(z)(z)>0$ and we apply the dominated convergence theorem for the integral $\displaystyle \int_{0}^{x} e^{\Re k_{+}(z) (y-x)} |f_{\varepsilon}(y) |\, \dd y$, it yields that 
\[ \lim_{x\to+\infty} \left\vert \left(\int_{0}^{x} e^{k_{+}(z) y} f(y)\, \dd y\right) e^{-k_{+}(z)x} \right\vert \leq \frac{\varepsilon}{\sqrt{2 \Re k_{+}(z)}} . \]
From the arbitrariness of $\varepsilon$, it leads to
\begin{equation}
\lim_{x\to+\infty}  \left(\int_{0}^{x} e^{k_{+}(z) y} f(y)\, \dd y\right) e^{-k_{+}(z)x} =0.
\end{equation}
In other words, we have shown that $\displaystyle\lim_{x\to+\infty} \beta_{+}(x)e^{-k_{+}(z)x}=0$. Therefore, from the formula of $u_{+}$ in \eqref{u+-},   we have the following equivalences
\begin{equation}\label{A+}
\lim_{x\to +\infty} u_{+}(x) = 0 \Longleftrightarrow \lim_{x\to +\infty} \alpha_{+}(x)e^{k_{+}(z)x}=0  \Longleftrightarrow  A_{+}= \frac{1}{2k_{+}(z)}\int_{0}^{+\infty} e^{-k_{+}(z) y} f(y)\, \dd y.
\end{equation}
Indeed, the second equivalence (whose left-to-right implication is easy to see) comes from the density of $C_{c}^{\infty}(\R)$ in $L^2(\R)$, Holder inequality and the dominated convergence as above:
\begin{align*}
\left\vert \left(\int_{x}^{+\infty} e^{-k_{+}(z) y} f(y)\, \dd y\right) e^{k_{+}(z)x} \right\vert &\leq \int_{x}^{+\infty} e^{\Re k_{+}(z) (x-y)} |f(y)-f_{\varepsilon}(y) |\, \dd y \\
&\qquad+\int_{x}^{+\infty} e^{\Re k_{+}(z) (x-y)} |f_{\varepsilon}(y) |\, \dd y\\
&\leq \frac{\varepsilon}{\sqrt{2\Re k_{+}(z)}} +\int_{0}^{+\infty} e^{\Re k_{+}(z) (x-y)}\chi_{[x,+\infty)}(y) |f_{\varepsilon}(y) |\, \dd y\\
&\xrightarrow[]{x\to +\infty} \frac{\varepsilon}{\sqrt{2\Re k_{+}(z)}}.
\end{align*}
Similarly, we also have
\begin{equation}\label{B-}
 \lim_{x \to -\infty} u_{-}(x)=0 \qquad \Longleftrightarrow \qquad B_{-}= \frac{1}{2k_{-}(z)}\int_{-\infty}^{0} e^{k_{-}(z) y} f(y)\, \dd y.
\end{equation}
Since $u$ is found to satisfy regularity conditions
\[ u_{+}(0)=u_{-}(0),\qquad u_{+}'(0)=u_{-}'(0),\]
we obtain the system
\begin{equation*}
\left\{ \begin{aligned}
&A_{+}+B_{+}=A_{-}+B_{-},\\
&k_{+}(z)A_{+}-k_{+}(z)B_{+}=k_{-}(z)A_{-}-k_{-}(z)B_{-},
\end{aligned}
\right.
\end{equation*}
which allows us to determine $A_{-}$ and $B_{+}$ in terms of $A_{+}$ and $B_{-}$:
\begin{equation}\label{A+B-}
\left\{
\begin{aligned}
A_{-}&=\frac{2k_{+}(z)}{k_{+}(z)+k_{-}(z)}A_{+} -\frac{k_{+}(z)-k_{-}(z)}{k_{+}(z)+k_{-}(z)}B_{-},\\
B_{+}&=\frac{k_{+}(z)-k_{-}(z)}{k_{+}(z)+k_{-}(z)}A_{+}+\frac{2k_{-}(z)}{k_{+}(z)+k_{-}(z)}B_{-}.
\end{aligned}
\right.
\end{equation}
Here, $k_{+}(z)+k_{-}(z)\neq0 $ for all $z \notin [V_{\pm},+\infty)$ because $\Re \left(k_{+}(z)+k_{-}(z)\right)>0$ (from the choice of the principle branch of the square root). By replacing the values of constants $A_{\pm}$, $B_{\pm}$ in \eqref{A+}, \eqref{B-} and in \eqref{A+B-} into the formula of $u_{\pm}$ in \eqref{u+-}, we have
\begin{align*}
u_{+}(x)=&\frac{1}{2k_{+}(z)}\int_{0}^{+\infty} e^{-k_{+}(z)|x-y|}f(y)\, \dd y+\frac{k_{+}(z)-k_{-}(z)}{2k_{+}(z)(k_{+}(z)+k_{-}(z))}\int_{0}^{+\infty} e^{-k_{+}(z)(x+y)}f(y)\, \dd y\\
&+\frac{1}{k_{+}(z)+k_{-}(z)}\int_{-\infty}^{0} e^{-k_{+}(z)x+k_{-}(z)y}f(y)\, \dd y, \qquad x>0,\\
u_{-}(x)=&\frac{1}{2k_{-}(z)}\int_{-\infty}^{0} e^{-k_{-}(z)|x-y|}f(y)\, \dd y-\frac{k_{+}(z)-k_{-}(z)}{2k_{-}(z)(k_{+}(z)+k_{-}(z))}\int_{-\infty}^{0} e^{k_{-}(z)(x+y)}f(y)\, \dd y\\
&+\frac{1}{k_{+}(z)+k_{-}(z)}\int_{0}^{+\infty} e^{k_{-}(z)x-k_{+}(z)y}f(y)\, \dd y, \qquad x<0.
\end{align*}
Thus, given $f\in L^2(\R)$, we constructed a solution $u$ of the differential equation $(\mathscr{L}-z)u=f$ that has the integral form
\[ u(x)=\int_{\R} \mathcal{R}_{z}(x,y)f(y)\, \dd y,\]
where $\mathcal{R}_{z}(x,y)$ is given in the statement of Proposition \ref{Prop Resolvent}. After having a solution $u$ for the resolvent equation, we need to show that $u\in L^2(\R)$ and this can be done by using the Schur test, cf. \cite[Lem. 7.1]{Helffer13}. We will check that
\begin{equation}\label{Schur test}
\sup_{x\in \R} \int_{\R} |\mathcal{R}_{z}(x,y)|\, \dd y< +\infty \qquad\text{ and } \qquad\sup_{y\in \R} \int_{\R} |\mathcal{R}_{z}(x,y)|\, \dd x< +\infty.
\end{equation}
After noticing that $\mathcal{R}_{z}(x,y)$ is symmetric, \emph{i.e.,} $\mathcal{R}_{z}(x,y)=\mathcal{R}_{z}(y,x)$ for almost everywhere $(x,y)\in \R^2$, we just need to check the first one in \eqref{Schur test}. Directly from the formula of the kernel $\mathcal{R}_{z}(x,y)$, we have, for all $x>0$,
\begin{align*}
\int_{\R} |\mathcal{R}_{z}(x,y)|\, \dd y\leq &\frac{1}{|k_{+}(z)+k_{-}(z)|}\int_{-\infty}^{0} e^{-\Re k_{+}(z)x+ \Re k_{-}(z)y}\, \dd y + \frac{1}{2|k_{+}(z)|}\int_{0}^{+\infty} e^{-\Re k_{+}(z) |x-y|}\, \dd y\\
&+ \frac{|k_{+}(z)-k_{-}(z)|}{2|k_{+}(z)||k_{+}(z)+k_{-}(z)|} \int_{0}^{+\infty} e^{-\Re k_{+}(z)(x+y)}\, \dd y\\
\leq &\frac{1}{\Re k_{-}(z)|k_{+}(z)+k_{-}(z)|}+\frac{1}{\Re k_{+}(z)|k_{+}(z)|}+\frac{|k_{+}(z)-k_{-}(z)|}{2 \Re k_{+}(z)|k_{+}(z)||k_{+}(z)+k_{-}(z)|}.
\end{align*}
In the same way, for $x<0$, we obtain
\begin{align*}
\int_{\R} |\mathcal{R}_{z}(x,y)|\, \dd y
\leq &\frac{1}{\Re k_{+}(z)|k_{+}(z)+k_{-}(z)|}+\frac{1}{\Re k_{-}(z)|k_{-}(z)|}+\frac{|k_{+}(z)-k_{-}(z)|}{2 \Re k_{-}(z)|k_{+}(z)||k_{+}(z)+k_{-}(z)|}.
\end{align*}
Both the right hand sides of the above bounds for the integral $\int_{\R} |\mathcal{R}_{z}(x,y)|\, \dd y$ are finite provided $z\in \C\setminus \left( [V_{+},+\infty)\cup [V_{-},+\infty) \right)$. Thus, $u\in L^2(\R)$ and then $u$ will automatically belong to $H^2(\R)$ since $u''=(V-z)u-f\in L^2(\R)$.

\subsection{Characterization of the spectrum: Proof of Theorem \ref{Theo Spectrum of L}}\label{Subsec Spectrum L}
Let us consider $\xi\in C_{c}^{\infty}(\R)$ such that $\Vert \xi \Vert_{L^2}=1$ and its support lives in $(-1,1)$. We set
\[ \xi_{n}^{\pm}(x)\coloneqq \frac{1}{\sqrt{n}} \xi\left(\frac{x}{n} \mp n\right).\]
For $n\geq 1$, it is not hard to see that $\textup{Supp } \xi_{n}^{+}\subset (n^2-n,n^2+n)\subset \R_{+}$ and $\textup{Supp } \xi_{n}^{-}\subset (-n^2-n,-n^2+n)\subset \R_{-}$. Let us consider fixed but arbitrary $z_{\pm}\in [V_{\pm},+\infty)$ and we define sequence $u_{n}$ as follows
\[ u_{n}^{\pm}(x)\coloneqq \xi_{n}^{\pm}(x)e^{k_{\pm}(z_{\pm}) x}.\]
Notice that $k_{\pm}(z_{\pm})=\sqrt{z_{\pm}-V_{\pm}}\in i \overline{\R_{+}}$, it leads to the fact that $u_{n}^{\pm}$ is normalized:
\[ \Vert u_{n}^{\pm} \Vert_{L^2}^2 = \int_{\R} \frac{1}{n} \left\vert \xi\left(\frac{x}{n} \mp n\right)\right\vert^2\, \dd x=\Vert \xi \Vert_{L^2}^2=1.\]
Since $e^{k_{\pm}(z_{\pm}) x}$ satisfies $(\mathscr{L}-z_{\pm})e^{k_{\pm}(z_{\pm}) x}=0$ on $\R_{\pm}$, by taking the support of $\xi_{n}^{\pm}$ into account, we obtain
\begin{align*}
(\mathscr{L}-z_{\pm})u_{n}^{\pm}(x) &= \xi_{n}^{\pm}(x)(\mathscr{L}-\lambda_{\pm})e^{k_{\pm}(z_{\pm}) x}+\left[\mathscr{L}-\lambda,\xi_{n}^{\pm}(x) \right]e^{k_{\pm}(z_{\pm}) x}\\
&=\left[-\frac{\dd^2}{\dd x^2}, \xi_{n}^{\pm}(x)\right]e^{k_{\pm}(z_{\pm}) x}\\
&=\left( -\frac{\dd^2 \xi_{n}^{\pm}}{\dd x^2}(x)-2k_{\pm}(z_{\pm})\frac{\dd \xi_{n}^{\pm}}{\dd x}(x)\right)e^{k_{\pm}(z_{\pm}) x}.
\end{align*}
Here $[A,B]\coloneqq AB-BA$ is the notation of the commutator. Notice that, for each $k\in \{1,2\}$, we have
\[ \left\Vert \frac{\dd^{k}\xi_{n}^{\pm}}{\dd x^k}(x) \right\Vert^2 = \int_{\pm n^2-n}^{\pm n^2+n} \frac{1}{n^{2k+1}}\left\vert\xi^{(k)}\left(\frac{x}{n}-n\right)\right\vert^2\, \dd x=\frac{1}{n^{2k}}\int_{-1}^{1} \left\vert\xi^{(k)}(x)\right\vert^2\, \dd x \xrightarrow[]{n\to+\infty} 0.\]
It yields that $\Vert (\mathscr{L}-z_{\pm})u_{n}^{\pm}(x) \Vert\xrightarrow[]{n\to+\infty}0$. In other words, $\left(u_{n}^{+}\right)_{n\geq 1}$ forms a Weyl sequence for $z^{+}$ and $\left(u_{n}^{-}\right)_{n\geq 1}$ forms a Weyl sequence for $z^{-}$. Using for instance \cite[Lem. 3.3]{Cheverry-Raymond21}, we have $z_{\pm}\in \sigma(\mathscr{L})$ and from the arbitrariness of $z_{\pm}$ in $[V_{\pm},+\infty)$ we conclude that 
$$[V_{+},+\infty)\cup [V_{-},+\infty)\subset \sigma(\mathscr{L}).$$
From Proposition \ref{Prop Resolvent}, we deduce that
\[\sigma(\mathscr{L})\subset [V_{+},+\infty)\cup [V_{-},+\infty).\]
Then, the conclusion on the spectrum of $\mathscr{L}$ is followed. Since $\mathscr{L}$ is $J$-self-adjoint, its residual spectrum is empty (see \cite[Section 5.2.5.4]{Krejcirik-Siegl15}). Now, we will check that no points in $[V_{+},+\infty)\cup [V_{-},+\infty)$ can be the eigenvalue of $\mathscr{L}$, and thus $\sigma_{p}(\mathscr{L})=\emptyset$. Take $z\in [V_{+}, +\infty)$, assume that $z$ is the eigenvalue of $\mathscr{L}$ and $u$ is its associated eigenfunction, \emph{i.e.}, $(\mathscr{L}-z)u=0$. From \eqref{u+-}, the restriction of $u$ on $\R_{+}$ and $\R_{-}$, denoted by, respectively, $u_{+}$ and $u_{-}$, has the expression
\begin{equation}\label{Eigenfunction 1}
 u_{\pm}(x)=A_{\pm} e^{k_{\pm}(z)x} + B_{\pm}e^{-k_{\pm}(z)x},\qquad \text{for }\pm x >0.
\end{equation}
Since $z\in [V_{+}, +\infty)$, then $\Re k_{+}(z)=0$ and thus,
\[ | u_{+}(x)|^2 = |A_{+}|^2 + |B_{+}|^2 + 2 \Re \left(A_{+}\overline{B}_{+} e^{i\,\Im k_{+}(z) x}\right)\geq \left(|A_{+}|-|B_{+}|\right)^2.\]
The fact that $u_{+}$ is in $L^2(\R_{+})$ and that $\R_{+}$ is unbounded, it implies that $|A_{+}|=|B_{+}|$. By writing $A_{+}=|A_{+}| e^{i \textup{Arg }A_{+}}, B_{+}=|B_{+}| e^{i \textup{Arg }B_{+}}$, here $\textup{Arg }(w)$ denotes the principal argument of a complex number $w$, we obtain
\begin{align*}
|u_{+}(x)|^2 =& 2|A_{+}|^2\left( 1 + \cos\left(\textup{Arg }A_{+} - \textup{Arg }B_{+} + \Im k_{+}(z) x  \right)\right)\\
=&4|A_{+}|^2 \cos^{2}\left( \frac{\textup{Arg }A_{+} - \textup{Arg }B_{+} + \Im k_{+}(z) x }{2}\right).
\end{align*}
Since $u_{+}\in L^2(\R^{+})$, it yields that $A_{+}=0$ (since the trigonometry is not integrable on unbounded domain as $\R_{+}$) and thus $u_{+}=0$ on $\R_{+}$. Similarly, $u_{-}\in L^2(\R_{-})$ also implies that $u_{-}=0$ on $\R_{-}$ and thus, we have a contradiction. Therefore $[V_{+},+\infty)$ is not a subset of the point spectrum $\sigma_{p}(\mathscr{L}_{\alpha})$. This argument can be made for $[V_{-},+\infty)$ in the same manner. Therefore, the spectrum of $\mathscr{L}$ is purely continuous. 

In order to obtain the statement on essential spectra of $\mathscr{L}$, we will show that the above sequence $u_{n}$ is singular, \emph{i.e.}, $u_{n}$ converges weakly to zero.  Indeed, take $f\in L^2(\R)$, by the density of $C_{c}^{\infty}(\R)$ in $L^2(\R)$, for arbitrary $\varepsilon>0$, there exists a sequence $(f_{k})_{k\geq 1} \subset C_{c}^{\infty}(\R)$ such that $\Vert f-f_{k} \Vert\leq \varepsilon$. Then, by using triangle and Cauchy-Schwarz inequalities with $\Vert u_{n}^{\pm} \Vert=1$ and combining with the definition of $\xi_{n}$, we obtain
\begin{align*}
\vert \langle u_{n}^{\pm}, f \rangle \vert\leq |\langle u_{n}^{\pm}, f-f_{k}\rangle\vert + \vert \langle u_{n}^{\pm}, f_{k} \rangle \vert\leq \Vert f-f_{k}\Vert + \frac{1}{\sqrt{n}} \Vert \xi \Vert_{L^{\infty}} \Vert f_{k} \Vert_{L^1}\leq \varepsilon + \frac{1}{\sqrt{n}} \Vert \xi \Vert_{L^{\infty}} \Vert f_{k} \Vert_{L^1}
\end{align*}
Let $n\to +\infty$ and consider arbitrary $\varepsilon>0$, it implies that $\langle u_{n}, f \rangle \xrightarrow[n\to +\infty]{} 0$. Since $f$ is taken arbitrarily in $L^2(\R)$, we have just proved the weakly convergence to zero of $u_{n}$. Thanks to \cite[Theo. IX.1.3]{Edmunds-Evans18}, we implies that $\sigma_{\textup{e2}}(\mathscr{L})=\sigma(\mathscr{L})$. Since $\mathscr{L}$ is $J$-self-adjoint, four first essential spectra $\sigma_{\textup{ek}}(\mathscr{L})$ ($k\in \{1,2,3,4\}$) are identical (see \cite[Theo. IX.1.6]{Edmunds-Evans18}). Since $\C \setminus \sigma_{\textup{e1}}(\mathscr{L})=\rho(\mathscr{L})$ is connected, the fifth essential spectrum is also the same (cf. \cite[Prop. 5.4.4]{Krejcirik-Siegl15}).
\section{Pseudospectral estimates}\label{Section Resolvent Estimate}
From now on, unless otherwise stated, for simplicity, we will denote $z=\tau+i\delta$ with $(\tau,\delta)\in \R_{+} \times |(\Im V_{+},\Im V_{-})|$ and each time we write some asymptotic formula with big $\mathcal{O}$ notation, we understand that this formula happen as $\tau \to +\infty$ and uniformly for all $\delta \in |(\Im V_{+},\Im V_{-})|$.
\subsection{Resolvent estimate inside the numerical range: Proof of Theorem \ref{Theo Norm of resolvent}}\label{Subsec Resolvent inside N}
We rewrite $(\mathscr{L}-z)^{-1}$ as the sum of two integral operators
\[ (\mathscr{L}-z)^{-1}= \mathcal{R}_{1,z}+\mathcal{R}_{2,z},\]
where
\begin{equation}\label{R12}
\mathcal{R}_{1,z}f(x)\coloneqq\int_{\R} \mathcal{R}_{1,z}(x,y) f(y)\, \dd y,\qquad \mathcal{R}_{2,z}f(x)\coloneqq\int_{\R} \mathcal{R}_{2,z}(x,y) f(y)\, \dd y,
\end{equation}
whose kernel are given by
\begin{align}
\mathcal{R}_{1,z}(x,y) &=  \left\{
\begin{aligned}
			 & K_{++}(z)e^{-k_{+}(z)(x+y)}, \qquad && \text{for } \{x> 0, y> 0\}; \\
			 &  K_{--}(z)e^{k_{-}(z)(x+y)},\qquad && \text{for } \{x < 0, y< 0\};\\  
			 & K_{+-}(z)e^{-k_{+}(z)x+k_{-}(z)y}, \qquad&& \text{for } \{x>0, y< 0\}; \\			 
			 & K_{+-}(z)e^{k_{-}(z)x-k_{+}(z)y}, \qquad && \text{for } \{x< 0, y> 0\}; 
		\end{aligned}
\right.\label{R1}\\
\mathcal{R}_{2,z}(x,y)&=\left\{
\begin{aligned}
			 &  \frac{1}{2k_{+}(z)}e^{-k_{+}(z)|x-y|},\qquad && \text{for } \{x> 0, y> 0\};\\ 		 
			 &  \frac{1}{2k_{-}(z)}e^{-k_{-}(z)|x-y|},\qquad && \text{for } \{x< 0, y< 0\};\\ 
			 & 0,\qquad && \text{for } \{x.y<0\}.
		\end{aligned}
\right.\label{R2}
\end{align}
Here, we denote
\begin{equation}\label{Big K}
\begin{aligned}
K_{++}(z)&\coloneqq\frac{k_{+}(z)-k_{-}(z)}{2k_{+}(z)\left(k_{+}(z)+k_{-}(z) \right)},\qquad K_{--}(z)\coloneqq-\frac{k_{+}(z)-k_{-}(z)}{2k_{-}(z)\left(k_{+}(z)+k_{-}(z) \right)},\\
K_{+-}(z) &\coloneqq\frac{1}{k_{+}(z)+k_{-}(z)}.
\end{aligned}
\end{equation}
Our strategy is to show that the norm of $\mathcal{R}_{1,z}$ will play the main role, while the norm of $\mathcal{R}_{2,z}$ is just a small perturbation compared with $\mathcal{R}_{1,z}$ in the divergence of the resolvent norm inside the numerical range. To do that, two-sided estimate of the norm of $\mathcal{R}_{1,z}$ will be clearly established with the help of the following optimization lemma.
\begin{lemma}\label{Lem Optimize}
Let $A,B,C\in \R$, consider a function of two variables
\[f(x,y)=Ax^2+Bxy+Cy^2\]
on the circle $x^2+y^2=1$. Then it attains the maximum on this circle  and
\begin{align*}
&\max_{x^2+y^2=1} f(x,y)=\frac{A+C+\sqrt{(A-C)^2+B^2}}{2}.
\end{align*}
\end{lemma}
\begin{proof}
Let us write $x=\cos(\theta)$ and $y=\sin(\theta)$ for $\theta \in [0,2 \pi)$ and write
\begin{align*}
f(\cos(\theta),\sin(\theta))=&A \cos^2(\theta) + B \sin(\theta) \cos(\theta) + C \sin^2(\theta)\\
=&\frac{1}{2}\left( A+C + (A-C)\cos(2\theta)+ B \sin(2\theta)\right)
\end{align*}
By using the Cauchy-Schwarz inequality, we obtain the upper bound, for every $\theta\in [0,2\pi)$,
\[ f(\cos(\theta),\sin(\theta))\leq \frac{1}{2}\left( A+C+\sqrt{(A-C)^2+B^2}\right),\]
and the equality can be obtained when
\[ \left(\cos(2\theta), \sin(2\theta)\right)=\pm \left(\frac{A-C}{\sqrt{(A-C)^2+B^2}},\frac{B}{\sqrt{(A-C)^2+B^2}}\right).\]
Then, the conclusion of the lemma follows.
\end{proof}
\begin{proposition}\label{Prop Norm R1}
Let $\mathcal{R}_{1,z}$ be the integral operator with the kernel $\mathcal{R}_{1,z}(x,y)$ defined as in \eqref{R1}. For all $z\in \rho(\mathscr{L})$, $\mathcal{R}_{1,z}$ is a bounded operator on $L^2(\R)$ whose norm satisfies
\begin{equation}\label{Bound R1}
\begin{aligned}
\Vert \mathcal{R}_{1,z} \Vert &\leq \frac{1}{\sqrt{2}} \sqrt{\sqrt{(A(z)-C(z))^2+B(z)^2}+A(z)+C(z)+2D(z)},\\
\Vert \mathcal{R}_{1,z} \Vert &\geq \frac{1}{\sqrt{2}} \sqrt{\sqrt{(A(z)-C(z))^2+\widetilde{B}(z)^2}+A(z)+C(z)+2D(z)},
\end{aligned}
\end{equation}
where 
\begin{align*}
A(z) &\coloneqq\frac{|K_{++}(z)|^2}{4 (\Re k_{+}(z))^2},\qquad C(z) \coloneqq \frac{|K_{--}(z)|^2}{4 (\Re k_{-}(z))^2},\qquad D(z)\coloneqq\frac{|K_{+-}(z)|^2}{4 (\Re k_{-}(z))(\Re k_{+}(z))},\\
B(z) &\coloneqq \frac{|K_{+-}(z)|}{2\sqrt{(\Re k_{-}(z))(\Re k_{+}(z))}}\left(\frac{|K_{++}(z)|}{\Re k_{+}(z)}+ \frac{|K_{--}(z)|}{\Re k_{-}(z)}\right),\\
\widetilde{B}(z) &\coloneqq\frac{1}{2\sqrt{(\Re k_{-}(z))(\Re k_{+}(z))}}\left(\frac{\Re\left( K_{++}(z)\overline{K_{+-}(z)}\right)}{\Re k_{+}(z)}+\frac{\Re\left( K_{--}(z)\overline{K_{+-}(z)}\right)}{\Re k_{-}(z)}\right).
\end{align*}
\end{proposition}
\begin{proof}
Consider $f\in L^2(\R)$ such that $\Vert f \Vert=1$, we have
\begin{align*}
 \Vert \mathcal{R}_{1,z} f \Vert^2=& \int_{\R} \vert \mathcal{R}_{1,z} f(x) \vert^2\, \dd x\\
=&\int_{0}^{+\infty} \left\vert \int_{-\infty}^{0}K_{+-}(z)e^{-k_{+}(z)x+k_{-}(z)y}f(y)\dd y+\int_{0}^{+\infty}K_{++}(z)e^{-k_{+}(z)(x+y)}f(y)\,\dd y\right\vert^2\, \dd x\\
&+\int_{-\infty}^{0} \left\vert \int_{-\infty}^{0}K_{--}(z)e^{k_{-}(z)(x+y)}f(y)\dd y+\int_{0}^{+\infty}K_{+-}(z)e^{k_{-}(z)x-k_{+}(z)y}f(y)\,\dd y\right\vert^2\, \dd x\\
=&\frac{1}{2\Re k_{+}(z)}\left\vert \int_{-\infty}^{0}K_{+-}(z)e^{k_{-}(z)y}f(y)\dd y+\int_{0}^{+\infty}K_{++}(z)e^{-k_{+}(z)y}f(y)\,\dd y\right\vert^2\\
&+\frac{1}{2 \Re k_{-}(z)} \left\vert \int_{-\infty}^{0}K_{--}(z)e^{k_{-}(z)y}f(y)\dd y+\int_{0}^{+\infty}K_{+-}(z)e^{-k_{+}(z)y}f(y)\,\dd y\right\vert^2.
\end{align*}
Using Holder's inequality and remember that $\Vert f \Vert_{L^2(\R_{-})}^2+\Vert f \Vert_{L^2(\R_{+})}^2=1$, we obtain
\begin{align*}
\Vert \mathcal{R}_{1,z} f \Vert^2\leq &\frac{1}{2 \Re k_{+}(z)}\left( \frac{|K_{+-}(z)|}{\sqrt{2 \Re k_{-}(z)}} \Vert f \Vert_{L^2(\R_{-})}+\frac{|K_{++}(z)|}{\sqrt{2 \Re k_{+}(z)}} \Vert f \Vert_{L^2(\R_{+})}\right)^2\\
&+\frac{1}{2 \Re k_{-}(z)} \left( \frac{|K_{--}(z)|}{\sqrt{2 \Re k_{-}(z)}} \Vert f \Vert_{L^2(\R_{-})}+\frac{|K_{+-}(z)|}{\sqrt{2 \Re k_{+}(z)}} \Vert f \Vert_{L^2(\R_{+})}\right)^2\\
=&  \frac{|K_{++}(z)|^2}{4 (\Re k_{+}(z))^2} \Vert f \Vert_{L^2(\R_{+})}^2+\frac{|K_{--}(z)|^2}{4 (\Re k_{-}(z))^2} \Vert f \Vert_{L^2(\R_{-})}^2+\frac{|K_{+-}(z)|^2}{4 (\Re k_{-}(z))(\Re k_{+}(z))}\\
&+\frac{|K_{+-}(z)|}{2\sqrt{(\Re k_{-}(z))(\Re k_{+}(z))}}\left( \frac{|K_{--}(z)|}{\Re k_{-}(z)}+\frac{|K_{++}(z)|}{\Re k_{+}(z)}\right)\Vert f \Vert_{L^2(\R_{-})} \Vert f \Vert_{L^2(\R_{+})}.
\end{align*}
By applying Lemma \ref{Lem Optimize}, it yields that
\begin{align*}
\Vert \mathcal{R}_{1,z} f \Vert^2\leq \frac{1}{2} \left(A(z)+C(z)+\sqrt{(A(z)-C(z))^2+B(z)^2} +2D(z) \right),
\end{align*}
where $A(z),B(z),C(z),D(z)$ defined as in the statement of this proposition. For all $z\in \rho(\mathscr{L})$, $\Re k_{\pm}(z)>0$, then all coefficients $A(z),B(z),C(z),D(z)$ are finite, for that reason, $\mathcal{R}_{1,z}$ is bounded and the upper bound in \eqref{Bound R1} is obtained. In order to get the lower bound, we introduce a test function
\begin{equation}\label{f0}
f_{0}(y)=\left\{
\begin{aligned}
&\alpha e^{-\overline{k_{+}(z)} y} &&\text{for } y\geq 0,\\
&\beta e^{\overline{k_{-}(z)} y} &&\text{for } y\leq 0,
\end{aligned}
\right.
\end{equation}
where $\alpha, \beta \in \R$ will be determined later.\\
By straightforward computation, the action of the operator $\mathcal{R}_{1,z}f_{0}$ is given by
\begin{align*}
\left[\mathcal{R}_{1,z}f_{0}\right](x) =\left\{
\begin{aligned}
&e^{-k_{+}(z)x}\left( \frac{\alpha K_{++}(z)}{2\Re k_{+}(z)}+\frac{\beta K_{+-}(z)}{2 \Re k_{-}(z)}  \right) && \text{ for }x>0,\\
&e^{k_{-}(z)x}\left(\frac{\alpha K_{+-}(z)}{2\Re k_{+}(z)}+ \frac{\beta K_{--}(z)}{2 \Re k_{-}(z)} \right) && \text{ for }x<0.
\end{aligned}
\right.
\end{align*}
It yields that
\begin{align*}
\Vert \mathcal{R}_{1,z} f_{0} \Vert^2 &=\frac{1}{2\Re k_{+}(z)} \left\vert \frac{\alpha K_{++}(z)}{2\Re k_{+}(z)}+\frac{\beta K_{+-}(z)}{2 \Re k_{-}(z)}  \right\vert^2+\frac{1}{2\Re k_{-}(z)} \left\vert\frac{\alpha K_{+-}(z)}{2\Re k_{+}(z)}+\frac{\beta K_{--}(z)}{2 \Re k_{-}(z)}\right\vert^2.
\end{align*}
By the definition of $f_{0}$, we have
\[ \Vert f_{0} \Vert^2 = \frac{|\alpha|^2}{2\Re k_{+}(z)}+\frac{|\beta|^2}{2\Re k_{-}(z)}. \]
To normalize the norm of $f_{0}$, let us write
\[ \alpha = \sqrt{2 \Re k_{+}(z)} a,\qquad \beta = \sqrt{2 \Re k_{-}(z)}b, \text{ with some } a,b\in \R: a^2+b^2=1.\]
Then, our problem turns to find $(a,b)\in \R$ such that the quantity
\begin{align*}
\Vert \mathcal{R}_{1,z} f_{0} \Vert^2=&\frac{1}{2\Re k_{+}(z)} \left\vert \frac{a K_{++}(z)}{\sqrt{2\Re k_{+}(z)}}+\frac{b K_{+-}(z)}{\sqrt{2 \Re k_{-}(z)}} \right\vert^2+\frac{1}{2\Re k_{-}(z)} \left\vert \frac{a K_{+-}(z)}{\sqrt{2\Re k_{+}(z)}}+\frac{b K_{--}(z)}{\sqrt{2 \Re k_{-}(z)}}  \right\vert^2\\
=&\frac{a^2|K_{++}(z)|^2}{4 (\Re k_{+}(z))^2}+\frac{b^2 |K_{--}(z)|^2}{4 (\Re k_{-}(z))^2} + \frac{|K_{+-}(z)|^2}{4(\Re k_{-}(z))(\Re k_{+}(z))}\\
&+\frac{1}{2\sqrt{(\Re k_{-}(z))(\Re k_{+}(z))}}\left(\frac{\Re\left( K_{++}(z)\overline{K_{+-}(z)}\right)}{\Re k_{+}(z)}+\frac{\Re\left( K_{--}(z)\overline{K_{+-}(z)}\right)}{\Re k_{-}(z)}\right)ab
\end{align*}
attains its maximum. Lemma \ref{Lem Optimize} tells us that there exists such a couple $(a,b)$ and its maximum is
\[ \Vert \mathcal{R}_{1,z} f_{0} \Vert^2= \frac{1}{2} \left(A(z)+C(z)+\sqrt{(A(z)-C(z))^2+\widetilde{B}(z)^2} +2D(z) \right).\]
With this choice of test function, the quantity $\Vert \mathcal{R}_{1,z} f_{0} \Vert$ becomes the lower bound for the norm of $\mathcal{R}_{1,z}$.
\end{proof}
After having explicit formulas for two-sided bounds of the norm of $\mathcal{R}_{1,z}$ in the above proposition, we will send $z$ to infinity inside the numerical range of $\mathscr{L}$ to see the asymptotic behavior of these two-sided bounds. The following lemma will be employed to provide us some useful asymptotic formulas from which we deduce the asymptotic behavior of these two-sides bounds.
\begin{lemma}\label{Lem Asymptotic}
Let $z=\tau+i\delta$ with $\tau>0$, $\delta\in K\setminus \{\Im V_{+},\Im V_{-}\}$ with $K$ is some compact set in $\R$. Then, the following asymptotic formulas hold as $\tau\to +\infty$ and uniformly for all $\delta\in K \setminus \{\Im V_{+},\Im V_{-}\}$:
\begin{equation*}
\begin{aligned}
k_{+}(z)=&\sgn(\Im V_{+}-\delta)\left[i\sqrt{\tau} + \frac{(\Im V_{+}-\delta)-i\Re V_{+}}{2\sqrt{\tau}} +\mathcal{O}\left( \frac{1}{|\tau|^{3/2}}\right)\right],\\
k_{-}(z)=&\sgn(\Im V_{-}-\delta)\left[i\sqrt{\tau} + \frac{(\Im V_{-}-\delta)-i\Re V_{-}}{2\sqrt{\tau}} +\mathcal{O}\left( \frac{1}{|\tau|^{3/2}}\right)\right],
\end{aligned}
\end{equation*}
and 
\begin{equation*}
\begin{aligned}
&\Re k_{+}(z) = \frac{|\Im V_{+}-\delta|}{2\sqrt{\tau}}\left(1+\mathcal{O}\left(\frac{1}{|\tau|}\right)\right),\qquad &\Re k_{-}(z) = \frac{|\Im V_{-}-\delta|}{2\sqrt{\tau}}\left(1+\mathcal{O}\left(\frac{1}{|\tau|}\right)\right).
\end{aligned}
\end{equation*}
\end{lemma}
\begin{proof}
Let us give a proof for $k_{+}(z)$ and $\Re k_{+}(z)$, the asymptotics for $k_{-}(z)$ and $\Re k_{-}(z)$ is obtained by just changing from the plus sign to the minus sign. We start with $k_{+}(z)$. Recalling that the square root that we fixed in this article is in principal branch, thus, for $\tau>0$, we have
\begin{equation*}
\begin{aligned}
k_{+}(z)=&  \sqrt{\Re V_{+}-\tau+i(\Im V_{+}-\delta)}\\
=&i\sqrt{\tau}\sgn(\Im V_{+}-\delta)\sqrt{\frac{-\Re V_{+}-i(\Im V_{+}-\delta)}{\tau}+1}.
\end{aligned}
\end{equation*}
We consider a smooth complex-valued function $F_{\delta}$ defined on $\R$ by
\[ F_{\delta}(x)=\sqrt{ax+1},\qquad a=-\Re V_{+}-i(\Im V_{+}-\delta).\]
By applying Taylor's theorem (for example, \cite[Theo. 1.36]{Knapp05}) for $F_{\delta}$ expanding at zero, we have, for all $x$ in some neighbourhood of zero,
\[ F_{\delta}(x)=1+\frac{a}{2}x+\frac{x^2}{2} \int_{0}^{1} (1-s) F_{\delta}^{(2)}(xs)\, \dd s.\]
The second derivative of $F_{\delta}$ can be calculated explicitly as follows
\[ F_{\delta}^{(2)}(x)=-\frac{a^2}{4\sqrt{(ax+1)^3}}.\]
By considering $|x|$ sufficiently small and using the fact that $|ax+1|\gtrsim 1$ uniformly for all $\delta \in K$, we get that $|F_{\delta}^{(2)}(x)|\lesssim 1$ uniformly for all $\delta\in K$ and thus, as $|x| \to 0$,
\[ F_{\delta}(x)=1+\frac{a}{2}x+\frac{x^2}{2} + \mathcal{O}(|x|^2) \qquad \text{uniformly for all } \delta\in K.\]
Replacing $x=\frac{1}{\tau}$, we obtain the asymptotic for $k_{+}(z)$ as in statement of the lemma. Next, to obtain the expansion for $\Re k_{+}(z)$, we use the algebraic formula for the square root of a non-real complex number $w$:
\[ \Re \sqrt{w}= \frac{|\Im w |}{\sqrt{2(|w|-\Re w)}}. \]
For $\delta \neq \Im V_{+}$, we have
\begin{equation}\label{Re k+}
\Re k_{+}(z)= \frac{|\Im V_{+}-\delta|}{\sqrt{2} \sqrt{|k_{+}(z)|^2+\tau-\Re V_{+}}}.
\end{equation}
Since $|k_{+}(z)|^2= \sqrt{(\tau-\Re V_{+})^2+(\Im V_{+}-\delta)^2}$, we can show that
\begin{align*}
\sqrt{|k_{+}(z)|^2+\tau-\Re V_{+}}- \sqrt{2\tau} = \frac{-4 (\Re V_{+}) \tau +(\Im V_{+}-\delta)^2}{\left(\sqrt{|k_{+}(z)|^2+\tau -\Re V_{+}}+\sqrt{2\tau} \right) \left(|k_{+}(z)|^2+\tau +\Re V_{+} \right)}.
\end{align*}
Therefore, as $\tau \to +\infty$, we get
\[\sqrt{|k_{+}(z)|^2+\tau-\Re V_{+}}=\sqrt{2\tau} + \mathcal{O}\left(\frac{1}{\sqrt{\tau}}\right). \]
Replacing this into the denominator of the right hand side of \eqref{Re k+}, we deduce the asymptotic behavior of $\Re k_{+}(z)$ as in the statement of the lemma.
\end{proof}
Now, we restrict ourselves to $\delta \in |(\Im V_{+},\Im V_{-})|$, note that we have the following relations
\begin{equation}\label{d-ImV}
\sgn(\Im V_{+}-\delta)=\sgn(\Im V_{+}-\Im V_{-}),\qquad \sgn(\Im V_{-}-\delta)=-\sgn(\Im V_{+}-\Im V_{-}).
\end{equation}
As a consequence, Lemma \ref{Lem Asymptotic} produces the following asymptotic expansion
\begin{equation}\label{Asymptotics k}
\begin{aligned}
&|k_{+}(z)|=\sqrt{\tau}\left(1+\mathcal{O}\left(\frac{1}{|\tau|}\right)\right), &|k_{-}(z)|=\sqrt{\tau}\left(1+\mathcal{O}\left(\frac{1}{|\tau|}\right)\right),\\
&|k_{+}(z)+k_{-}(z)|=\frac{|V_{+}-V_{-}|}{2\sqrt{\tau}}\left(1+\mathcal{O}\left(\frac{1}{|\tau|}\right)\right), &|k_{+}(z)-k_{-}(z)|=2\sqrt{\tau}\left(1+\mathcal{O}\left(\frac{1}{|\tau|}\right)\right).
\end{aligned}
\end{equation}
Using these expansions, we obtain the asymptotic formulas for $|K_{++}(z)|, |K_{--}(z)|$ and $|K_{+-}(z)|$ defined in \eqref{Big K}:
\begin{align*}
&|K_{++}(z)|=\frac{2\sqrt{\tau}}{|V_{+}-V_{-}|}\left(1+\mathcal{O}\left(\frac{1}{|\tau|}\right)\right), &|K_{--}(z)|=\frac{2\sqrt{\tau}}{|V_{+}-V_{-}|}\left(1+\mathcal{O}\left(\frac{1}{|\tau|}\right)\right),\\
&|K_{+-}(z)|=\frac{2\sqrt{\tau}}{|V_{+}-V_{-}|}\left(1+\mathcal{O}\left(\frac{1}{|\tau|}\right)\right),
\end{align*}
Next step, we compute the asymptotic expansions for $A(z), B(z), C(z), D(z)$ and $\widetilde{B}(z)$ appearing in Proposition \ref{Prop Norm R1}. For $A(z), C(z), D(z)$, it is easier to obtain their asymptotic formulas by using Lemma \ref{Lem Asymptotic} for $\Re k_{\pm}(z)$ and the above estimates for $K_{++}(z)$, $K_{--}(z)$ and $K_{+-}(z)$:
\begin{align*}
A(z)=&\frac{4\tau^2}{|V_{+}-V_{-}|^2|\Im V_{+}-\delta|^2}\left(1+\mathcal{O}\left(\frac{1}{|\tau|}\right)\right),\\
C(z)=&\frac{4\tau^2}{|V_{+}-V_{-}|^2|\Im V_{-}-\delta|^2}\left(1+\mathcal{O}\left(\frac{1}{|\tau|}\right)\right),\\
D(z)=&\frac{4\tau^2}{|V_{+}-V_{-}|^2|\Im V_{+}-\delta||\Im V_{-}-\delta|}\left(1+\mathcal{O}\left(\frac{1}{|\tau|}\right)\right).
\end{align*}
For $B(z)$ and $\widetilde{B}(z)$, we need more efforts, but it is a straightforward calculation:
\begin{align*}
B(z)=&\frac{2\tau \left(1+\mathcal{O}\left(\frac{1}{|\tau|}\right)\right)}{|V_{+}-V_{-}||\Im V_{+}-\delta|^{1/2}|\Im V_{-}-\delta|^{1/2}}\\
&\times\left(\frac{\frac{4\tau}{|V_{+}-V_{-}|}\left(1+\mathcal{O}\left(\frac{1}{|\tau|}\right)\right)}{|\Im V_{+}-\delta|}+\frac{\frac{4\tau}{|V_{+}-V_{-}|}\left(1+\mathcal{O}\left(\frac{1}{|\tau|}\right)\right)}{|\Im V_{-}-\delta|} \right)\\
=&\frac{2\tau \left(1+\mathcal{O}\left(\frac{1}{|\tau|}\right)\right)}{|V_{+}-V_{-}||\Im V_{+}-\delta|^{1/2}|\Im V_{-}-\delta|^{1/2}} \frac{4|\Im V_{+}-\Im V_{-}|\tau  \left(1+\mathcal{O}\left(\frac{1}{|\tau|}\right)\right)}{|V_{+}-V_{-}||\Im V_{+}-\delta||\Im V_{-}-\delta|}\\
=&\frac{8|\Im V_{+}-\Im V_{-}|\tau^2}{|V_{+}-V_{-}|^2|\Im V_{+}-\delta|^{3/2}|\Im V_{-}-\delta|^{3/2}}\left(1+\mathcal{O}\left(\frac{1}{|\tau|}\right)\right).
\end{align*}
Here, in the second equality, we used \eqref{d-ImV} to obtain 
\begin{equation}\label{Equality for delta}
|\Im V_{+}-\delta|+|\Im V_{-}-\delta|=|\Im V_{+}- \Im V_{-}|\qquad \text{for all }\delta \in |(\Im V_{+}, \Im V_{-})|.
\end{equation}
From the formula of $\widetilde{B}(z)$, in order to have its asymptotic, we firstly need to calculate $\Re\left( K_{++}(z) \overline{K_{+-}(z)}\right)$ and $\Re\left( K_{--}(z) \overline{K_{\pm}(z)}\right)$. From the definitions of $K_{++}(z)$ and $K_{+-}(z)$ in \eqref{Big K} and thanks to Lemma \ref{Lem Asymptotic}, \eqref{d-ImV} and \eqref{Asymptotics k}, we get
\begin{align*}
\Re\left( K_{++}(z) \overline{K_{+-}(z)}\right)= &\frac{1}{|k_{+}(z)+k_{-}(z)|^2} \Re \left( \frac{k_{+}(z)-k_{-}(z)}{2k_{+}(z)}\right)\\ =&\frac{4\tau}{|V_{+}-V_{-}|^2}\left(1+\mathcal{O}\left(\frac{1}{|\tau|}\right)\right)\Re \left(\frac{2i\sqrt{\tau}+\mathcal{O}\left(\frac{1}{|\tau|^{1/2}}\right)}{2i\sqrt{\tau}+\mathcal{O}\left(\frac{1}{|\tau|^{1/2}}\right)} \right)\\
=&\frac{4\tau}{|V_{+}-V_{-}|^2}\left(1+\mathcal{O}\left(\frac{1}{|\tau|}\right)\right)\left(1+\mathcal{O}\left(\frac{1}{|\tau|}\right)\right)\\
=&\frac{4\tau}{|V_{+}-V_{-}|^2}\left(1+\mathcal{O}\left(\frac{1}{|\tau|}\right)\right),
\end{align*}
and in the same way, we also obtain
\[ \Re\left( K_{--}(z) \overline{K_{\pm}(z)}\right)=\frac{4\tau}{|V_{+}-V_{-}|^2}\left(1+\mathcal{O}\left(\frac{1}{|\tau|}\right)\right).\]
From these estimates, the asymptotic expansion for $\widetilde{B}(z)$ is followed from Lemma \ref{Lem Asymptotic} (for $\Re k_{\pm}(z)$) and \eqref{Equality for delta}:
\begin{align*}
\widetilde{B}(z)=& \frac{\sqrt{\tau}\left(1+\mathcal{O}\left(\frac{1}{|\tau|}\right)\right)}{|\Im V_{+}-\delta|^{1/2}|\Im V_{-}-\delta|^{1/2}}\left(\frac{\frac{8\tau^{3/2}}{|V_{+}-V_{-}|^2}\left(1+\mathcal{O}\left(\frac{1}{|\tau|}\right)\right)}{|\Im V_{+}-\delta|}+\frac{\frac{8\tau^{3/2}}{|V_{+}-V_{-}|^2}\left(1+\mathcal{O}\left(\frac{1}{|\tau|}\right)\right)}{|\Im V_{-}-\delta|} \right)\\
=&\frac{8|\Im V_{+}-\Im V_{-}|\tau^2}{|V_{+}-V_{-}|^2|\Im V_{+}-\delta|^{3/2}|\Im V_{-}-\delta|^{3/2}}\left(1+\mathcal{O}\left(\frac{1}{|\tau|}\right)\right).
\end{align*}
We notice that $\widetilde{B}(z)$ has the same asymptotic limit as $B(z)$, this makes the upper bound and the lower bound of $\Vert \mathcal{R}_{1,z} \Vert$ in \eqref{Bound R1} also have the same asymptotic limit as $\tau \to +\infty$. Now, we can calculate the asymptotic behavior of all terms appearing in these bounds:
\begin{align*}
\sqrt{\left(A(z)-C(z)\right)^2+B(z)^2}=&\frac{4|\Im V_{+}-\Im V_{-}|^2\tau^2}{|V_{+}-V_{-}|^2|\Im V_{+}-\delta|^2|\Im V_{-}-\delta|^2}\left(1+\mathcal{O}\left(\frac{1}{|\tau|}\right)\right),\\
A(z)+C(z)+2D(z)=&\frac{4|\Im V_{+}-\Im V_{-}|^2\tau^2}{|V_{+}-V_{-}|^2|\Im V_{+}-\delta|^2|\Im V_{-}-\delta|^2}\left(1+\mathcal{O}\left(\frac{1}{|\tau|}\right)\right).
\end{align*}
Here, in the first and the second equalities, respectively, we used the fact that (follows from \eqref{d-ImV}), for all $\delta \in |(\Im V_{+},\Im V_{-})|$,
\begin{equation}\label{Relation Im V}
\begin{aligned}
|\Im V_{+}+\Im V_{-}-2\delta|^2+4 |\Im V_{+} -\delta||\Im V_{+} -\delta|=|\Im V_{+} -\Im V_{-}|^2,\\
|\Im V_{+} -\delta|^2+|\Im V_{+} -\delta|^2+ 2|\Im V_{+} -\delta||\Im V_{-} -\delta|=|\Im V_{+} -\Im V_{-}|^2.
\end{aligned}
\end{equation}
Therefore, the norm of $\mathcal{R}_{1,z}$ has the following asymptotic behavior 
\begin{equation}\label{Expansion R1}
\Vert \mathcal{R}_{1,z} \Vert = \frac{2|\Im V_{+}-\Im V_{-}|\tau}{|V_{+}-V_{-}||\Im V_{+}-\delta||\Im V_{-}-\delta|}\left(1+\mathcal{O}\left(\frac{1}{|\tau|}\right)\right).
\end{equation}
Next step, we will show that $\Vert \mathcal{R}_{2,z} \Vert$ is merely a small perturbation compared with $\Vert \mathcal{R}_{1,z} \Vert$ as $\tau\to+\infty$. To do that, the upper bound for the norm of $\mathcal{R}_{2,z}$ will be established by applying the Schur test. Considering the kernel of $\mathcal{R}_{2,z}$ given in \eqref{R2}, we have
\begin{align*}
\int_{\R} |\mathcal{R}_{2,z}(x,y)| \, \dd y =\left\{
\begin{aligned}
&\frac{1}{2|k_{+}(z)|}\frac{2-e^{-\Re k_{+}(z)x}}{\Re k_{+}(z)} &&\text{ for }x>0,\\
&\frac{1}{2|k_{-}(z)|}\frac{2-e^{\Re k_{-}(z)x}}{\Re k_{-}(z)} &&\text{ for }x<0,
\end{aligned}
\right.
\end{align*}
It yields that, for all $z\in \rho(\mathscr{L})$ (hence, $\Re k_{\pm}(z)>0$), 
\[ \sup_{x\in \R} \int_{\R} |\mathcal{R}_{2,z}(x,y)| \, \dd y =\max\left( \frac{1}{|k_{+}(z)| \Re k_{+}(z)}, \frac{1}{|k_{-}(z)| \Re k_{-}(z)}\right).\]
Since $\mathcal{R}_{2,z}(x,y)$ is symmetric for almost everywhere $(x,y)\in \R^2$, thus, by the Schur test, we obtain
\begin{equation}\label{Bound above norm R2}
\Vert \mathcal{R}_{2,z} \Vert \leq  \frac{1}{|k_{+}(z)| \Re k_{+}(z)}+ \frac{1}{|k_{-}(z)| \Re k_{-}(z)}=\frac{2|\Im V_{+}-\Im V_{-}|}{|\Im V_{+}-\delta||\Im V_{-}-\delta|}\left(1+\mathcal{O}\left(\frac{1}{|\tau|}\right)\right).
\end{equation}
Then, by considering \eqref{Expansion R1}, we obtain what we want to show, that is
\begin{align*}
\frac{\Vert \mathcal{R}_{2,z} \Vert }{\Vert \mathcal{R}_{1,z} \Vert}= \mathcal{O}\left(\frac{1}{|\tau|}\right).
\end{align*}
By applying triangular inequality, it yields that
\begin{equation}\label{Resolvent Estimate}
\left\Vert\left(\mathscr{L}-z \right)^{-1}\right\Vert = \Vert \mathcal{R}_{1,z} \Vert\left(1+\mathcal{O}\left(\frac{1}{|\tau|}\right)\right).
\end{equation}
Thus, the asymptotic expansion for $\Vert \left(\mathscr{L}-z\right)^{-1} \Vert$ is deduced from \eqref{Resolvent Estimate} and \eqref{Expansion R1}.
\subsection{(Non)-triviality of the pseudospectrum of $\mathscr{L}$: Proof of Corollary \ref{Cor Trivial}}\label{Subsec Non-trivial}
Let us consider two situations mentioned in this corollary:
\begin{enumerate}[label=\textbf{\textup{(\arabic*)}}]
\item If $\Im V_{+}=\Im V_{-}$, then, thanks to \eqref{Num Range} and \eqref{Spectrum of L}, we have
\[ \overline{\textup{Num}(\mathscr{L})}=[\widehat{V},+\infty) = \sigma(\mathscr{L}),\]
where $\widehat{V}$ defined as in \eqref{V hat}. By using Proposition \ref{Prop Resolvent out} in which $W(\mathscr{L})=\rho(\mathscr{L})$, we obtain, for all $z\in \rho(\mathscr{L})$,
\begin{equation}\label{Norm is 1/d}
\Vert (\mathscr{L}-z)^{-1} \Vert = \frac{1}{\textup{dist}(z,\sigma(\mathscr{L}))}.
\end{equation}
From the definition of pseudospectrum \eqref{Def Pseuspectrum 1}, we deduce that the $\varepsilon$-pseudospectrum of $\mathscr{L}$ is exact the $\varepsilon$-neighbourhood of the spectrum $\sigma(\mathscr{L})$. In fact, the formula \eqref{Norm is 1/d} is true for general unbounded normal operator $T$ (in this case, $\mathscr{L}$ is a normal operator, by Proposition \ref{Prop Property}), indeed, since $T$ is normal, its resolvent is a bounded normal operator (see \cite[Proposition 3.26(v)]{Schmudgen12}), therefore, the norm of the resolvent is equal to its spectral radius \cite[Proposition 3.27]{Cheverry-Raymond21}:
 \begin{align*}
 \Vert (T-z)^{-1} \Vert= &\sup \left\{ |\lambda|: \lambda \in \sigma((T-z)^{-1})\right\}=\sup \left\{ \frac{1}{|\eta-z|} :\eta \in \sigma(T)\right\}\\
 =&\frac{1}{\inf\{|\eta-z|: \eta \in \sigma(T) \}}=\frac{1}{\textup{dist}(z,\sigma(T))}.
 \end{align*}
 Here, in the second equality, we have used the fact that $\lambda \in \sigma((T-z)^{-1})\setminus \{0\}$ if and only if $\lambda=\frac{1}{\eta-z}$ for some $\eta\in \sigma(T)$.
\item If $\Im V_{+}\neq \Im V_{-}$,  Theorem \ref{Theo Norm of resolvent} implies that, for all $\varepsilon'\in (0,1)$, there exists $M>0$ such that, for all $\Re z>M$ and for all $\Im z \in |(\Im V_{+}, \Im V_{-})|$, we have
\[ \left\Vert (\mathscr{L}-z)^{-1}\right\Vert \geq (1-\varepsilon')\frac{2|\Im V_{+}-\Im V_{-}|}{|V_{+}-V_{-}|} \frac{\Re z}{|\Im V_{+}-\Im z||\Im V_{-}-\Im z|}. \]
Thus, for any $\varepsilon>0$, if we consider
\[ \Re z>\frac{1}{\varepsilon(1-\varepsilon')} \frac{|V_{+}-V_{-}||\Im V_{+}-\Im z||\Im V_{-}-\Im z|}{2|\Im V_{+}-\Im V_{-}|}\]
then $\left\Vert (\mathscr{L}-z)^{-1}\right\Vert >\frac{1}{\varepsilon}$ and the conclusion in this case follows.
\end{enumerate}
In order to see the non-triviality of the pseudospectrum of $\mathscr{L}$ when $\Im V_{+}\neq \Im V_{-}$, we just need to take $\varepsilon$ sufficiently small and consider $z_{\varepsilon}$ on the line $\Im z = \frac{\Im V_{+}+\Im V_{-}}{2}$ whose real part is large enough such that $z_{\varepsilon}\in \sigma_{\varepsilon}(\mathscr{L})$, it is easy to see that $z_{\varepsilon}$ always stay away from the spectrum at distance $\frac{1}{2}|\Im V_{+}- \Im V_{-}|$.
\subsection{Accurate pseudomode for $\mathscr{L}$: Proof of Theorem \ref{Theo Pseudomode}}\label{Subsec Pseudomode}
The aim of this part is to construct the pseudomode for the operator $\mathscr{L}$. Let us fix $z\in \rho(\mathscr{L})$, we construct the Ansatz in the form
\[ \Psi_{z}(x)= \Psi_{1,z}(x) + \Psi_{2,z}(x), \]
where
\begin{equation*}
\begin{aligned}
\Psi_{1,z}(x) &= n_{1}(z) e^{k_{-}(z)x}\textbf{\textup{1}}_{\R_{-}}(x)+p_{1}(z) e^{-k_{+}(z)x}\textbf{\textup{1}}_{\R_{+}}(x),\\
\Psi_{2,z}(x) &= n_{2}(z) e^{\overline{k_{-}(z)}x}\textbf{\textup{1}}_{\R_{-}}(x)+p_{2}(z) e^{-\overline{k_{+}(z)}x}\textbf{\textup{1}}_{\R_{+}}(x),
\end{aligned}
\end{equation*}
in which $n_{1}(z),n_{2}(z),p_{1}(z),p_{2}(z)$ are complex numbers to be determined later. Here, we notice that $\Psi_{1,z}$ and $\Psi_{2,z}$ satisfy, respectively, for all $x\in \R \setminus \{0\}$,
\begin{equation}\label{Psi12}
\left(-\frac{\dd^2}{\dd x^2} +V(x)-z \right) \Psi_{1,z}(x)=0, \qquad \left(-\frac{\dd^2}{\dd x^2} +\overline{V(x)}-\overline{z} \right) \Psi_{2,z}(x)=0.
\end{equation}
In order $\Psi_{z}$ to belong to $ H^2(\R)$, the domain of $\mathscr{L}$, it is necessary to satisfy the conditions $\Psi_{z}(0^{+})=\Psi_{z}(0^{-})$ and $\Psi_{z}'(0^{+})=\Psi_{z}'(0^{-})$. These conditions read
\begin{equation}\label{Continuity Cond}
\left\{
\begin{aligned}
&n_{1}(z)+n_{2}(z)=p_{1}(z)+p_{2}(z),\\
&k_{-}(z)n_{1}(z)+\overline{k_{-}(z)}n_{2}(z)=-k_{+}(z)p_{1}(z)-\overline{k_{+}(z)}p_{2}(z),
\end{aligned}
\right.
\end{equation}
This allows us to compute $n_{1}(z)$ and $p_{1}(z)$ in terms of $n_{2}(z)$ and $p_{2}(z)$ as follows
\begin{equation}\label{System n1p1}
\left\{ \begin{aligned}
&n_{1}(z)=-\frac{k_{+}(z)+\overline{k_{- }(z)}}{k_{+}(z)+k_{-}(z)}n_{2}(z)+\frac{k_{+}(z)-\overline{k_{+}(z)}}{k_{+}(z)+k_{-}(z)}p_{2}(z),\\
&p_{1}(z)=\frac{k_{-}(z)-\overline{k_{-}(z)}}{k_{+}(z)+k_{-}(z)}n_{2}(z)-\frac{\overline{k_{+}(z)}+k_{-}(z)}{k_{+}(z)+k_{-}(z)}p_{2}(z).
\end{aligned}
\right.
\end{equation}
In the following, we will show that this Ansatz $\Psi_{z}$ belongs to $H^2(\R)$ and $n_{2}$ and $p_{2}$ can be calculated through optimizing (in fact, minimizing) the quotient $\frac{\Vert (\mathscr{L}-z)\Psi_{z} \Vert}{\Vert \Psi_{z} \Vert}$. Indeed, assume that $\Psi_{z}$ satisfies conditions in \eqref{Continuity Cond}, since $z\in \rho(\mathscr{L})$, we have $\Re k_{\pm}(z)= \Re \overline{k_{\pm}(z)}>0$, it is easy to see that $\Psi_{z}\in L^2(\R)$. Let $\varphi \in C_{c}^{1}(\R)$, by integration by parts on $\R_{-}$ and on $\R_{+}$ and using the first equality in \eqref{Continuity Cond}, we get
\begin{align*}
\int_{\R} \Psi_{z}(x) \varphi'(x)\, \dd x= &(n_{1}(z)+n_{2}(z)-p_{1}(z)-p_{2}(z))\varphi(0)\\
&- \int_{-\infty}^{0} \left(n_{1}(z)k_{-}(z) e^{k_{-}(z)x}+n_{2}(z) \overline{k_{-}(z)} e^{\overline{k_{-}(z)}x}\right)\varphi(x)\,\dd x\\
&+ \int_{0}^{+\infty} \left(p_{1}(z)k_{+}(z) e^{-k_{+}(z)x}+p_{2}(z) \overline{k_{-}(z)} e^{-\overline{k_{+}(z)}x}\right)\varphi(x)\,\dd x\\
=&- \int_{-\infty}^{0} \left(n_{1}(z)k_{-}(z) e^{k_{-}(z)x}+n_{2}(z) \overline{k_{-}(z)} e^{\overline{k_{-}(z)}x}\right)\varphi(x)\,\dd x\\
&+ \int_{0}^{+\infty} \left(p_{1}(z)k_{+}(z) e^{-k_{+}(z)x}+p_{2}(z) \overline{k_{+}(z)} e^{-\overline{k_{+}(z)}x}\right)\varphi(x)\,\dd x.
\end{align*}
Therefore, the distributional derivative of $\Psi_{z}$ is given by
\begin{equation}\label{Psi'}
\begin{aligned}
\Psi_{z}'(x) = &\left( n_{1}(z)k_{-}(z) e^{k_{-}(z)x}+n_{2}(z)\overline{k_{-}(z)} e^{\overline{k_{-}(z)}x}\right)\textbf{\textup{1}}_{\R_{-}}(x)\\
&+\left(-p_{1}(z) k_{+}(z) e^{-k_{+}(z)x}-p_{2}(z)\overline{k_{+}(z)} e^{-\overline{k_{+}(z)}x}\right)\textbf{\textup{1}}_{\R_{+}}(x),
\end{aligned}
\end{equation}
which also belongs to $L^2(\R)$. In other words, $\Psi_{z}\in H^1(\R)$. In the same manner, by using the second equality in \eqref{Continuity Cond}, we can obtain the  distributional second derivative of $\Psi_{z}$ as follows
\begin{align*}
\Psi_{z}''(x) = &\left( n_{1}(z)k_{-}(z)^2 e^{k_{-}(z)x}+n_{2}(z)\overline{k_{-}(z)}^2 e^{\overline{k_{-}(z)}x}\right)\textbf{\textup{1}}_{\R_{-}}(x)\\
&+\left(p_{1}(z) k_{+}(z)^2 e^{-k_{+}(z)x}+p_{2}(z)\overline{k_{+}(z)}^2 e^{-\overline{k_{+}(z)}x}\right)\textbf{\textup{1}}_{\R_{+}}(x),
\end{align*}
and it belongs to $L^2(\R)$. In conclusion, $\Psi_{z} \in H^2(\R)$. 

Now, we will send $z$ to infinity in the numerical range to get the asymptotic behavior of the quotient $\frac{\Vert (\mathscr{L}-z)\Psi_{z} \Vert}{\Vert \Psi_{z} \Vert}$. Thanks to Lemma \ref{Lem Asymptotic} and \eqref{d-ImV}, we obtain the following expressions
\begin{equation}\label{Asymtotic 2}
\begin{aligned}
&k_{+}(z) + \overline{k_{-}(z)}=\sgn(\Im V_{+}-\Im V_{-})2 i\sqrt{\tau}+\mathcal{O}\left(\frac{1}{|\tau|^{1/2}}\right),\\
&k_{+}(z) - \overline{k_{+}(z)}=\sgn(\Im V_{+}-\Im V_{-})2 i\sqrt{\tau}+\mathcal{O}\left(\frac{1}{|\tau|^{1/2}}\right),\\
&k_{-}(z) - \overline{k_{-}(z)}=-\sgn(\Im V_{+}-\Im V_{-})2 i\sqrt{\tau}+\mathcal{O}\left(\frac{1}{|\tau|^{1/2}}\right),\\
&\overline{k_{+}(z)}+k_{-}(z) =-\sgn(\Im V_{+}-\Im V_{-})2 i\sqrt{\tau}+\mathcal{O}\left(\frac{1}{|\tau|^{1/2}}\right).
\end{aligned}
\end{equation}
Now, we impose the following conditions on the coefficients $n_{2}(z)$ and $p_{2}(z)$
\begin{equation}\label{Ass 1}
|n_{2}(z)|=\mathcal{O}(1), \qquad |p_{2}(z)|=\mathcal{O}(1),\qquad 1= \mathcal{O}(|n_{2}(z)-p_{2}(z)|),
\end{equation}
as $\tau\to+\infty$ and uniformly for all $\delta \in |(\Im V_{+}, \Im V_{-})|$. Then, it follows from the first equation in \eqref{System n1p1}, \eqref{Asymtotic 2} and \eqref{Asymptotics k} that
\begin{align}\label{n1}
|n_{1}(z)|^2= &\left(\left\vert k_{+}(z)+\overline{k_{-}(z)} \right\vert^2 |n_{2}(z)|^2+\left\vert k_{+}(z)-\overline{k_{+}(z)}\right\vert^2 |p_{2}(z)|^2\right.\nonumber\\
&\left.-2 \Re \left( (k_{+}(z)+\overline{k_{-}(z)})\overline{(k_{+}(z)-\overline{k_{+}(z)})}n_{2}(z)\overline{p_{2}(z)}\right)\right)\div |k_{+}(z)+k_{-}(z) |^2\nonumber\\
=&\frac{4\tau|n_{2}(z)-p_{2}(z)|^2\left(1+\mathcal{O}\left(\frac{1}{|\tau|}\right)\right)}{\frac{|V_{+}-V_{-}|^2}{4\tau}\left(1+\mathcal{O}\left(\frac{1}{|\tau|}\right)\right)},\nonumber\\
=&\frac{16\tau^2|n_{2}(z)-p_{2}(z)|^2}{|V_{+}-V_{-}|^2}\left(1+\mathcal{O}\left(\frac{1}{|\tau|}\right)\right).
\end{align}
Similarly, we also obtain the asymptotic expression for $|p_{1}(z)|^2$,
\begin{equation}\label{p1}
|p_{1}(z)|^2 = \frac{16\tau^2|n_{2}(z)-p_{2}(z)|^2}{|V_{+}-V_{-}|^2}\left(1+\mathcal{O}\left(\frac{1}{|\tau|}\right)\right).
\end{equation}
With $\Psi_{1,z}$ and $\Psi_{2,z}$ defined as in \eqref{Psi12}, we get
\begin{equation}\label{Norm Psi1}
\Vert \Psi_{1,z} \Vert^2 = \frac{|n_{1}(z)|^2}{2\Re k_{-}(z)} +\frac{|p_{1}(z)|^2}{2\Re k_{+}(z)},\qquad\Vert \Psi_{2,z} \Vert^2 = \frac{|n_{2}(z)|^2}{2\Re k_{-}(z)} +\frac{|p_{2}(z)|^2}{2\Re k_{+}(z)}.
\end{equation}
By using the assumption \eqref{Ass 1} for $n_{2}(z), p_{2}(z)$ and \eqref{n1}, \eqref{p1}, $\Vert \Psi_{z,2} \Vert$ can be shown to be a small perturbation compared with $\Vert \Psi_{z,1} \Vert$:
\begin{align*}
\frac{\Vert \Psi_{2,z} \Vert^2}{\Vert \Psi_{1,z} \Vert^2 }\leq \frac{|n_{2}(z)|^2}{|n_{1}(z)|^2}+\frac{|p_{2}(z)|^2}{|p_{1}(z)|^2} =\mathcal{O}\left(\frac{1}{|\tau|^2}\right).
\end{align*}
Thanks to the triangle inequality, it yields that
\begin{equation}\label{Norm Psi}
\Vert \Psi_{z} \Vert = \Vert \Psi_{1,z} \Vert \left(1 +\mathcal{O}\left(\frac{1}{|\tau|}\right) \right).
\end{equation}
From \eqref{Psi12}, it yields that, for all $x\in \R\setminus\{0\}$, 
\[ \left(\mathscr{L}-z\right)\Psi_{z}(x)=2i(\Im V(x)-\Im z) \Psi_{2,z}(x), \]
and the square of its norm is given by
\begin{align*}
\Vert (\mathscr{L}-z)\Psi_{z} \Vert^2=\frac{2|n_{2}(z)|^2}{\Re k_{-}(z)}|\Im V_{-}-\delta|^2 +\frac{2|p_{2}(z)|^2}{\Re k_{+}(z)}|\Im V_{+}-\delta|^2 .
\end{align*}
Applying \eqref{Norm Psi} and \eqref{Norm Psi1}, we have
\begin{align*}
\frac{\Vert (\mathscr{L}-z)\Psi_{z} \Vert^2}{\Vert \Psi_{z} \Vert^2}=& \frac{\frac{2|\Im V_{-}-\delta|^2}{\Re k_{-}(z)}|n_{2}(z)|^2 +\frac{2|\Im V_{+}-\delta| ^2}{\Re k_{+}(z)}|p_{2}(z)|^2}{\frac{|n_{1}(z)|^2}{2\Re k_{-}(z)}+\frac{|p_{1}(z)|^2}{2\Re k_{+}(z)}}\left(1 +\mathcal{O}\left(\frac{1}{|\tau|}\right) \right)\\
=& 4\frac{\Re k_{+}(z)|\Im V_{-}-\delta|^2|n_{2}(z)|^2+\Re k_{-}(z)|\Im V_{+}-\delta|^2|p_{2}(z)|^2}{\Re k_{+}(z)|n_{1}(z)|^2+\Re k_{-}(z)|p_{1}(z)|^2}\left(1 +\mathcal{O}\left(\frac{1}{|\tau|}\right) \right).
\end{align*}
Employing Lemma \ref{Lem Asymptotic} for $\Re k_{\pm}(z)$, \eqref{Ass 1}, \eqref{n1},  and \eqref{p1}, we get
\begin{align*}
\Re k_{+}(z)|\Im V_{-}-\delta|^2|n_{2}(z)|^2=&\frac{|\Im V_{+}-\delta||\Im V_{-}-\delta|^2}{2\sqrt{\tau}}|n_{2}(z)|^2\left(1+\mathcal{O}\left(\frac{1}{|\tau|}\right)\right),\\
\Re k_{-}(z)|\Im V_{+}-\delta|^2|p_{2}(z)|^2=&\frac{|\Im V_{+}-\delta|^2|\Im V_{-}-\delta|}{2\sqrt{\tau}}|p_{2}(z)|^2\left(1+\mathcal{O}\left(\frac{1}{|\tau|}\right)\right),\\
\Re k_{+}(z)|n_{1}(z)|^2=&\frac{8\tau^{3/2}|\Im V_{+}-\delta|}{|V_{+}-V_{-}|^2}|n_{2}(z)-p_{2}(z)|^2\left(1+\mathcal{O}\left(\frac{1}{|\tau|}\right)\right),\\
\Re k_{-}(z)|p_{1}(z)|^2=&\frac{8\tau^{3/2}|\Im V_{-}-\delta|}{|V_{+}-V_{-}|^2}|n_{2}(z)-p_{2}(z)|^2\left(1+\mathcal{O}\left(\frac{1}{|\tau|}\right)\right).
\end{align*}
Then, by using \eqref{Equality for delta}, the quotient has the asymptotic behavior as follows
\begin{equation}\label{quotient}
\begin{aligned}
\frac{\Vert (\mathscr{L}-z)\Psi_{z} \Vert^2}{\Vert \Psi_{z} \Vert^2}=&\frac{|V_{+}-V_{-}|^2}{4|\Im V_{+}-\Im V_{-}|} \frac{|\Im V_{+}-\delta||\Im V_{-}-\delta|}{\tau^2} \\
& \times \frac{|\Im V_{-}-\delta| |n_{2}(z)|^2+|\Im V_{+}-\delta||p_{2}(z)|^2}{|n_{2}(z)-p_{2}(z)|^2}\left(1+\mathcal{O}\left(\frac{1}{|\tau|}\right)\right).
\end{aligned}
\end{equation}
Since when $z\in \rho(\mathscr{L})$, the inverse of the resolvent's norm can be expressed as follows
\[ \Vert \left( \mathscr{L}-z\right)^{-1} \Vert^{-1} = \inf_{\Psi \in H^2(\R)\setminus \{0\}} \frac{\Vert (\mathscr{L}-z)\Psi \Vert}{\Vert \Psi \Vert}.\]
From this expression, $n_{2}(z)$ and $p_{2}(z)$ will be chosen in order to minimize the function
\[F(n_{2}(z),p_{2}(z))\coloneqq\frac{|\Im V_{-}-\delta| |n_{2}(z)|^2+|\Im V_{+}-\delta||p_{2}(z)|^2}{|n_{2}(z)-p_{2}(z)|^2}. \]
From \eqref{System n1p1}, in order to have $\Psi_{z} \neq 0$, at least one of two numbers $n_{2}(z)$ or $p_{2}(z)$ must be non-zero, we assume that $p_{2}\neq  0$. By dividing both numerator and denominator of $F(n_{2}(z),p_{2}(z))$ by $|p_{2}(z)|^2$, our problem turns into searching the infimum of a one-variable function 
\[ f(x)\coloneqq\frac{|\Im V_{-}-\delta| x^2+|\Im V_{+}-\delta|}{(x-1)^2}.\]
For simplification, we search its infimum on $\R$: its derivative is
\[ f'(x)=-2\frac{|\Im V_{-}-\delta| x+|\Im V_{+}-\delta|}{(x-1)^3},\]
By investigation function $f(x)$, we will see that it attains its global minimum at the point $x_{0}=-\frac{|\Im V_{+}-\delta|}{|\Im V_{-}-\delta|}$ and using \eqref{Equality for delta}, we can calculate the minimum value of $f(x)$
\[ f(x_{0})= \frac{|\Im V_{+}-\delta||\Im V_{-}-\delta|}{|\Im V_{+}-\Im V_{-}|}.\]
Therefore, we need to choose $n_{2}(z)$ and $p_{2}(z)$ such that $\frac{n_{2}(z)}{p_{2}(z)}=-\frac{|\Im V_{+}-\delta|}{|\Im V_{-}-\delta|}$ and satisfying the assumption \eqref{Ass 1}. We have $|n_{2}(z)-p_{2}(z)|= \frac{|\Im V_{+}-\Im V_{-}|}{|\Im V_{-}-\delta|} |p_{2}(z)|$, thus, in order to have $1\lesssim |n_{2}(z)-p_{2}(z)| \lesssim 1$, we can choose $p_{2}(z)$ in the form $p_{2}(z)=C |\Im V_{-}-\delta|$ where $C$ is a non-zero complex constant and thus, $n_{2}(z)=-C |\Im V_{+}-\delta|$. However, by replacing these numbers $n_{2}(z)$, $p_{2}(z)$ into two first equations in \eqref{Constants} to have $n_{1}(z),p_{1}(z)$, we see that the constant $C$ merely plays the role of a (normalizing) multiplying constant. Therefore, we can choose $C=1$ and it completes the proof of Theorem \ref{Theo Pseudomode}.
\section{Pseudospectrum of the complex point interaction}\label{Section Complex point interaction}
\subsection{Stability of essential spectra under point interaction: Proof of Lemma \ref{Lem Essential Spectra delta}}
We prove this lemma by finite extensional method, cf. \cite[Sec. IX.4]{Edmunds-Evans18}. This method says that if $\mathscr{L}$ is a closed $n$-dimensional extension of a closed, densely defined operator $\mathscr{T}$, \emph{i.e.}, $\mathscr{T}\subset \mathscr{L}$ and there is an $n$-dimensional subspace $F$ of $\textup{Dom}(\mathscr{L})$ such that $\textup{Dom}(\mathscr{L})=\textup{Dom}(\mathscr{T})\dotplus F$ (\emph{i.e.}, $\textup{Dom}(\mathscr{L})$ is a direct sum of two linear vector spaces $\textup{Dom}(\mathscr{T})$ and $F$), then the third three essential spectra of $\mathscr{L}$ and $\mathscr{T}$ are identical: $\sigma_{\textup{ek}}(\mathscr{L})=\sigma_{\textup{ek}}(\mathscr{T})$ for $\textup{k}\in [[1,3]]$. We can not apply directly this method to two operators $\mathscr{L}$ and $\mathscr{L}_{\alpha}$, since their domains are not extensions of each other. We need to find a mediator operator between two operators $\mathscr{L}$ and $\mathscr{L}_{\alpha}$ so that we obtain the transitive property between essential spectra of these operators. Let us make the previous sentence to be clear: We consider an operator $\mathscr{T}$ which is a common restriction of $\mathscr{L}$ and $\mathscr{L}_{\alpha}$, defined by
\begin{align*}
\textup{Dom}(\mathscr{T})&=\{u\in H^{1}(\R)\cap H^2(\R): u(0)=u'(0)=0 \},\\ 
 \mathscr{T}u&=-u''+V(x)u,\qquad \forall u\in  \textup{Dom}(\mathscr{T}).
\end{align*}
and we will show that
\begin{enumerate}[label=\textup{\textbf{(\alph*)}}]
\item \label{Closedness} $\mathscr{T}$ is a closed densely defined operator;
\item \label{Finite exten 1} $\mathscr{L}$ is a $2$-dimensional extension of $\mathscr{T}$;
\item \label{Finite exten 2} $\mathscr{L}_{\alpha}$ is a $2$-dimensional extension of $\mathscr{T}$.
\end{enumerate}
Let us start with \ref{Closedness}. It is not hard to prove that $\textup{Dom}(\mathscr{T})$ is dense in $L^2(\R)$. Indeed, we take $f\in L^2(\R)$ and we consider a sequence of smooth function $\left(\xi_{n}\right)_{n\geq 1}$ such that $\textup{supp }\xi_{n} \subset \{x\in \R: 0<|x|<n+1\}$ and $\xi_{n}(x)=1$ for all $x\in \R$ satisfying $\frac{1}{n}\leq |x| \leq n$, then $\xi_{n} f \in \textup{Dom}(\mathscr{T})$ and $\Vert \xi_{n} f - f \Vert \xrightarrow{n\to +\infty} 0$ by the dominated convergence theorem. To prove that $\mathscr{T}$ is closed, we consider the graph norm $\Vert u \Vert_{\mathscr{T}}\coloneqq \sqrt{\Vert \mathscr{T}u \Vert^2+\Vert u \Vert^2 }$ defined for all $u\in \textup{Dom}(\mathscr{T})$, this norm is equivalent to the norm in $H^2(\R)$. Therefore, take any Cauchy sequence $u_{n}$ in $\textup{Dom}(\mathscr{T})$ under the graph norm, this sequence will converge to a function $u$ in  $H^2(\R)$. By the Sobolev embedding of $H^2(\R)$ into $L^{\infty}(\R,\C^2)$, we have $u(0)=u'(0)=0$. Thus, we have shown that $\left(\textup{Dom}(\mathscr{T}), \Vert \cdot \Vert_{\mathscr{T}}\right)$ is complete or equivalently, $\mathscr{T}$ is closed.

Next, we prove \ref{Finite exten 1}. Let $\phi \in \C_{c}^{\infty}(\R)$ such that $\phi(x)=1$ in some neighbourhood of zero. Take $u\in \textup{Dom}(\mathscr{L})$, then, we can show that
\[ u(x)-\left[u(0)\phi(x)+u'(0)x\phi(x)\right]\in \textup{Dom}(\mathscr{T}).\]
It means that 
\[ \textup{Dom}(\mathscr{L})\subset \textup{Dom}(\mathscr{T}) +F,\qquad F:=\textup{span}\{\phi(x), x\phi(x)\}.\]
Note that $\dim F=2$ (since $\phi(x)$ and $x\phi(x)$ are linear independent). Since both $\textup{Dom}(\mathscr{T})$ and $F$ are subspaces of $\textup{Dom}(\mathscr{L})$, therefore, we have the direction $\textup{Dom}(\mathscr{L})\supset \textup{Dom}(\mathscr{T}) +F$. Furthermore, we have $\textup{Dom}(\mathscr{T})\cap F=\{0\}$, thus it implies the statement \ref{Finite exten 1}. 

To prove \ref{Finite exten 2}, we modify functions in $F$ such that we have functions belonging to $\textup{Dom}(\mathscr{L}_{\alpha})$. Let us consider a function $\phi_{\alpha}(x)$ defined on $\R$ as follows
\begin{align*}
\phi_{\alpha}(x):= \left(\frac{\alpha}{2}|x|+1\right)\phi(x).
\end{align*}
It is clear that $\phi_{\alpha}\in H^1(\R)\cap H^2(\R\setminus\{0\})$ and  $\phi_{\alpha}(0)=1$, $\phi_{\alpha}'(0^{+})=\frac{\alpha}{2}$ and $\phi_{\alpha}'(0^{-})=-\frac{\alpha}{2}$. This implies that $\phi_{\alpha} \in \textup{Dom}(\mathscr{L}_{\alpha})$. It is also clear that $x\Phi_{a}(x)\in \textup{Dom}(\mathscr{L}_{\alpha})$, therefore, we have 
$$G:=\textup{span}\{\phi_{\alpha}(x), x\phi_{\alpha}(x)\}\subset \textup{Dom}(\mathscr{L}_{\alpha}).$$ Take $u \in \textup{Dom}(\mathscr{L}_{\alpha})$, we can verify that
\[v(x)\coloneqq u(x)-\left[u(0)\phi_{\alpha}(x)+\frac{u'(0^{+})+u'(0^{-})}{2} x\phi_{\alpha}(x)\right]\in \textup{Dom}(\mathscr{T}).\]
Indeed, by direct computation, we have $v(0)=u(0)-u(0)=0$ and
\begin{align*}
v'(0^{+})&= u'(0^{+})-\frac{\alpha u(0)}{2}- \frac{u'(0^{+})+u'(0^{-})}{2} =\frac{u'(0^{+})-u'(0^{-})-\alpha u(0)}{2}=0,\\
v'(0^{-})&=u'(0^{-})+\frac{\alpha u(0)}{2}- \frac{u'(0^{+})+u'(0^{-})}{2}=-\frac{u'(0^{+})-u'(0^{-})-\alpha u(0)}{2}=0,
\end{align*}
where the last equalities comes from the jump condition of $u$ in $\textup{Dom}(\mathscr{L}_{\alpha})$. In other words, we have shown that $\textup{Dom}(\mathscr{L}_{\alpha})\subset \textup{Dom}(\mathscr{T})+G$. Therefore, we obtain the equality $\textup{Dom}(\mathscr{L}_{\alpha})= \textup{Dom}(\mathscr{T})+G$ (since both $\textup{Dom}(\mathscr{T})$ and $G$ are subspaces of $\textup{Dom}(\mathscr{L}_{\alpha})$). It is easy to check that $\textup{Dom}(\mathscr{T})\cap G=\emptyset$, thus the statement \ref{Finite exten 2} follows since $\dim G=2$.

Thanks to \ref{Closedness}, \ref{Finite exten 1}, \ref{Finite exten 2}, we can apply \cite[Corollary IX.4.2]{Edmunds-Evans18} to obtain
\[ \sigma_{\textup{ek}}(\mathscr{L})=\sigma_{\textup{ek}}(\mathscr{T})=\sigma_{\textup{ek}}(\mathscr{L}_{\alpha}) \qquad \text{ for }k\in \{1,2,3\}.\]
The statement on essential spectra of $\mathscr{L}_{\alpha}$ in other type, \emph{i.e.}, $\textup{k}\in \{4,5\}$,  is followed from \cite[Proposition 5.4.4]{Krejcirik-Siegl15} because $\C \setminus \sigma_{\textup{e1}}(\mathscr{L}_{\alpha})$ is connected.
\subsection{Existence of the eigenvalue: Proof of Theorem \ref{Theo Spectrum L delta}}
Since $\mathscr{L}_{\alpha}$ is $\mathcal{T}$-self-adjoint (see Proposition \ref{Prop Property delta}), its residual spectrum is empty, thus its spectrum is decomposed as a union of the point spectrum and the continuous spectrum. Thanks to Lemma \ref{Lem Essential Spectra delta}, the set $[V_{+},+\infty)\cup [V_{-},+\infty)$ are essential spectrum of $\mathscr{L}_{\alpha}$. Therefore, take $z\in \C \setminus [V_{+},+\infty)\cup [V_{-},+\infty)$, then $\Ran(\mathscr{L}_{\alpha}-z)$ is closed and this implies that $z \in \C \setminus \sigma_{\textup{c}}(\mathscr{L}_{\alpha})$ (because if $z\in \sigma_{\textup{c}}(\mathscr{L}_{\alpha})$, we have a contradiction with the definition of the continuous spectrum). In other words, we have shown that
\begin{equation}\label{Continous Spec delta 1}
\sigma_{\textup{c}}(\mathscr{L}_{\alpha})\subset [V_{+},+\infty)\cup [V_{-},+\infty).
\end{equation}
Since the action of $\mathscr{L}_{\alpha}$ and $\mathscr{L}$ are the same, we can use the argument in the proof of Proposition \ref{Prop Resolvent} to solve the eigenvalue equation $(\mathscr{L}_{\alpha}-z)u=0$. From \eqref{u+-} with a notice that $f=0$ in the argument, the general solution restricted on $\R_{\pm}$, denoted as $u_{\pm}$, reads as 
\begin{equation}\label{Eigenfunction}
 u_{\pm}(x)=A_{\pm} e^{k_{\pm}(z)x} + B_{\pm}e^{-k_{\pm}(z)x},\qquad \text{for }\pm x >0.
\end{equation}
In the same manner as in the proof of Theorem \ref{Theo Spectrum of L} (subsection \ref{Subsec Spectrum L}), we can easily show that no point in $[V_{+},+\infty)\cup [V_{-},+\infty)$ can be the eigenvalue of $\mathscr{L}_{\alpha}$ and thus, 
\[ [V_{+},+\infty)\cup [V_{-},+\infty) \subset \sigma_{\textup{c}}( \mathscr{L}_{\alpha}).\]
Combining this inclusion with \eqref{Continous Spec delta 1}, we conclude \eqref{Continuous Spec delta}. The fact that the point spectrum is the discrete spectrum comes from \cite[Proposition 5.4.3]{Krejcirik-Siegl15} which states that $\sigma_{\textup{dis}}(\mathscr{L}_{\alpha})=\sigma(\mathscr{L}_{\alpha})\setminus \sigma_{\textup{e5}}(\mathscr{L}_{\alpha})$. 

Now, we look for the eigenvalue of $\mathscr{L}_{\alpha}$ in the set $\C \setminus \left([V_{+},+\infty)\cup [V_{-},+\infty) \right)$. Take $z \in \C \setminus \left([V_{+},+\infty)\cup [V_{-},+\infty) \right)$ and assume that $(z,u)$ is the eigenpair of $\mathscr{L}_{\alpha}$. From the decaying of $u$ at $\pm \infty$, with reasons as in \eqref{A+} and \eqref{B-}, it implies that $A_{+}=0$ and $B_{-}=0$. Replacing these conditions into \eqref{Eigenfunction}, we get
\begin{equation}\label{Eigenfunction delta}
u(x) =  \left\{\begin{aligned}
&B_{+} e^{-k_{+}(z)x}, \qquad &&\text{ for }x>0,\\
&A_{-} e^{k_{-}(z)x}, \qquad &&\text{ for }x<0.
\end{aligned} \right.
\end{equation}
In order $u$ belong to $\textup{Dom}(\mathscr{L}_{\alpha})$, it is necessary that $A_{-}$ and $B_{+}$ are non-trivial solution of the system
\[ \left\{ \begin{aligned}
&B_{+}-A_{-}=0,\\
&-k_{+}(z)B_{+} -(k_{-}(z)+\alpha) A_{-} = 0.
\end{aligned}\right.\]
Vice versa, if there exists non zero constants $A_{-}$ and $B_{+}$ satisfies the above system, it is easy to check that $u(x)$ given by \eqref{Eigenfunction delta} is indeed the eigenfunction associated with the eigenvalue $z \in \C \setminus \left([V_{+},+\infty)\cup [V_{-},+\infty) \right)$. This is equivalent to the fact that $z$ is eigenvalue if and only if $z$ satisfies the following algebraic equation depending on the parameter $\alpha \in \C$
\begin{equation}\label{Eig Equation}
k_{+}(z)+k_{-}(z) = -\alpha,
\end{equation}
and $z \notin \left([V_{+},+\infty)\cup [V_{-},+\infty) \right)$. We can assume that $\alpha\neq 0$ (since when $\alpha=0$, then $\mathscr{L}_{0}=\mathscr{L}$, we already know that there is not eigenvalue in this case). Using the fact that
\[ k_{+}(z)^2 - k_{-}(z)^2 = V_{+}-V_{-},\]
from \eqref{Eig Equation}, we implies that
\begin{equation}\label{Eq 1}
k_{+}(z)-k_{-}(z) = -\frac{V_{+}-V_{-}}{\alpha}.
\end{equation}
Thanks to \eqref{Eig Equation} and \eqref{Eq 1}, we get
\begin{equation}\label{Eq 2}
\sqrt{V_{+}-z}=-\frac{\alpha}{2}-\frac{V_{+}-V_{-}}{2\alpha} \text{ and } \sqrt{V_{-}-z}=-\frac{\alpha}{2}+\frac{V_{+}-V_{-}}{2\alpha}.
\end{equation}
By squaring one of two equations in \eqref{Eq 2}, it yields that
\begin{equation}\label{Sol of eigenvalue}
z=z(\alpha)=\frac{V_{+}+V_{-}}{2}-\frac{(V_{+}-V_{-})^2}{4\alpha^2}-\frac{\alpha^2}{4}.
\end{equation}
Thus, we have shown that if the equation \eqref{Eig Equation} has solution, then its solution is necessary in the form \eqref{Sol of eigenvalue}.
It means that, let $\alpha\in \C$, the point $z(\alpha)$ is the eigenvalue of $\mathscr{L}_{\alpha}$ if and only if $\alpha$ satisfies
\begin{equation}\label{Eig Equation alpha}
\left\{
\begin{aligned}
&\sqrt{V_{+}-z(\alpha)}+\sqrt{V_{-}-z(\alpha)}=-\alpha,\\
&z(\alpha)\notin [V_{+},+\infty)\cup [V_{-},+\infty).
\end{aligned}
\right.
\end{equation}
Since two equations in \eqref{Eq 2} are implication equations of \eqref{Eig Equation}, then $\alpha$ satisfying \eqref{Eig Equation alpha} if and only if
\begin{equation}\label{Eig Equation alpha 2}
\left\{
\begin{aligned}
&\sqrt{V_{+}-z(\alpha)}=-\frac{\alpha}{2}-\frac{V_{+}-V_{-}}{2\alpha},\\
&\sqrt{V_{-}-z(\alpha)}=-\frac{\alpha}{2}+\frac{V_{+}-V_{-}}{2\alpha},\\
&z(\alpha)\notin [V_{+},+\infty)\cup [V_{-},+\infty).
\end{aligned}
\right.
\end{equation}
We will show that \eqref{Eig Equation alpha 2} is equivalent to
\begin{equation}\label{Eig Equation alpha 3}
\left\{
\begin{aligned}
&\Re \left(-\frac{\alpha}{2}-\frac{V_{+}-V_{-}}{2\alpha}\right)>0,\\
&\Re \left(-\frac{\alpha}{2}+\frac{V_{+}-V_{-}}{2\alpha}\right)>0.\\
\end{aligned}
\right.
\end{equation}
Indeed we prove the statement in two directions
\begin{itemize}
\item \eqref{Eig Equation alpha 2} $\Longrightarrow$ \eqref{Eig Equation alpha 3}: Since $V_{+}-z(\alpha) \notin (-\infty,0]$, then the real part of the square root $\sqrt{V_{+}-z(\alpha)}$ must be positive. The same argument for the real part of the square root $\sqrt{V_{-}-z(\alpha)}$.
\item \eqref{Eig Equation alpha 3} $\Longrightarrow$ \eqref{Eig Equation alpha 2}: Using the fact that, for $w$ is a complex number, $\sqrt{w^2}=w$ if $\Re w>0$, we have
\begin{align*}
\sqrt{V_{+}-z(\alpha)} =\sqrt{\left(-\frac{\alpha}{2}-\frac{V_{+}-V_{-}}{2\alpha}\right)^2}=-\frac{\alpha}{2}-\frac{V_{+}-V_{-}}{2\alpha},\\
\sqrt{V_{-}-z(\alpha)} =\sqrt{\left(-\frac{\alpha}{2}+\frac{V_{+}-V_{-}}{2\alpha}\right)^2}=-\frac{\alpha}{2}+\frac{V_{+}-V_{-}}{2\alpha}.
\end{align*}
If $V_{\pm}-z(\alpha)\in (-\infty,0]$, then $\Re \sqrt{V_{\pm}-z(\alpha)}=0$ and thus, we have a contradiction with \eqref{Eig Equation alpha 3}. Therefore, the last condition in \eqref{Eig Equation alpha 2} is obtained.
\end{itemize}
In summary, we have just shown that the point $z(\alpha)$ is the eigenvalue of $\mathscr{L}_{\alpha}$ if and only if $\alpha$ satisfies conditions in \eqref{Eig Equation alpha 3}. By elementary calculations, the conditions for $\alpha$ in \eqref{Eig Equation alpha 3}  can be read as 
\[ \left\vert \langle V_{+}-V_{-}, \alpha \rangle_{\R^2}\right\vert< -|\alpha|^2\Re \alpha ,\]
which is the description of the set $\Omega$ in the statement of this theorem. The eigenfunction associated with $z(\alpha)$ must be in the form of \eqref{Eigenfunction delta} with $A_{-}=B_{+}$, then, thanks to \eqref{Eig Equation alpha 2}, the eigenspace associated with $z(\alpha)$ is expressed as 
\[\textup{Ker}(\mathscr{L}_{\alpha}-z(\alpha)) =\{C u_{\alpha}: C\in \C \},\]
where $u_{\alpha}$ is given by \eqref{Eigenfunction delta 0}.

\subsection{Asymptotic resolvent norm of $\mathscr{L}_{\alpha}$: Proof of Theorem \ref{Theo Norm of resolvent delta}}.
Since the action of $\mathscr{L}$ and $\mathscr{L}_{\alpha}$ are the same, by following the proof of Proposition \ref{Prop Resolvent}, the resolvent of $\mathscr{L}_{\alpha}$ can be constructed in the same way. It means that the solution $u$ of the resolvent equation $(\mathscr{L}_{\alpha}-z)u=f$ can be found in the form of \eqref{u+-}, in which
\begin{itemize}
\item $u_{\pm}$ is the restriction of $u$ on $\R_{\pm}$;
\item $\alpha_{\pm}$ and $\beta_{\pm}$ are functions defined in \eqref{alpha beta} where $A_{+}$ and $B_{-}$ are numbers defined in \eqref{A+} and \eqref{B-} by the decaying conditions of $u_{\pm}$ at $\pm \infty$ while $A_{-}$ and $B_{+}$ are found by the continuity of $u$ at zero ($u_{+}(0)=u_{-}(0)$) and the jump condition ($u_{+}'(0)-u_{-}'(0)=\alpha u(0)$).
\end{itemize}
Then, the resolvent of $\mathscr{L}_{\alpha}$ can be expressed in the integral form,
\begin{equation}
\left[\left(\mathscr{L}_{\alpha}-z\right)^{-1} f\right](x) = \int_{\R} \mathcal{R}_{z,\alpha}(x,y) f(y)\, \dd y ,
\end{equation}
where $\mathcal{R}_{z,\alpha}(x,y)$ is defined by
\begin{equation*}
\mathcal{R}_{z,\alpha}(x,y)\coloneqq \left\{
\begin{aligned}
			 & \frac{1}{2k_{+}(z)}e^{-k_{+}(z)|x-y|}+\frac{k_{+}(z)-k_{-}(z)-\alpha}{2k_{+}(z)\left(k_{+}(z)+k_{-}(z)+\alpha\right)}e^{-k_{+}(z)(x+y)} && \text{for } \{ x> 0, y> 0\}; \\
			 &  \frac{1}{2k_{-}(z)}e^{-k_{-}(z)|x-y|}-\frac{k_{+}(z)-k_{-}(z)+\alpha}{2k_{-}(z)\left(k_{+}(z)+k_{-}(z) +\alpha\right)}e^{k_{-}(z)(x+y)} && \text{for } \{x < 0, y< 0\};\\ 		 
			 & \frac{1}{k_{+}(z)+k_{-}(z)+\alpha}e^{-k_{+}(z)x+k_{-}(z)y} && \text{for } \{x> 0, y< 0\}; \\
			 & \frac{1}{k_{+}(z)+k_{-}(z)+\alpha}e^{k_{-}(z)x-k_{+}(z)y}  && \text{for } \{x< 0,y> 0\}.
		\end{aligned}
\right.
\end{equation*}
Next, we decompose the resolvent as in Section \ref{Section Resolvent Estimate}, that is
\[ (\mathscr{L}_{\alpha}-z)^{-1}= \mathcal{R}_{1,z,\alpha}+\mathcal{R}_{2,z,\alpha},\]
where
\begin{equation}
\mathcal{R}_{1,z,\alpha}f(x)\coloneqq\int_{\R} \mathcal{R}_{1,z,\alpha}(x,y) f(y)\, \dd y,\qquad \mathcal{R}_{2,z,\alpha}f(x)\coloneqq\int_{\R} \mathcal{R}_{2,z,\alpha}(x,y) f(y)\, \dd y,
\end{equation}
with
\begin{align}
\mathcal{R}_{1,z,\alpha}(x,y) &=  \left\{
\begin{aligned}
			 &K_{++}(z,\alpha)e^{-k_{+}(z)(x+y)}, \qquad && \text{for } \{x> 0, y> 0\}; \\
			 & K_{--}(z,\alpha)e^{k_{-}(z)(x+y)},\qquad && \text{for } \{x < 0, y< 0\};\\  
			 & K_{+-}(z,\alpha)e^{-k_{+}(z)x+k_{-}(z)y}, \qquad&& \text{for } \{x>0, y< 0\}; \\			 
			 & K_{+-}(z,\alpha)e^{k_{-}(z)x-k_{+}(z)y}, \qquad && \text{for } \{x< 0, y> 0\}; 
		\end{aligned}
\right.\\
\mathcal{R}_{2,z,\alpha}(x,y)&=\mathcal{R}_{2,z}(x,y).
\end{align}
in which
\begin{equation*}
\begin{aligned}
K_{++}(z,\alpha)&\coloneqq\frac{k_{+}(z)-k_{-}(z)-\alpha}{2k_{+}(z)\left(k_{+}(z)+k_{-}(z)+\alpha\right)},\qquad K_{--}(z,\alpha)\coloneqq-\frac{k_{+}(z)-k_{-}(z)+\alpha}{2k_{-}(z)\left(k_{+}(z)+k_{-}(z) +\alpha\right)},\\
K_{+-}(z,\alpha) &\coloneqq\frac{1}{k_{+}(z)+k_{-}(z)+\alpha}.
\end{aligned}
\end{equation*}
By applying the proof of Proposition \ref{Prop Norm R1}, the norm of $\mathcal{R}_{1,z,\alpha}$ is estimated as follows, for all $z\in \rho(\mathscr{L}_{\alpha})$,
\begin{equation}\label{Bound R1 delta}
\begin{aligned}
\Vert \mathcal{R}_{1,z, \alpha} \Vert &\leq \frac{1}{\sqrt{2}} \sqrt{\sqrt{(A(z,\alpha)-C(z,\alpha))^2+B(z,\alpha)^2}+A(z,\alpha)+C(z,\alpha)+2D(z,\alpha)},\\
\Vert \mathcal{R}_{1,z, \alpha} \Vert &\geq \frac{1}{\sqrt{2}} \sqrt{\sqrt{(A(z,\alpha)-C(z,\alpha))^2+\widetilde{B}(z,\alpha)^2}+A(z,\alpha)+C(z,\alpha)+2D(z,\alpha)},
\end{aligned}
\end{equation}
where the formulas of $A(z,\alpha)$, $B(z,\alpha)$, $C(z,\alpha)$, $D(z,\alpha)$, $\widetilde{B}(z,\alpha)$ are defined by replacing $K_{++}(z)$, $K_{--}(z)$  and $K_{+-}(z)$ in the formulas of $A(z)$, $B(z)$, $C(z)$, $D(z)$, $\widetilde{B}(z)$ by the above $K_{++}(z,\alpha)$, $K_{--}(z,\alpha)$ and $K_{+-}(z,\alpha)$. Let us recall and reuse here the conventions at the beginning of Section \ref{Section Resolvent Estimate}. Thanks to Lemma \ref{Lem Asymptotic}, it is easily to obtain the following asymptotic expansions
\begin{align*}
|k_{+}(z)-k_{-}(z)-\alpha|&=2\sqrt{\tau}\left(1+\mathcal{O}\left(\frac{1}{|\tau|^{1/2}}\right)\right),\\
|k_{+}(z)-k_{-}(z)+\alpha|&=2\sqrt{\tau}\left(1+\mathcal{O}\left(\frac{1}{|\tau|^{1/2}}\right)\right),\\
|k_{+}(z)+k_{-}(z)+\alpha|&=|\alpha|\left(1+\mathcal{O}\left(\frac{1}{|\tau|^{1/2}}\right)\right),\\
|K_{++}(z,\alpha)|&= \frac{1}{|\alpha|}\left(1+\mathcal{O}\left(\frac{1}{|\tau|^{1/2}}\right)\right),\\
|K_{--}(z,\alpha)|&= \frac{1}{|\alpha|}\left(1+\mathcal{O}\left(\frac{1}{|\tau|^{1/2}}\right)\right),\\
|K_{+-}(z,\alpha)|&= \frac{1}{|\alpha|}\left(1+\mathcal{O}\left(\frac{1}{|\tau|^{1/2}}\right)\right).
\end{align*}
From these expansions, the asymptotic formulas of $A(z,\alpha),C(z,\alpha),D(z,\alpha)$ and $B(z,\alpha),\widetilde{B}(z,\alpha)$ are deduced, that is
\begin{align*}
A(z,\alpha)&=\frac{\tau}{|\alpha|^2|\Im V_{+}-\delta|^2}\left(1+\mathcal{O}\left(\frac{1}{|\tau|^{1/2}}\right)\right),\\
C(z,\alpha)&=\frac{\tau}{|\alpha|^2|\Im V_{-}-\delta|^2}\left(1+\mathcal{O}\left(\frac{1}{|\tau|^{1/2}}\right)\right),\\
D(z,\alpha)&=\frac{\tau}{|\alpha|^2|\Im V_{+}-\delta||\Im V_{-}-\delta|}\left(1+\mathcal{O}\left(\frac{1}{|\tau|^{1/2}}\right)\right),\\
B(z,\alpha)&=\frac{2|\Im V_{+}-\Im V_{-}|\tau}{|\alpha|^2|\Im V_{+}-\delta|^{3/2}|\Im V_{-}-\delta|^{3/2}}\left(1+\mathcal{O}\left(\frac{1}{|\tau|^{1/2}}\right)\right),\\
\widetilde{B}(z,\alpha)&=\frac{2|\Im V_{+}-\Im V_{-}|\tau}{|\alpha|^2|\Im V_{+}-\delta|^{3/2}|\Im V_{-}-\delta|^{3/2}}\left(1+\mathcal{O}\left(\frac{1}{|\tau|^{1/2}}\right)\right).
\end{align*}
Then, by using \eqref{Relation Im V}, we obtain
\begin{align*}
\sqrt{\left(A(z,\alpha)-C(z,\alpha)\right)^2+B(z,\alpha)^2} = \frac{|\Im V_{+}-\Im V_{-}|^2\tau}{|\alpha|^2|\Im V_{+}-\delta|^{2}|\Im V_{-}-\delta|^{2}}\left(1+\mathcal{O}\left(\frac{1}{|\tau|^{1/2}}\right)\right),\\
A(z,\alpha)+C(z,\alpha)+2D(z,\alpha)= \frac{|\Im V_{+}-\Im V_{-}|^2\tau}{|\alpha|^2|\Im V_{+}-\delta|^{2}|\Im V_{-}-\delta|^{2}}\left(1+\mathcal{O}\left(\frac{1}{|\tau|^{1/2}}\right)\right).\\
\end{align*}
From \eqref{Bound R1 delta}, since $B(z,\alpha)$ and $\widetilde{B}(z,\alpha)$ have the same asymptotic behaviors, we deduce that
\[\Vert \mathcal{R}_{1,z,\alpha} \Vert =\frac{|\Im V_{+}-\Im V_{-}|\sqrt{\tau}}{|\alpha||\Im V_{+}-\delta||\Im V_{-}-\delta|}\left(1+\mathcal{O}\left(\frac{1}{|\tau|^{1/2}}\right)\right).\]
By employing the upper bound \eqref{Bound above norm R2} for the norm of $\mathcal{R}_{2,z}$, it implies that $\mathcal{R}_{2,z,\alpha}$ is indeed a small perturbation of $\mathcal{R}_{1,z,\alpha}$:
\[ \frac{\Vert \mathcal{R}_{2,z,\alpha} \Vert }{\Vert \mathcal{R}_{1,z,\alpha} \Vert}=\mathcal{O}\left(\frac{1}{|\tau|^{1/2}}\right).\]
Then, the triangle inequality provides us 
\[ \Vert (\mathscr{L}_{\alpha}-z)^{-1} \Vert= \Vert \mathcal{R}_{1,z,\alpha} \Vert\left(1+\mathcal{O}\left(\frac{1}{|\tau|^{1/2}}\right)\right),\]
and the conclusion of the theorem follows.
\subsection{Accurate pseudomode for $\mathscr{L}_{\alpha}$: Proof of Theorem \ref{Theo Pseudomode delta}}
Let us fix $\alpha\in \C\setminus \{0\}$ and $z\in \rho(\mathscr{L}_{\alpha})$. What we need to do is to follow the proof of Theorem \ref{Theo Pseudomode} in subsection \ref{Subsec Pseudomode} and modify some calculations of this proof. We started by choosing the pseudomode in the form 
\[ \Psi_{z,\alpha}(x)= \Psi_{1,z,\alpha}(x) + \Psi_{2,z,\alpha}(x), \]
where
\begin{equation*}
\begin{aligned}
\Psi_{1,z,\alpha}(x) &= n_{1}(z,\alpha) e^{k_{-}(z)x}\textbf{\textup{1}}_{\R_{-}}(x)+p_{1}(z,\alpha) e^{-k_{+}(z)x}\textbf{\textup{1}}_{\R_{+}}(x),\\
\Psi_{2,z,\alpha}(x) &= n_{2}(z,\alpha) e^{\overline{k_{-}(z)}x}\textbf{\textup{1}}_{\R_{-}}(x)+p_{2}(z,\alpha) e^{-\overline{k_{+}(z)}x}\textbf{\textup{1}}_{\R_{+}}(x),
\end{aligned}
\end{equation*}
in which $n_{1}(z,\alpha),n_{2}(z,\alpha),p_{1}(z,\alpha),p_{2}(z,\alpha)$ are complex numbers to be determined later. At the end of the proof, we will see that $n_{2}$ and $p_{2}$ can be chosen independently of $\alpha$, and thus, so is $\Psi_{2,z,\alpha}$. In order $\Psi_{z,\alpha}$ to belong to the domain $\textup{Dom}(\mathscr{L}_{\alpha})$,  it is necessary that $\Psi_{z,\alpha}(0^{+})=\Psi_{z,\alpha}(0^{-})$ and $\Psi_{z,\alpha}'(0^{+})-\Psi_{z,\alpha}'(0^{-})=\alpha \Psi_{z,\alpha}(0)$ which impose the following conditions on the coefficients of the pseudomode
\begin{equation*}
\left\{ \begin{aligned}
&n_{1}(z,\alpha)+n_{2}(z,\alpha)=p_{1}(z,\alpha)+p_{2}(z,\alpha),\\
&-k_{+}(z)p_{1}(z,\alpha)-\overline{k_{+}(z)}p_{2}(z,\alpha)-k_{-}(z)n_{1}(z,\alpha)-\overline{k_{-}(z)}n_{2}(z,\alpha)=\alpha(p_{1}(z,\alpha)+p_{2}(z,\alpha)).
\end{aligned}
\right.
\end{equation*}
Since $z\in \rho(\mathscr{L}_{\alpha})$, $k_{+}(z)+k_{-}(z)+\alpha\neq 0$ (see the argument around \eqref{Eig Equation}), then $n_{1}(z,\alpha)$ and $p_{1}(z,\alpha)$ can be calculated in terms of $n_{2}(z,\alpha)$ and $p_{2}(z,\alpha)$ as follows
\begin{equation}
\left\{ \begin{aligned}
&n_{1}(z,\alpha)=-\frac{k_{+}(z)+\overline{k_{- }(z)}+\alpha}{k_{+}(z)+k_{-}(z)+\alpha}n_{2}(z,\alpha)+\frac{k_{+}(z)-\overline{k_{+}(z)}}{k_{+}(z)+k_{-}(z)+\alpha}p_{2}(z,\alpha),\\
&p_{1}(z,\alpha)=\frac{k_{-}(z)-\overline{k_{-}(z)}}{k_{+}(z)+k_{-}(z)+\alpha}n_{2}(z,\alpha)-\frac{\overline{k_{+}(z)}+k_{-}(z)+\alpha}{k_{+}(z)+k_{-}(z)+\alpha}p_{2}(z,\alpha).
\end{aligned}
\right.
\end{equation}
Thanks to the condition $\Psi_{z,\alpha}(0^{+})=\Psi_{z,\alpha}(0^{-})$, we can easily show that $\Psi_{z,\alpha}\in H^1(\R)$ and then, $\Psi_{z,\alpha}' \in H^1(\R\setminus \{0\})$, see \eqref{Psi'}. After that, the jump condition $\Psi_{z,\alpha}'(0^{+})-\Psi_{z,\alpha}'(0^{-})=\alpha \Psi_{z,\alpha}(0)$ implies that $\Psi_{z}\in \textup{Dom}(\mathscr{L}_{\alpha})$.
Let us recall and reuse the convention in the beginning of Section \ref{Section Resolvent Estimate} and now we send $z \to \infty$ inside the strip bounded by two essential spectrum lines $[V_{+},+\infty)$ and $[V_{-},+\infty)$. By imposing the same assumptions \eqref{Ass 1} on coefficients $n_{2}$, $p_{2}$ and working as in \eqref{n1}, we obtain
\begin{equation}
\begin{aligned}
|n_{1}(z,\alpha)|^2 &= \frac{4\tau |n_{2}(z,\alpha)-p_{2}(z,\alpha)|^2}{|\alpha|^2}\left(1+\mathcal{O}\left(\frac{1}{|\tau|^{1/2}}\right)\right),\\ |p_{1}(z,\alpha)|^2 &= \frac{4\tau |n_{2}(z,\alpha)-p_{2}(z,\alpha)|^2}{|\alpha|^2}\left(1+\mathcal{O}\left(\frac{1}{|\tau|^{1/2}}\right)\right).
\end{aligned}
\end{equation}
The norms squared of $\Psi_{1,z,\alpha}$ and $\Psi_{2,z,\alpha}$ can be computed explicitly as same as in \eqref{Norm Psi1} and then, we can show that $\Psi_{2,z,\alpha}$ is just a small perturbation compared with $\Psi_{1,z,\alpha}$, more precisely, $\Vert \Psi_{2,z,\alpha} \Vert = \mathcal{O}\left(\frac{1}{|\tau|^{1/2}}\right) \Vert \Psi_{1,z,\alpha} \Vert$ and thus, by triangle inequality, we have
\[
\Vert \Psi_{z,\alpha} \Vert = \Vert \Psi_{1,z,\alpha} \Vert \left(1 +\mathcal{O}\left(\frac{1}{|\tau|^{1/2}}\right) \right).\]
Without difficulty, we obtain the asymptotic behavior for the quotient
\begin{align*}
\frac{\Vert (\mathscr{L}_{\alpha}-z) \Psi_{z,\alpha} \Vert^2}{\Vert \Psi_{z,\alpha}\Vert^2}=&\frac{|\alpha|^2|\Im V_{+}-\delta| | \Im V_{-}-\delta|}{\tau |\Im V_{+}-\Im V_{-}|} \\
& \times\frac{|\Im V_{-}-\delta| |n_{2}(z,\alpha)|^2+|\Im V_{+}-\delta| |p_{2}(z,\alpha)|^2}{|n_{2}(z,\alpha)-p_{2}(z,\alpha)|^2}\left(1 +\mathcal{O}\left(\frac{1}{|\tau|^{1/2}}\right) \right).
\end{align*}
As same as the argument at the end of the subsection \ref{Subsec Pseudomode}, $n_{2}$ and $p_{2}$ will be chosen to minimize the function 
$$F(n_{2},p_{2})=\frac{|\Im V_{-}-\delta| |n_{2}|^2+|\Im V_{+}-\delta| |p_{2}|^2}{|n_{2}-p_{2}|^2},$$
and satisfies the assumptions in \eqref{Ass 1}, and they are
\[ n_{2}(z,\alpha)=n_{2}(z)=-|\Im V_{+}-\delta|,\qquad p_{2}(z,\alpha)=p_{2}(z)=|\Im V_{-}-\delta|.\]
With these values of $n_{2}$ and $p_{2}$, the conclusion of Theorem \ref{Theo Pseudomode delta} follows.
\appendix
\section{Properties of the operator $\mathscr{L}$: Proof of Proposition \ref{Prop Property}}\label{Appendix 1}
Let us introduce a translated sesquilinear of $Q$, that is
\[ \widetilde{Q}(u,v)\coloneqq Q(u,v)+(1-\min \{\Re V_{+}, \Re V_{-}\})\langle u,v\rangle, \qquad \textup{Dom}(\widetilde{Q})\coloneqq H^1(\R).\]
The coercivity of the sesquilinear $\widetilde{Q}$ is given by, for all $u\in H^1(\R)$,
\begin{align*}
\left\vert \widetilde{Q}(u,u) \right\vert\geq \Re\, \widetilde{Q}(u,u) &= \Vert u' \Vert^2+ \Vert u \Vert^2+(\Re V_{+}-\min \{\Re V_{+}, \Re V_{-}\})\int_{0}^{+\infty}\vert u \vert^2\,\dd x \\
&\qquad +(\Re V_{-}-\min \{\Re V_{+}, \Re V_{-}\})\int_{-\infty}^{0}\vert u \vert^2\,\dd x\\
&\geq  \Vert u \Vert_{H^1}^2.
\end{align*}
It is easy to obtain the continuity of $\widetilde{Q}$, \emph{i.e.,} there exists a positive constant $C$ such that $\left\vert\widetilde{Q}(u,v)\right\vert\leq C \Vert u \Vert_{H^1}\Vert v \Vert_{H^1}$. Then, by Lax-Milgram theorem (cf. \cite[Theorem 2.89]{Cheverry-Raymond21}), $\widetilde{Q}$ is associated with a closed, densely defined and bijective operator $\widetilde{\mathscr{L}}$ whose domain is given by
\begin{align*}
&\textup{Dom}\left(\widetilde{\mathscr{L}}\right)=\left\{\begin{aligned}
u \in H^1(\R): &\text{ there exists } \widetilde{f}\in L^2(\R) \text{ such that }\\
&\widetilde{Q}(u,v)=\left\langle \widetilde{f},v \right\rangle \text{ for all } v\in H^1(\R)
\end{aligned} \right\},\\
&\left\langle \widetilde{\mathscr{L}}u,v \right\rangle =\widetilde{Q}(u,v), \qquad \forall u \in \textup{Dom}\left(\widetilde{\mathscr{L}}\right), \forall v\in H^1(\R).
\end{align*}
Then, the operator $\mathscr{L}$ defined by shifting the operator $\widetilde{\mathscr{L}}$, that is
\[ \mathscr{L}\coloneqq \widetilde{\mathscr{L}}- \left(1-\min \{\Re V_{+}, \Re V_{-}\}\right),\qquad \textup{Dom}(\mathscr{L})=\textup{Dom}(\widetilde{\mathscr{L}}).\]
\begin{enumerate}[label=\textbf{(\arabic*)}]
\item We will show that $\textup{Dom}\left(\mathscr{L}\right)=H^2(\R)$. Let $u\in \textup{Dom}\left(\mathscr{L}\right)$, by considering $v$ in \eqref{Laxmilgram eq} on the space of test functions $C_{c}^{\infty}(\R)$, then for each $u \in \textup{Dom}\left(\widetilde{\mathscr{L}}\right)$, we get in the distributional sense that
\[ \mathscr{L}u = -u''+V(x)u. \]
Since $\mathscr{L}u \in L^2(\R)$ and $V(x)u\in L^2(\R)$, we implies that $u''\in L^2(\R)$ and thus $u\in H^{2}(\R)$. Therefore, $\textup{Dom}\left(\mathscr{L}\right)\subset H^2(\R)$. The remaining direction $H^2(\R)\subset \textup{Dom}\left(\mathscr{L}\right)$ is easily obtained by integration by parts. Since $\widetilde{\mathscr{L}}$ is bijective, it implies that $-\left(1-\min \{\Re V_{+}, \Re V_{-}\}\right)\in \rho(\mathscr{L})$. In other words, $\rho(\mathscr{L})$ is nonempty.
\item Let us recall here the definition of the numerical range of an operator, it is a subset in $\C$ defined by
\[ \textup{Num}(\mathscr{L})=\left\{ \left\langle \mathscr{L} \psi, \psi \right\rangle \in \C:\, \psi \in \textup{Dom}(\mathscr{L}),\, \Vert \psi \Vert=1 \right\}.\]
Given $\psi \in \textup{Dom}(\mathscr{L})$  such that $\Vert \psi \Vert=1$, we have
\begin{align*}
\left\langle \mathscr{L} \psi, \psi \right\rangle = &\Vert\psi'\Vert^2+  V_{+}\int_{0}^{+\infty} \vert \psi \vert^2\, \dd x+V_{-}\int_{-\infty}^{0} \vert \psi \vert^2\, \dd x .
\end{align*}
Let $s=\int_{0}^{+\infty} |\psi|^2\, \dd x$, it is clear that $s\in [0,1]$ because of the normalization of $\psi$, then we can write
\begin{align*}
\left\langle \mathscr{L} \psi, \psi \right\rangle = \Vert\psi'\Vert^2+  sV_{+}+(1-s)V_{-}.
\end{align*}
This implies that $ \textup{Num}(\mathscr{L})\subset \{(0,+\infty)+sV_{+}+(1-s)V_{-}:\, s\in [0,1]\}$. In order to prove the remaining direction, we fix a function $f \in C_{c}^{\infty}(\R)$ such that $\textup{supp} \, f\subset \R_{+}$ and $\Vert f \Vert=\Vert f \Vert_{L^2(\R_{+})}=1$. By setting a family of functions in $C_{c}^{\infty}(\R)$, for $\lambda>0$,
\[ \psi_{\lambda}(x) := \lambda^{\frac{1}{2}}f(\lambda x),\]
it is obvious that $\Vert \psi_{\lambda} \Vert= \Vert f \Vert=1$ and $\Vert \psi_{\lambda}' \Vert= \lambda \Vert f' \Vert$. Therefore,
\[ \left\langle \mathscr{L} \psi_{\lambda}, \psi_{\lambda} \right\rangle=t+V_{+},\]
where $t=\lambda \Vert f' \Vert^2$ can take arbitrary positive value when $\lambda$ runs on $(0,\infty)$. Consequently, $[V_{+},+\infty)\subset  \textup{Num}(\mathscr{L})$. In the same manner by choosing a function $f$ whose support lies inside $\R_{-}$, we also obtain $[V_{-},+\infty)\subset  \textup{Num}(\mathscr{L})$. In other words, two lines $[V_{+},+\infty)$ and $[V_{-},+\infty)$ are contained in $\textup{Num}(\mathscr{L})$. Since $\textup{Num}(\mathscr{L})$ is a convex set, it also contains the convex hull (\emph{i.e.} the smallest convex set containing) of the set $[V_{+},+\infty)\cup [V_{-},+\infty)$, which is precisely the set $\{(0,+\infty)+sV_{+}+(1-s)V_{-}:\, s\in [0,1]\}$. Therefore, the numerical range is described exactly as in \eqref{Num Range}.

We use \cite[Proposition 3.19]{Schmudgen12} to show that $\mathscr{L}$ is a $m$-sectorial operator. By choosing a sector  $S_{c,\theta}$ whose vertex $c$ is some point in the middle of two points $\min\{\Re V_{+}, \Re V_{-}\}-1$ and $\min\{\Re V_{+}, \Re V_{-}\}$ on the real axis with a suitable semi-angle $\theta \in [0, \frac{\pi}{2})$ such that the numerical range $\textup{Num}(\mathscr{L})$ is included inside $S_{c,\theta}$, since the point $\min\{\Re V_{+}, \Re V_{-}\}-1 \in \rho(\mathscr{L})\setminus S_{c,\theta}$, $\mathscr{L}$ is $m$-sectorial.
\item The operator $\mathscr{L}$ can be seen as the sum of a self-adjoint operator $-\frac{\dd^2}{\dd x^2}$ (with domain $H^{2}(\R)$) and a bounded operator $V(x)$ (on $L^2(\R)$), then, thanks to \cite[Prop. 1.6(vii)]{Schmudgen12}, we have
\[ \mathscr{L}^{*} = \left(-\frac{\dd^2}{\dd x^2} + V(x)\right)^{*}= \left(-\frac{\dd^2}{\dd x^2} \right)^{*}+\left(V(x) \right)^{*}=-\frac{\dd^2}{\dd x^2} + \, \overline{V(x)}.\]
By directing computation, we get that, for all $f\in H^2(\R)$,
\begin{equation}\label{Normality Eq}
\begin{aligned}
\Vert \mathscr{L} f \Vert^2 &=  \Vert -f'' + (\Re V)f \Vert^2 + \Vert (\Im V)f \Vert^2 + 2 \Im \left\langle -f'', \Im V f\right\rangle,\\
\Vert \mathscr{L}^{*} f \Vert^2 &=  \Vert -f'' + (\Re V)f \Vert^2 + \Vert (\Im V)f \Vert^2 - 2 \Im \left\langle -f'', \Im V f\right\rangle.
\end{aligned}
\end{equation}
Since the operator $\mathscr{L}$ and its adjoint $\mathscr{L}^{*}$ share the same domain, $\mathscr{L}$ is normal if and only if $\Vert \mathscr{L} f \Vert=\Vert \mathscr{L}^{*} f \Vert$ for all $f\in H^2(\R)$, or
\[ \Im \left\langle -f'', \Im V f\right\rangle =0, \qquad \forall f \in H^2(\R).\]
By using integral by parts, this condition is equivalent to
\[ \left( \Im V_{+} - \Im V_{-}\right) \Im \left( f'(0) \overline{f(0)}\right)=0,\qquad \forall f \in H^2(\R).\]
Since there exists $f\in H^2(\R)$ such that $\Im \left( f'(0) \overline{f(0)}\right) \neq 0$, for example, $f(x)=\frac{1}{\sqrt{x^2+1}}+i\frac{x}{x^2+1}$, thus the normality of $\mathscr{L}$ is equivalent to $\Im V_{+}=\Im V_{-}$.
From the formula of $\mathscr{L}$ and $\mathscr{L}^{*}$, it is obviously that $\mathscr{L}$ is self-adjoint iff $V=\overline{V}$, \emph{i.e.}, $\Im V=0$. Next, let us check when $\mathscr{L}$ is $\mathcal{P}$-self-adjoint, $\mathcal{T}$-self-adjoint and $\mathcal{PT}$-symmetric by straightforward  computation, given $u\in H^2(\R)$ and $x\in \R$,
\begin{equation}\label{PJ self adjoint}
\begin{aligned}
\mathcal{P}\mathscr{L} \mathcal{P}u(x) &= -u''(x)+V(-x) u(x),\\
\mathcal{T}\mathscr{L} \mathcal{T}u(x) &= -u''(x)+\overline{V}(x) u(x),\\
\mathscr{L} \mathcal{PT}u(x) &=-\overline{u}''(-x)+V(x)\overline{u}(-x),\\
\mathcal{PT}\mathscr{L}u(x) &=-\overline{u}''(-x)+\overline{V}(-x)\overline{u}(-x).
\end{aligned}
\end{equation}
Thus, all the conclusions on $\mathcal{P}$-self-adjointness, $\mathcal{T}$-self-adjointness and $\mathcal{PT}$-symmetry of $\mathscr{L}$ follows obviously.
\end{enumerate}
\section{Properties of the operator $\mathscr{L}_{\alpha}$: Proof of Proposition \ref{Prop Property delta}}\label{Appendix L a}
First recall that for $u\in H^1(\R)$ then $|u|^2 \in H^1(\R)$ and $\left(|u|^2\right)'=2 \Re(u' \overline{u})$ (cf. \cite[Corollary 8.10]{Brezis11}). Thus, for any $\varepsilon>0$ and for any $u\in H^1(\R)$, we have
\begin{equation}\label{Inequality}
\begin{aligned}
|u(0)|^2 =\int_{-\infty}^{0} \left(|u(x)|^2\right)' \dd x =\int_{-\infty}^{0} 2 \Re(u' \overline{u})\, \dd x \leq &2 \Vert u\Vert_{L^2(\R_{-})} \Vert u' \Vert_{L^2(\R_{-})}\\
\leq &2 \Vert u\Vert \Vert u' \Vert \leq \varepsilon \Vert u' \Vert^2+ \frac{1}{\varepsilon} \Vert u \Vert^2.
\end{aligned}
\end{equation}
In the same manner as defining the operator $\mathscr{L}$ as in Proposition \ref{Prop Property}, we define the operator $\mathscr{L}_{\alpha}$ via Lax-Milgram theorem applying for a translated operator. By choosing and fixing some $\varepsilon>0$ such that $\varepsilon|\Re \alpha|<1$, we introduce a translated sesquilinear 
\begin{align*}
\widetilde{Q_{\alpha}}(u,v)\coloneqq Q_{\alpha}(u,v)+C_{\Re \alpha, \Re V}\langle u,v\rangle,\qquad\textup{Dom}(\widetilde{Q_{\alpha}})\coloneqq H^1(\R),
\end{align*}
where $C_{\Re \alpha, \Re V}$ is a constant depending on $\Re \alpha$, $\Re V_{+}$, $\Re V_{-}$ defined by
\[ C_{\Re \alpha, \Re V}\coloneqq 1-\varepsilon|\Re \alpha|-\min \{\Re V_{+}, \Re V_{-}\}+\frac{|\Re \alpha|}{\varepsilon}.\]
By using the inequality \eqref{Inequality}, we can show that, for all $u\in H^1(\R)$,
\begin{align*}
\left\vert \widetilde{Q_{\alpha}}(u,u) \right\vert\geq &\Re \widetilde{Q_{\alpha}}(u,u)\\
 =& \Vert u' \Vert^2 + \left(1-\varepsilon|\Re \alpha|+\frac{|\Re \alpha|}{\varepsilon}\right)\Vert u \Vert^2+(\Re V_{+}-\min \{\Re V_{+}, \Re V_{-}\})\int_{0}^{+\infty}\vert u \vert^2\,\dd x \\
&+(\Re V_{-}-\min \{\Re V_{+}, \Re V_{-}\})\int_{-\infty}^{0}\vert u \vert^2\,\dd x +\Re(\alpha)|u(0)|^2\\
\geq &(1-\varepsilon|\Re \alpha|) \Vert u \Vert_{H^1}^2.
\end{align*}
Thus, $ \widetilde{Q_{\alpha}}$ is coercive. Using \eqref{Inequality} again,  the continuity of $\widetilde{Q_{\alpha}}$ is easily obtained.  Then, thanks to Lax-Milgram theorem, there exists a closed densely defined and bijective operator $\widetilde{\mathscr{L}}_{\alpha}$ that is defined by
\begin{align*}
&\textup{Dom}\left(\widetilde{\mathscr{L}_{\alpha}}\right)=\left\{\begin{aligned}
u \in H^1(\R): &\text{ there exists } \widetilde{f}\in L^2(\R) \text{ such that }\\
&\widetilde{Q_{\alpha}}(u,v)=\left\langle \widetilde{f},v \right\rangle \text{ for all } v\in H^1(\R)
\end{aligned} \right\},\\
&\left\langle \widetilde{\mathscr{L}_{\alpha}}u,v \right\rangle =\widetilde{Q_{\alpha}}(u,v), \qquad \forall u \in \textup{Dom}\left(\widetilde{\mathscr{L}_{\alpha}}\right), \forall v\in H^1(\R).
\end{align*}
Then, we define the operator $\mathscr{L}_{\alpha}$ as usual
\[ \mathscr{L}_{\alpha}\coloneqq \widetilde{\mathscr{L}_{\alpha}}- C_{\Re \alpha, \Re V},\qquad \textup{Dom}(\mathscr{L}_{\alpha})\coloneqq\textup{Dom}(\widetilde{\mathscr{L}_{\alpha}}).\]
\begin{enumerate}[label=\textbf{\textup{(\arabic*)}}]
\item Let us analyze to clear the domain $\textup{Dom}\left(\widetilde{\mathscr{L}_{\alpha}}\right)$ up. Take $u\in  \textup{Dom}(\widetilde{\mathscr{L}_{\alpha}})$, there exists $f\in L^2(\R)$ such that, for all $v\in H^1(\R)$,
\begin{equation}\label{Lax-Mil Eq}
\int_{\R}u'(x) \overline{v'(x)}\, \dd x+\int_{\R} V(x)u(x) \overline{v(x)}\, \dd x+\alpha u(0)\overline{v(0)} =\int_{\R} f(x) \overline{v(x)} \, \dd x.
\end{equation}
By restricting our consideration for $v\in \C_{c}^{1}(\R\setminus \{0\})$, we get
\[ \int_{\R}u'(x) \overline{v'(x)}\, \dd x =-\int_{\R}\left( -f(x)+V(x)u(x)\right)\overline{v(x)}\, \dd x.\]
Thus, $u'\in H^1(\R\setminus \{0\})$ and hence $u\in H^2(\R\setminus \{0\})$. From the Sobolev embedding, $u'(0^{+})$ and $u'(0^{-})$ are well-defined. Now, considering $v\in H^{1}(\R)$ and using integral by parts for two functions $u'$ and $v$ on $\R_{+}$ and $\R_{-}$, we have
\begin{equation}\label{Integration by parts}
\int_{\R} u'(x) \overline{v'(x)} \dd x = \left(u'(0^{-})-u'(0^{+}) \right)\overline{v(0)}-\int_{\R}u''(x) \overline{v(x)} \, \dd x. 
\end{equation}
Replacing this into \eqref{Lax-Mil Eq}, we deduce that
\begin{align*}
\left\vert u'(0^{-})-u'(0^{+})+\alpha u(0) \right\vert \vert v(0) \vert \leq \left( \Vert u'' \Vert +\Vert V \Vert_{L^{\infty}} \Vert u \Vert+ \Vert f \Vert \right) \Vert v \Vert.
\end{align*}
It follows that $u'(0^{-})-u'(0^{+})+\alpha u(0)=0$, since there exists a sequence $v_{n}$ in $H^{1}(\R)$ such that $v_{n}(0)=1$ and $\Vert v_{n} \Vert \xrightarrow{n\to +\infty} 0$ (for example, take a function $\varphi \in C_{c}^{1}(\R)$ satisfying $\varphi(0)=1$ and consider $v_{n}(x)=\varphi(n x)$ for all $n\in \N$). Hence, we have proved that
\[ u\in \textup{Dom}(\mathscr{L}_{\alpha}) \Rightarrow u\in H^1(\R)\cap H^2(\R \setminus \{0\}) \text{ and }u'(0^{+})-u'(0^{-})=\alpha u(0).\]
Conversely, if $u \in H^1(\R)\cap H^2(\R\setminus\{0\})$, by using \eqref{Integration by parts} with the jump condition $u'(0^{+})-u'(0^{-})=\alpha u(0)$, we have
\[ Q_{\alpha}(u,v)= \langle -u"+V(x)u, v \rangle, \qquad \forall v\in H^1(\R).\]
and thus, $u$ belongs to the domain of $\mathscr{L}_{\alpha}$. Furthermore, this also implies the action of $\mathscr{L}_{\alpha}$ on its domain. Clearly, $-C_{\Re \alpha, \Re V}\in  \rho(\mathscr{L})$.
\item Let $u\in \textup{Dom}(\mathscr{L}_{\alpha})$ such that $\Vert u \Vert=1$, from \eqref{Laxmilgram eq Interaction}, we have
\[ \langle \mathscr{L}_{\alpha} u, u \rangle = \Vert u' \Vert^2+ V_{+} \Vert u' \Vert_{L^2(\R_{+})}^2+ V_{-} \Vert u' \Vert_{L^2(\R_{+})}^2+ \alpha |u(0)|^2.\]
Let $\varepsilon>0$ and using inequality \eqref{Inequality}, we get
\begin{align*}
\vert\Im  \langle \mathscr{L}_{\alpha} u, u \rangle \vert\leq  \varepsilon|\Im \alpha| \Vert u' \Vert^2+\frac{|\Im \alpha| }{\varepsilon} +\max\{|\Im V_{+}|,|\Im V_{-}| \}, 
\end{align*}
and
\begin{align*}
\Re \langle \mathscr{L}_{\alpha} u, u \rangle\geq (1-\varepsilon |\Re \alpha|) \Vert u' \Vert^2 + \min \{\Re V_{+}, \Re V_{-}\} - \frac{|\Re \alpha|}{\varepsilon}.
\end{align*}
By choosing $\varepsilon=\frac{1}{|\Re \alpha|+|\Im \alpha|}$, we deduce that
\[ \Re \langle \mathscr{L}_{\alpha} u, u \rangle -\vert\Im  \langle \mathscr{L}_{\alpha} u, u \rangle \vert\geq C_{V,\alpha},\]
where $C_{V,\alpha}$ defined in \eqref{Vertex}. In other words, we have shown that
$\langle\mathscr{L}_{\alpha} u, u \rangle \in \{\lambda \in \C: |\Im \lambda| \leq \Re \lambda -C_{V,\alpha}\}=S_{C_{V,\alpha},\frac{\pi}{4}}$. The arbitrariness of $u\in \textup{Dom}(\mathscr{L}_{\alpha})$ gives us the inclusion $N(\mathscr{L}_{\alpha})\subset S_{C_{V,\alpha},\frac{\pi}{4}}$. From this, it is also clear that $\mathscr{L}_{\alpha}$ is $m$-sectorial.

\item Let us consider a conjugate transpose sequilinear of $Q_{\alpha}(u,v)$, that is
\[ \widehat{Q_{\alpha}}(u,v)\coloneqq\overline{Q_{\alpha}(v,u)},\qquad \textup{Dom}(\widehat{Q_{\alpha}})=H^1(\R),\]
More precisely, for all $u,v\in H^1(\R)$,
\[ \widehat{Q_{\alpha}}(u,v)=\int_{\R} u'(x) \overline{v'(x)}\, \dd x + \overline{V_{+}}\int_{0}^{+\infty} u(x) \overline{v}(x)\, \dd x +\overline{V_{-}}\int_{-\infty}^{0} u(x) \overline{v}(x)\, \dd x+\overline{\alpha} u(0)\overline{v(0)}.\]
By working as above (translating the sesquilinear by the real constant $C_{\Re \alpha, \Re V}$ to obtain a coercive and continuous one, using Lax-Milgram to define a corresponding operator, and translating back this operator by the constant $C_{\Re \alpha, \Re V}$), we also can define an operator $\widehat{\mathscr{L}_{\alpha}}$ by
\begin{equation*}
\begin{aligned}
&\textup{Dom}(\widehat{\mathscr{L}_{\alpha}})=\{u\in H^1(\R)\cap H^2(\R\setminus\{0\}): u'(0^{+})-u'(0^{-})=\overline{\alpha} u(0) \},\\
&\widehat{\mathscr{L}_{\alpha}}u=-u''+\overline{V(x)}u, \qquad \forall u \in \textup{Dom}(\widehat{\mathscr{L}_{\alpha}}).
\end{aligned}
\end{equation*}
Thanks to \cite[Theorem 2.90]{Cheverry-Raymond21}, we have $\mathscr{L}_{\alpha}^{*}=\widehat{\mathscr{L}_{\alpha}}$. After obtain the formula of the adjoint, we will discuss about the normality of $\mathscr{L}_{\alpha}$. By comparing the domain of $\mathscr{L}_{\alpha}$ and $\mathscr{L}_{\alpha}^{*}$, we first notice that $\textup{Dom}(\mathscr{L}_{\alpha})=\textup{Dom}(\mathscr{L}_{\alpha}^{*})$ if and only if $\alpha\in \R$. Indeed, if $\alpha\in \R$, it is clear that two domains are the same and if $\alpha \in \C\setminus \R$, we consider, for example, $u(x)=e^{\frac{\alpha}{2}|x|} \varphi(x)$ with $\varphi$ is a smooth cut-off function equal $1$ in the neighborhood of zero, then $u\in \textup{Dom}(\mathscr{L}_{\alpha})$ but $u\notin \textup{Dom}(\mathscr{L}_{\alpha}^{*})$. Therefore, working as in \eqref{Normality Eq} and integrating by parts, we have a statement that $\mathscr{L}_{\alpha}$ is normal if and only if $\alpha \in \R$ and
\begin{equation}\label{Normality Interaction}
\left( \Im V_{+} f'(0^{+}) - \Im V_{-} f'(0^{-}) \right)\overline{f(0)}=0, \qquad \forall f\in \textup{Dom}(\mathscr{L}_{\alpha}).
\end{equation}
We consider the following cases:
\begin{enumerate}
\item[\textbf{Case 1}:] $\Im V_{+} \neq \Im V_{-}$. We consider, for example, the function 
$$f(x) = \left(e^{\frac{(1-\Im V_{+} \alpha) x}{\Im V_{+}-\Im V_{-}}}\textbf{\textup{1}}_{\R_{-}}(x)+e^{\frac{(1-\Im V_{-} \alpha) x}{\Im V_{+}-\Im V_{-}}}\textbf{\textup{1}}_{\R_{+}}(x)\right) \varphi(x), $$
where $\varphi(x)$ is a smooth cut-off function mentioned above, then,  
\[ f(0)=1,\qquad f'(0^{+}) = \frac{1-\Im V_{-} \alpha}{\Im V_{+}-\Im V_{-}}, \qquad f'(0^{-}) = \frac{1-\Im V_{+} \alpha}{\Im V_{+}-\Im V_{-}}.\]
It is clear that $f\in \textup{Dom}(\mathscr{L}_{\alpha})$ and $\left( \Im V_{+} f'(0^{+}) - \Im V_{-} f'(0^{-}) \right)\overline{f(0)}=1$. Therefore, $\mathscr{L}_{\alpha}$ is nonnormal in this case.
\item[\textbf{Case 2}:] $\Im V_{+} = \Im V_{-}=0$. Since \eqref{Normality Interaction} is automatically satisfied for all $f\in \textup{Dom}(\mathscr{L}_{\alpha})$, $\mathscr{L}_{\alpha}$ is normal if and only if $\alpha\in \R$.
\item[\textbf{Case 3}:] $\Im V_{+} = \Im V_{-}\neq 0$. Condition \eqref{Normality Interaction} is equivalent to
\[ \alpha |f(0)|^2=0 \qquad \forall f\in \textup{Dom}(\mathscr{L}_{\alpha}).\]
It is easy to check, using the function $e^{\frac{\alpha}{2}|x|} \varphi(x)$ mentioned above, that $\mathscr{L}_{\alpha}$ is normal if and only if $\alpha=0$.
\end{enumerate}
From the above cases, the necessary and sufficient conditions for the normality of $\mathscr{L}_{\alpha}$ is shown as in the statement of this proposition. The condition for self-adjointness of $\mathscr{L}_{\alpha}$ follows easily from comparing the domains and the actions of $\mathscr{L}_{\alpha}$ and $\mathscr{L}_{\alpha}^{*}$. In order to check when $\mathscr{L}_{\alpha}$ is $\mathcal{P}$-self-adjoint, $\mathcal{T}$-self-adjoint and $\mathcal{PT}$-symmetric, we notice that since actions of $\mathscr{L}_{\alpha}$ and $\mathscr{L}$ (similarly, $\mathscr{L}_{\alpha}^{*}$ and $\mathscr{L}^{*}$) are the same when they act on their domains, thus we can use \eqref{PJ self adjoint} to get the conditions on $V$ as in Proposition \ref{Prop Property}. However, unlike the operator $\mathscr{L}$, the domains of $\mathscr{L}_{\alpha}$ and its adjoint $\mathscr{L}_{\alpha}^{*}$ are not the same, it leads to the fact that $\textup{Dom}(\mathscr{L}_{\alpha}^{*})$ and $\textup{Dom}(\mathcal{P}\mathscr{L}_{\alpha}\mathcal{P})$ (or $\textup{Dom}(\mathcal{T}\mathscr{L}_{\alpha}\mathcal{T})$) are not necessary the same. Thus, we need to check them carefully to get condition on $\alpha$.
\begin{align*}
\textup{Dom}(\mathcal{P}\mathscr{L}_{\alpha}\mathcal{P})=&\{u\in L^2(\R): u(-x)\in \textup{Dom}(\mathscr{L}_{\alpha}) \}\\
=&\{u\in H^1(\R)\cap H^2(\R\setminus\{0\}): u'(0^{+})-u'(0^{-})=-\alpha u(0) \},
\end{align*}
and
\begin{align*}
\textup{Dom}(\mathcal{T}\mathscr{L}_{\alpha}\mathcal{T})=&\{u\in L^2(\R): \overline{u(x)}\in \textup{Dom}(\mathscr{L}_{\alpha}) \}\\
=&\{u\in H^1(\R)\cap H^2(\R\setminus\{0\}): u'(0^{+})-u'(0^{-})=\overline{\alpha} u(0) \}.
\end{align*}
Therefore, we have $\textup{Dom}(\mathcal{P}\mathscr{L}_{\alpha}\mathcal{P})=\textup{Dom}(\mathscr{L}_{\alpha}^{*})$ if and only if $\overline{\alpha}=-\alpha$, in other words, $\Re \alpha=0$, while $\textup{Dom}(\mathcal{T}\mathscr{L}_{\alpha}\mathcal{T})=\textup{Dom}(\mathscr{L}_{\alpha}^{*})$ is always true for any $\alpha\in \C$. In the same manner, we can check that $\textup{Dom}(\mathcal{PT}\mathscr{L}_{\alpha}) \subset \textup{Dom}(\mathscr{L}_{\alpha}\mathcal{PT})$ is ensured if and only if $\Re \alpha=0$.
\end{enumerate}
\bibliographystyle{abbrv}
\bibliography{Ref}
\end{document}